\documentclass[10pt,a4paper,final]{article}
\usepackage[latin1]{inputenc}
\usepackage{amsmath}
\usepackage{amsfonts}
\usepackage{amssymb}
\usepackage{makeidx}
\usepackage{graphicx}

\usepackage{todo}
\usepackage{amssymb}
\usepackage{amsmath}
\usepackage{latexsym}
\usepackage{amssymb}
\usepackage{amsthm}
\usepackage{mathrsfs}
\usepackage{bbm}
\usepackage{bbold}
\usepackage{subfigure}
\usepackage{float}
\usepackage{epsfig}
\usepackage[numbers]{natbib}
\usepackage{mathtools}
\usepackage{listings}
\usepackage{pgf}
\usepackage{listings}
\usepackage{epstopdf}

\theoremstyle{plain} 
\newtheorem{thm}{Theorem}[section] 
\newtheorem{defn}{Definition}[section] 
\newtheorem{prop}{Proposition}[section]
\newtheorem{cor}{Corollary}[section]
\newtheorem{lem}{Lemma}[section]

\theoremstyle{remark} 
\newtheorem{obs}{Remark}

\newcommand{\C}{{\mathbb{C}}}   
\newcommand{\E}{{\mathbb{E}}}   
\newcommand{\Prob}{{\mathbb{P}}}   

\newcommand{\AAA}{{\mathcal{A}}} 
\newcommand{\CC}{{\mathcal{C}}}

\title{Upper bounds for the largest component in critical inhomogeneous random graphs}

\author{Umberto De Ambroggio\thanks{University of Bath, Department of Mathematical Sciences - umbidea@gmail.com} \and Angelica Pachon\thanks{University of South Wales, Faculty of Computing, Engineering and Science - angelica.pachon@southwales.ac.uk}}

\begin{document}
	
	\maketitle

	\begin{abstract}
		We consider the Norros-Reittu random graph $NR_n(\textbf{w})$, where edges are present independently but edge probabilities are moderated by vertex weights, and use probabilistic arguments based on martingales to study component sizes in this model when considered at criticality. In particular, we obtain stronger bounds (with respect to those available in the literature) for the probability of observing an unusually large maximal cluster and simplify the arguments needed to derive (polynomial) bounds for the probability of observing an unusually small largest component.
	\end{abstract}

	\section{Introduction}
	During the last few decades, much attention in the field of random graphs has been devoted to create models capable of capturing the complexity of real-world networks.
	In \cite{New} it has been observed that many real-world networks are \textit{inhomogeneous}, in the sense that they may contain distinct groups of vertices behaving differently from a probabilistic point of view. 
	
	Inhomogeneous random graphs are random graph models in which edges are present independently and the probability of presence of a given edge depends on the vertices incident to it. Such random graphs were studied extensively in the seminal paper by Bollob\'as, Janson and Riordan \cite{BJR}. In this paper (see Theorem 3.1 in \cite{BJR}) the size of the largest components was analysed in the sub- and super-critical regimes. 
	The class of models studied in \cite{BJR} is very general and includes previous inhomogeneous random graphs like the one introduced in \cite{BJR2}. 
	
	Other models of inhomogeneous random graphs were considered in  \cite{CL1, CL2,CL3},   \cite{NRn}, and in  \cite{BDML}. These models are called \textit{rank-1 inhomogeneous random graphs} in \cite{BJR}; see section 16.4 in \cite{BJR} for a discussion about how these models are related to the general inhomogeneous random graph studied there.
	
	\cite{Hofstadt} considered the Norros-Reittu model, in which vertices are endowed with weights and each edge is present between a pair of distinct vertices (independently and) with a probability that is approximately proportional to the product of the weights of the vertices in the edge, and analysed component sizes in this model when considered at criticality. In particular, in \cite{Hofstadt} it has been shown how the size of the largest components depends sensitively on the asymptotic degree sequences of these graphs, i.e the sequences formed by the limiting proportions of vertices with degree $k$, for $k\geq 1$.
	
	During the last years inhomogeneous random graphs were further investigated by Kang, Koch and Pachon in \cite{KKP}, by  \cite{Pen}, and by Kang, Pachon and Rodriguez in \cite{KP}. 
	
	In the former work, the authors studied the near-critical behaviour of the so-called $2$-type binomial random graph. In this model, each one of the $n$ vertices is either of type $1$ or $2$, so $n=n_1+n_2$ where $n_i$ is the number of vertices of type $i=1,2$. An edge between a pair of vertices of types $i$ and $j$ is present with probability $p_{i,j}$ ($i,j=1,2$), independently of all other pairs. In the weakly supercritical regime, i.e. when the distance to the critical point characterising the phase transition is given by an $\epsilon = \epsilon_{n_1,n_2} \rightarrow 0$ ($n_1\geq n_2\rightarrow \infty$), the behaviour of the random graph depends very sensitively on the model parameters and, as a consequence, it could not be analysed using the parametrization in \cite{BJR}. However, the authors managed to show in \cite{KKP}, that with probability tending to one, the size of the largest component in this regime contains asymptotically $(2+o(1))\epsilon n$ vertices and all other components are of size  $o(\epsilon n)$, whenever $\epsilon^3 n_2 (1 \wedge \epsilon^{-1}p_{2,1}n_1)\rightarrow \infty$.
	
	Concerning the work in  \cite{Pen}, the author considered a graph on randomly scattered points in an arbitrary space in which any two elements $v$ and $u$ in this space are connected with a probability depending on the points $v$ and $u$, and studied the number of vertices of fixed degree, the number of components of fixed order, and the number of edges.
	
	Concerning \cite{KP}, instead, the authors considered an inhomogeneous random graph obtained in a constructive way from the Erd\H{o}s-R\'enyi random graph. Specifically, in their model the $n$ vertices of the Erd\H{o}s-R\'enyi graph are grouped into $N$ subsets of $\{1,\dots,n\}$, called \textit{super-vertices}, and then they defined a random graph on the $N$ super-vertices by letting any two of them being connected if, and only if, there is at least one edge between them in the Erd\H{o}s-R\'enyi graph. For this model, they studied the degree distribution, the threshold for connectedness, and further identified the phase transition for the emergence of the giant component.
	
	\textcolor{black}{In this paper we consider the critical Norros-Reittu random graph as described in \cite{Hofstadt} and adapt the martingale method introduced by \cite{NP,NCHperc} (used by the authors to to study component sizes in the near-critical Erd\H{o}s-R\'enyi random graph and in the random graph obtained through near-critical percolation on a (simple) random $d$-regular graph) to analyse the critical behaviour of the Norros-Reittu model (see also \cite{PER}, \cite{DEA} and \cite{DEAROB,DEAROB2} for recent results in this direction).}
	
	In particular, we show that the martingale method of Nachmias and Peres yields better upper bounds for the probability of observing unusually large maximal components with respect to those established in \cite{Hofstadt} \textcolor{black}{(provided we strengthen a condition related to the distribution function that specifies the vertex weights, as we explain later).}
	
	We also derive upper bounds for the probability of observing unusually small maximal components. Even though the latter bounds are not stronger than those in \cite{Hofstadt}, our proofs only rely on probabilistic arguments and avoid some of the involved analytical calculations used in \cite{Hofstadt}.
	
	A similar approach to the one introduced by \cite{NP,NCHperc} was used in \cite{HATMOL} to analyse the critical behaviour of a random graph with a given degree sequence. 
	
	However, to the best of our knowledge, the martingale argument of Nachmias and Peres has never been used to analyse inhomogeneous random graphs. \\
	

	\textbf{\textit{Structure of the paper.}} We start by formally introducing the model in Section 2, and proceed by stating our main results in Section 3. Subsequently, in Section 4, we describe a connection between clusters exploration in the random graph model considered in this paper and a particular class of branching processes, and we conclude with Section 5 where we prove our results.\\

	\textbf{\textit{Notation.}} \textcolor{black}{We denote by $\mathbb{N}_0$ the set of non-negative integers and set $[n]\coloneqq \{1,\dots,n\}$ for $n\in \mathbb{N}$.} Given two sequences of real numbers $(x_n)_n$ and $(y_n)_n$ we write $x_n=O(y_n)$ provided that, for all large enough $n$, we have $x_n\leq Cy_n$ for some finite constant $C>0$. We either write $x_n=o(y_n)$ or $x_n\ll y_n$ if $x_n/y_n\rightarrow 0$ as $n\rightarrow \infty$, and we write $x_n\asymp y_n$ if $x_n=O(y_n)$ and $y_n=O(x_n)$. Given any two real numbers $a$ and $b$, we set $a\wedge b\coloneqq \min\{a,b\}$ and $a\vee b\coloneqq \max\{a,b\}$. If $G=(V(G),E(G))$ is a (simple, undirected) graph, we write $v\leftrightarrow u$ if there exists a path of occupied edges connecting vertices $v$ and $u$; we adopt the convention that $v\leftrightarrow v$ for every vertex $v$. Moreover, we denote by $\CC(v)\coloneqq \{u\in V(G) :u\leftrightarrow v\}$ the connected component (or simply component, cluster) of vertex $v$. We denote the size of $\CC(v)$ by $|\CC(v)|$, and define a largest component $\CC_{max}$ to be any cluster $\CC(v)$ for which $|\CC(v)|$ is maximal; hence $|\CC_{max}|=\text{max}_{v\in V(G)}|\CC(v)|$. The letters $C,C',C'',c,c',c''$ etc. are reserved for constants appearing throughout the proofs, and each one of them could be used many times in a single proof even though its actual value may change from line to line.

	\section{The model of Norros and Reittu}
	The random graph model that we investigate has vertex set $[n]$ and vertices are endowed with \textit{weights}, which are used to model the tendencies of vertices to establish connections with other nodes. Specifically, let $\mathbf{w}=(w_i)_{i\in [n]}$ be a sequence of positive real numbers, which we call the sequence of vertex weights. Define $l_n\coloneqq \sum_{i\in[n]}w_i$, the sum of all weights. 
	
	The Norros-Reittu random graph, denoted by $NR_n(\mathbf{w})=([n],E(\textbf{w}))$ and introduced in \cite{NRn}, is an inhomogeneous random graph where, for $1\leq i<j\leq n$, the probability that the edge $ij$ is present is given by
	\begin{equation}\label{pij}
	\Prob(ij\in E(\mathbf{w}))=  1-e^{-w_iw_j/l_n},
	\end{equation}
	and edges are present independently. As explained in \cite{JAN2} and further remarked in section 1.3.5 in \cite{Hofstadt}, the $NR_n(\mathbf{w})$ random graph is closely related to the model studied in \cite{CL1,CL2,CL3}, so that the results proved for the $NR_n(\mathbf{w})$ random graph apply as well to these other models.
	
	It is intuitively clear that the topology of the graph is highly dependent upon the choice of the sequence $\mathbf{w}$, which we now specify.
	
	Let $F:\mathbb{R}\mapsto [0,1]$ be a distribution function. We construct the weights as in \cite{Hofstadt}, namely we set
	\begin{equation}\label{wj}
	w_j\coloneqq [1-F]^{-1}(j/n),\hspace{0.2cm}j\in [n],
	\end{equation}
	where $[1-F]^{-1}$ is the \textit{generalized inverse} of $1-F$, defined by $[1-F]^{-1}(1)\coloneqq 0$ and
	\begin{equation}\label{F}
	\hspace{0.3cm}[1-F]^{-1}(u)\coloneqq \inf\{s:1-F(s)\leq u\}, \quad u\in(0,1).
	\end{equation}
	Notice that $w_{i}\geq w_{i+1}$ for all $1\leq i\leq n-1$. Indeed, if $1-F(s)\leq i/n$ then clearly $1-F(s)\leq (i+1)/n$ too, and therefore $\{s:1-F(s)\leq i/n \} \subset \{s:1-F(s)\leq (i+1)/n \}$, so that taking the infimum in both sets we obtain $w_{i}\geq w_{i+1}$.
	
	In \cite{BJR} it has been shown that in the random graph $NR_n(\mathbf{w})$ with vertex weights as in \eqref{wj}, the number of vertices having degree $k$, denoted by $N_k$, satisfies (as $n\rightarrow \infty$)
	\[\frac{N_k}{n}\overset{\Prob}{\longrightarrow}p_k\coloneqq \E\Big(e^{-W}\frac{W^k}{k!}\Big)\hspace{0.15cm}k\geq 0,\]
	where $W$ is a $(0,\infty)$-valued random variable with distribution function $F$ \textcolor{black}{(and $\overset{\mathbb{P}}{\rightarrow}$ stands for convergence in probability).} We recognize the limiting sequence $(p_k)_{k\geq 0}$ as a so-called \textit{mixed Poisson distribution with mixing distribution $F$}. (Given a random variable $Z$ with distribution function $F_Z$, we say that $X$ follows a mixed Poisson distribution with mixing distribution $F_Z$ when, conditionally on $Z=z$, $X$ is distributed as a Poisson random variable with mean $z$.)
	
	In order to describe the phase transition in this model, we introduce the parameter
	\begin{equation}\label{nu}
	\nu\coloneqq\E(W^2)/\E(W).
	\end{equation}
	As we explain in Section \ref{comparisonsection} (in particular, see Remark \ref{remcritparameter}) this positive real number corresponds to the (asymptotic) mean of the offspring distribution in a branching process approximation of the clusters exploration in $NR_n(\mathbf{w})$.
	
	In \cite{BJR} it was shown that the graph undergoes a phase transition as $\nu$ passes $1$. In particular, if $\nu>1$ the largest component contains approximately $n\gamma$ vertices (where $\gamma\in (0,1)$), \textcolor{black}{whereas if $\nu \leq 1$ a largest component contains} a vanishing proportion of vertices. When $\nu>1$ the random graph is said to be \textit{super-critical}, whereas when $\nu< 1$ it is called \textit{sub-critical}. Finally, when $\nu=1$, the random graph is said to be \textit{critical}.
	
	In \cite{Hofstadt} the author provided a complete picture of the component structure in the critical $NR_n(\textbf{w})$ model when $1-F(x)=\Prob(W>x)$ decays as a power law and $\mathbf{w}=(w_i)_{i\in [n]}$ is as in \eqref{wj}. 
	
	More specifically, in \cite{Hofstadt} (Theorems 1.1 and 1.2) \textcolor{black}{it was shown that when $\nu=1$ and }
	\begin{equation}\label{six}
	1-F(x)\leq c_Fx^{-(\tau-1)}\hspace{0.2cm}(x\geq 0)
	\end{equation}
	for some constants $c_F>0,\tau>4$, then there is a constant $b>0$ such that, for any $A>1$ and for all $n\geq 1$, the $NR_n(\mathbf{w})$ random graph satisfies
	\begin{equation}\label{pp1}
	\mathbb{P}(A^{-1}n^{2/3}\leq |\CC_{max}|\leq A n^{2/3})\geq 1-b/A.
	\end{equation}
	On the other hand, \textcolor{black}{when $\nu=1$ and }
	\begin{equation}\label{five}
	\lim_{x\rightarrow \infty}x^{-(\tau-1)}(1-F(x))=c_F
	\end{equation}
	for some constants $c_F>0,\tau \in (3,4)$, then there exists a constant $b>0$ such that for any $A>1$ and for all $n\geq 1$, the $NR_n(\mathbf{w})$ model satisfies
	\begin{equation}\label{pp2}
	\mathbb{P}(A^{-1}n^{(\tau-2)/(\tau-1)}\leq |\CC_{max}|\leq A n^{(\tau-2)/(\tau-1)})\geq 1-b/A.
	\end{equation}
	(Actually  in \cite{Hofstadt} it is established a more general result, namely that \eqref{pp1} and \eqref{pp2} remain valid also after a small \textit{perturbation} of the vertex weights; see Theorems 1.1 and 1.2 in \cite{Hofstadt}.)
	
	For an explanation of the critical behaviour described by \eqref{pp1} and \eqref{pp2}, see section 1.3 in \cite{Hofstadt}, where the author also provided an heuristic description concerning the scaling limit of cluster sizes in both regimes $\tau\in (3,4)$ and $\tau>4$, studied extensively in \cite{Hscal1} and \cite{Hscal2}. (See also \cite{DVDH} for recent results concerning the scaling limits in the critical configuration model.)

	\section{Results}
	Our main results are the following three theorems, which we prove using probabilistic arguments based on martingales along the lines in \cite{NP,NCHperc}.
	\begin{thm}\label{teo1taugreater4}
		Let $\mathbf{w}=(w_i)_{i\in [n]}$ be defined as in \eqref{wj}. Suppose that there exist constants $\tau>4$ and $c_F>0$ such that $1-F(x)\leq c_Fx^{-(\tau-1)}$ for all $x\geq 0$. \textcolor{black}{Then, for any $A\geq 1$ and for all large enough $n$ we have
			\begin{equation}\label{firstbound}
			\Prob(|\CC_{\max}|>A n^{2/3})\leq \frac{c_1}{A}
			\end{equation}
			and
			\begin{equation}\label{lowertaugreater4}
			\Prob(|\CC_{\max}|<n^{2/3}/A)\leq \frac{c_2}{A^{1/4}},
			\end{equation}
			where $c_1,c_2>0$ are finite constants which depend on $c_F$ and $\tau$.}
	\end{thm}
	Strengthening our assumption on the distribution function $F$ which specifies the vertex weights $w_i$ through \eqref{wj}, we can prove a similar result for the case $\tau\in (3,4)$. \textcolor{black}{In particular, for the next two results we assume that there are constants $c_F>0$ and $\tau\in (3,4)$ such that
		\begin{equation}\label{COND}
		F(x)=1-c_Fx^{-(\tau-1)} \text{ for }x\geq c^{1/(\tau-1)}_F \text{ and }F(x)=0 \text{ for }x<c^{1/(\tau-1)}_F.
		\end{equation}
		As discussed in \cite{Hscal2}, in this case we have
		\[\mathbb{E}(W)=c^{1/(\tau-1)}_F\frac{\tau-1}{\tau-2} \text{ and }\mathbb{E}(W^2)=c^{1/(\tau-1)}_F\frac{\tau-1}{\tau-3},\]
		whence
		\[\nu=\frac{\mathbb{E}(W^2)}{\mathbb{E}(W)}=c^{1/(\tau-1)}_F\frac{\tau-2}{\tau-3}.\]
		Therefore criticality is reached when $c^{1/(\tau-1)}_F=(\tau-3)(\tau-1)^{-1}$. The main advantage for assuming an explicit  analytical form for $1-F(x)$ as given by (\ref{COND}) is that, in this case, we have an \textit{exact} expression for the vertex weights, which helps the computations.}
	\begin{thm}\label{teo2tauin34}
		Let $\mathbf{w}=(w_i)_{i\in [n]}$ be defined as in \eqref{wj}. \textcolor{black}{Suppose that there exist constants $\tau\in (3,4)$ and $c_F>0$ such that (\ref{COND}) holds. Then, for any $A\geq  1$ and for all large enough $n$, we have
			\begin{equation}\label{secondbound}
			\Prob\Big(|\CC_{\max}|> A n^{(\tau-2)/(\tau-1)}\Big)\leq  \frac{c_3}{A},
			\end{equation}
			and 
			\begin{equation}\label{lowertauinthreefour}
			\Prob\Big(|\CC_{\max}|< n^{(\tau-2)/(\tau-1)}/A\Big)\leq \frac{c_4}{A},
			\end{equation}
			where $c_3,c_4>0$ are finite constants which depend on $c_F$ and $\tau$.}
	\end{thm}
	\textcolor{black}{\begin{obs}
			Throughout the rest of the article, sometimes we keep writing that constants depend on $\tau$ \textit{and} $c_F$ even though, in the case where (\ref{COND}) is assumed, the dependence is only in terms of $\tau$ since in this case, as we have seen earlier, criticality ($\nu=1$) is reached when $c^{1/(\tau-1)}_F=(\tau-3)(\tau-1)^{-1}$.
	\end{obs}}
	\begin{obs}
		\textcolor{black}{With some extra effort it would be possible to provide an expression for the constants $c_i$ which appear in the statements of Theorems \ref{teo1taugreater4} and \ref{teo2tauin34}.} However, we refrained to do so in order to provide simpler calculations.
	\end{obs}
	Next result shows that we can considerably improve the polynomial upper bounds which appear in \eqref{firstbound} and \eqref{secondbound}. To achieve this, however, we need to have at our disposal the precise analytical form of $1-F(x)$ in both cases $\tau\in (3,4)$ and $\tau>4$. \textcolor{black}{That is, we need to assume that (\ref{COND}) holds in both regimes. We believe though that the exponential bounds displayed in the next result can be achieved without assuming (\ref{COND}). In particular, for the case $\tau>4$, assuming $1-F(x)\leq c_Fx^{-(\tau-1)}$ would suffice to obtain exponential tail probabilities.}
	\begin{thm}\label{teoexp}
		\textcolor{black}{Let $\mathbf{w}=(w_i)_{i\in [n]}$ be defined as in \eqref{wj}. Suppose that there exist constants $\tau>3$ and $c_F>0$ such that (\ref{COND}) holds. Then there exist constants $n_0\in \mathbb{N}$ and $A_0\geq 1$ such that the following statements hold. If $\tau>4$ then, for any $A_0\leq A =O\Big(n^{\frac{(\tau-4)\wedge 1}{3(\tau-1)}}\Big)$ and for all $n\geq n_0$, we have
			\begin{equation*}
			\mathbb{P}(|\CC_{\max}|>An^{2/3})\leq \frac{c_5}{A}e^{-c_6A^2(A-4)},
			\end{equation*}
			for some finite constants $c_5,c_6>0$ which depend on $\tau$.
			If $3<\tau<4$ then, for any $A_0\leq A =O\left(n^{\frac{(5-\tau)}{3(\tau-1)}}\right)$ and for all $n\geq n_0$, we have	
			\begin{equation*}
			\mathbb{P}(|\CC_{\max}|>An^{\frac{\tau-2}{\tau-1}})\leq \frac{c_7}{A}e^{-c_8A},
			\end{equation*}
			for some finite constants $c_7,c_8>0$ which depend on $\tau$. }
	\end{thm}
	\textcolor{black}{\begin{obs}
			We remark that it would be possible to provide expressions for the constants $c_6,c_8$ which appear in the argument of the exponential functions in the previous theorem; for instance, we can compute that 
			\[c_6=\frac{\big(\frac{\tau-2}{\tau-1}\big)^2}{128\big(\frac{\mathbb{E}(W^3)}{\mathbb{E}(W)}+4\big)}.\]
			However, most likely these constants are not the \textit{exact} constants in the asymptotic expansion of $\mathbb{P}(|\mathcal{C}_{\max}|>k)$, whence we preferred to report only the dependence on $A$ that we managed to obtain with the martingale method. Moreover, the constant $A_0$ could also be computed (actually do so in our proof); but we preferred to be a bit less precise for the sake of readability. 
	\end{obs}}
	
	\begin{obs}
		Comparing our estimates in \eqref{firstbound} and \eqref{secondbound} with those appearing in \eqref{pp1} and \eqref{pp2}, we see that our arguments allow us to recover the bounds in \cite{Hofstadt} provided that, for the case $\tau \in (3,4)$, we strengthen our assumption concerning the distribution function $F$ which specifies the vertex weights. Indeed, for the case $\tau\in (3,4)$, \cite{Hofstadt} only assumed that $\lim_{x\rightarrow\infty}x^{\tau-1}(1-F(x))=c_F$, whereas in our proof of Theorem \ref{teo2tauin34} we make use of the precise analytical form of $F(x)$ for \textit{every} $x$, and not only for \textit{large} $x$, whence the assumption $\lim_{x\rightarrow\infty}x^{\tau-1}(1-F(x))=c_F$ does not suffice. Under this stronger assumption, however, we can considerably strengthen the polynomial bounds stated in \eqref{firstbound} and \eqref{secondbound}, as we manage to obtain exponential upper bounds as illustrated in Theorem \ref{teoexp}. On the other hand, our upper bound in \eqref{lowertaugreater4} for the probability of observing an unusually small maximal component for the case $\tau>4$ is weaker with respect to the one established in \cite{Hofstadt}, but our proof only rely on probabilistic arguments and do not require involved analytical calculations as in \cite{Hofstadt}.
	\end{obs}
	To prove our main results we exploit a connection (which we describe in the next section) between clusters exploration in $NR_n(\textbf{w})$ and a suitable family of branching processes, first appeared in \cite{NRn} and also used in \cite{Hofstadt}.

	\section{Branching process approximation of clusters exploration}\label{comparisonsection}
	We start by describing the clusters exploration in $NR_n(\textbf{w})$ and subsequently we construct branching processes for which the exploration of (reduced versions of) their generated trees resembles that of components in the random graph $NR_n(\mathbf{w})$. More specifically, we begin by describing three alternatives procedures to explore clusters in the $NR_n(\mathbf{w})$ model, which we call \textbf{Alg.1}, \textbf{Alg.2} and \textbf{Alg.3}, and subsequently we compare these three approaches to other three procedures, called \textbf{Alg.1.BP}, \textbf{Alg.2.BP} and \textbf{Alg.3.BP}, which we later use to explore the above-mentioned branching process trees.
	In particular, we use \textbf{Alg.1} and \textbf{Alg.1.BP} to prove our upper bounds for the probability of observing unusually large maximal components in both ranges $\tau\in (3,4)$ and $\tau>4$, while we use \textbf{Alg.2} and \textbf{Alg.2.BP} to bound the probability of observing unusually small maximal clusters for the case $\tau>4$. The \textcolor{black}{probability of observing unusually small maximal components for the case} $\tau\in (3,4)$ is analysed by means of \textbf{Alg.3} and \textbf{Alg.3.BP}. These three approaches only differ in the way we choose the vertex from which we start exploring. Indeed, with \textbf{Alg.1} we start the exploration from a vertex selected uniformly at random, whereas in \textbf{Alg.2} the vertex from which we start the procedure is selected with probability proportional to its weight. Finally, in \textbf{Alg.3} we (deterministically) start the procedure from vertex $1$. In due course we will explain why these three similar, yet different explorations are indeed useful for us. Our descriptions somehow follow the one appearing in \cite{DEAROB}; see also \cite{NP}, \cite{DEA} and references therein.
	
	Let $G=([n],E)$ be any simple (undirected) random graph. 
	During our exploration process, each vertex will be \textit{active, explored} or \textit{unseen} and its status will change during the course of the procedure. At each time $t\in[n]$, a vertex is explored, so that at time $t$ there are $t$ explored vertices. In particular, at time $t=n$ all vertices in $G$ are in status explored.
	
	The exploration starts from a vertex $V_n$, which is selected in different ways according to the algorithm at hand, as we now describe. In \textbf{Alg.1}, the vertex $V_n$ is sampled uniformly at random from the vertex set $[n]$; in \textbf{Alg.2}, we let $V_n=i$ with probability $w_i/l_n$ for $i\in [n]$ (where we recall that $l_n=\sum_{j=1}^{n}w_j$ is the sum of all weights); finally, in \textbf{Alg.3} we (deterministically) choose $V_n=1$. 
	
	At time $t=0$ we set $V_n$ to active and all the other vertices are declared unseen.  
	Denote the set of unseen, active and explored vertices at the end of step $t$, by $\mathcal{U}_t$, $\mathcal{A}_t$ and $\mathcal{E}_t$, respectively. Hence we have that $\mathcal{A}_0=\{V_n\}$, $\mathcal{U}_0=[n]\setminus \{V_n\}$, and $\mathcal{E}_0=\emptyset$ (the empty set). At time $t=1$ we reveal all the \textit{unseen} neighbours of $V_n$; that is, we reveal all the vertices directly connected to $V_n$ in $G$. If we denote by $\mathcal{U}^*_1$ this subset of $\mathcal{U}_0$, then we have $\mathcal{U}^*_1=\{j\in \mathcal{U}_0: V_n \sim j\}$. Change the status of the vertices in $\mathcal{U}^*_1$ to active and declare $V_n$ explored, so that $\mathcal{A}_1=\mathcal{U}^*_1$, $\mathcal{E}_1=\{V_n\}$ and $\mathcal{U}_1=[n]\setminus (\mathcal{A}_1\cup \mathcal{E}_1)$. 
	Then we continue in this fashion. Namely, for every $t>1$, we proceed as follows.
	\begin{itemize}
		\item [(a)] If $|\mathcal{A}_{t-1}|\geq 1$ (i.e. if there is at least one active vertex at the end of step $t-1$), we let $u_{t}$ be the vertex in $\mathcal{A}_{t-1}$ with the smallest label.
		\item [(b)] If \textcolor{black}{$|\mathcal{A}_{t-1}|=0$ and $|\mathcal{U}_{t-1}|\geq 1$}, we let $u_{t}$ be a vertex chosen from $\mathcal{U}_{t-1}=[n]\setminus \mathcal{E}_{t-1}$ (the set of unseen vertices at the end of step $t-1$) with probability proportional to its weight, i.e. we let $u_t=j$ with probability $w_j/l'_n(t)$, where $l'_n(t)\coloneqq \sum_{i\in[n]\setminus\mathcal{E}_{t-1}}w_i$. 
		\item [(c)] If $|\AAA_{t-1}|=0=|\mathcal{U}_{t-1}|$ then $\mathcal{E}_{t-1}= [n]$; that is, all the vertices have been explored and we halt the procedure.
	\end{itemize} 
	Then we set
	\begin{align*}
	\mathcal{U}^*_t\coloneqq 
	\begin{cases}
	\{j\in \mathcal{U}_{t-1}: u_{t} \sim j\},& \text{ if $|\mathcal{A}_{t-1}|\geq 1$}\\ 
	\{j\in \mathcal{U}_{t-1}\setminus \{u_{t}\}: u_{t} \sim j\}, & \text{ if $|\mathcal{A}_{t-1}|=0$},
	\end{cases}
	\end{align*}
	we change the status of the vertices in $\mathcal{U}^*_t$ to active and declare $u_{t}$ explored, so that $\mathcal{A}_{t}=(\mathcal{U}^*_t\cup \mathcal{A}_{t-1})\setminus \{u_{t}\}$, $\mathcal{E}_t=\mathcal{E}_{t-1}\cup\{u_{t}\}$ and $\mathcal{U}_t=[n]\setminus (\mathcal{A}_t\cup \mathcal{E}_t)$. 
	
	Observe that
	\begin{equation}\label{rec}
	|\mathcal{A}_t|=
	\begin{cases}
	|\mathcal{A}_{t-1}|+|\mathcal{U}^*_t|-1, \quad & \text{ if } |\mathcal{A}_{t-1}|\geq 1\\
	|\mathcal{U}^*_t|, \quad & \text{ if } |\mathcal{A}_{t-1}|=0.
	\end{cases}
	\end{equation}
	Set $t_0\coloneqq 0$ and denote by $(t_i:i\geq 1)$ the ordered times \textcolor{black}{(prior to $n$)} at which the set of active vertices becomes empty, so that $|\mathcal{A}_{t_j}|=0$ for each $j$. By \eqref{rec} we see that
	\begin{equation}\label{At}
	|\mathcal{A}_{t_{j-1}+t}|=1+\sum_{i=1}^t \big(|\mathcal{U}^*_{t_{j-1}+i}|-1\big), \hspace{0.2cm}1\leq t\leq t_j-t_{j-1}.
	\end{equation}
	Also, denoting by $\mathcal{C}_j$ the $j$-th explored component (so that $\mathcal{C}_1=\CC(V_n)$), we have that $|\mathcal{C}_j|=t_j-t_{j-1}$ for all $j$. 
	
	Therefore, thanks to the exploration process that we have just described, we can rewrite the probability of observing clusters of given sizes as the probability that the positive excursions of the random process $(|\mathcal{A}_t|)_t$ last for some specific number of steps.
	
	Our next goal is to construct \textit{mixed Poisson branching processes} for which the exploration of their (reduced) trees resembles that of clusters in the $NR_n(\textbf{w})$ random graph.
	
	Before starting, let us introduce a random variable $M$ with distribution
	\begin{equation}\label{M}
	\mathbb{P}(M=m)=\frac{w_m}{l_n} \hspace{0.2cm}(m\in [n]);
	\end{equation}
	we call the law of $M$ the \textit{mark distribution}, and in this context elements of $[n]$ are called \textit{marks}.
	
	The idea is to construct, sequentially, mixed Poisson branching processes and to explore \textit{thinned} versions of the their generated trees so that the exploration of different trees is comparable (in distribution) to the clusters exploration in the $NR_n(\textbf{w})$ random graph. 
	
	
	As it occurred for the clusters exploration of $NR_n(\textbf{w})$, also in this setting we distinguish between three alternatives procedures, which differ in the way we choose the mark of the root in the tree from which we start the exploration. As anticipated at the beginning of this section, the three procedures are called \textbf{Alg.1.BP}, \textbf{Alg.2.BP} and \textbf{Alg.3.BP}.
	
	During the exploration of these (reduced) branching process trees, we adopt the following notation. For each step $t\in \mathbb{N}_0$ of the procedure, we denote by $\mathcal{A}^{BP}_t$ the set of \textit{active marks} and by $\mathcal{E}^{BP}_t$ the set of \textit{explored marks} at the end of step $t$.
	
	We start \textcolor{black}{by constructing a (mixed) branching process} as follows. We assign to the root of the tree, call it $\rho$, a mark $J_0$, which is chosen in different ways according to the procedure employed. Specifically, in \textbf{Alg.1.BP} we let $J_0$ be a mark selected uniformly at random from the \textit{mark space} $[n]$; in \textbf{Alg.2.BP} instead, \textcolor{black}{we let $J_0=i$ with probability $w_i/l_n$, for $i\in [n]$}; finally, in \textbf{Alg.3.BP} we (deterministically) choose $J_0=1$. We give to $\rho$ a $\text{Poisson}(w_{J_0})$ number of children, say $Y_{\rho}$. Iteratively, to the $i$-th individual in generation $g\geq 1$ (if any) we assign a random mark $J^g_i$ distributed as $M$ in \eqref{M} and a $\text{Poisson}(w_{J^g_i})$ number of children, say $Y^g_i$. Marks are assigned independently, and they're also independent of the marks produced in previous generations; moreover, vertices produce offspring independently, so that the $Y^g_i$ are independent random variables. In particular, we see that the $Y^g_i$ are i.i.d. since the (random) marks are all distributed as $M$. However, note that in both settings where the mark $J_0$ is sampled uniformly at random from $[n]$ (i.e. when \textbf{Alg.1.BP} is used) and when $J_0$ is deterministically set equal to $1$ (i.e. when \textbf{Alg.3.BP} is employed), then $Y_{\rho}$ is not distributed as the $Y^g_i$, even though the random variables $Y_{\rho},Y^g_i$ are all independent.
	
	The exploration starts as follows. At time $t=0$, we declare $J_0$ active (whence $\mathcal{A}^{BP}_0=\{J_0\}$) and we set $\mathcal{E}^{BP}_0\coloneqq \emptyset$. For every $t\in \mathbb{N}$, we proceed as follows.
	\begin{itemize}
		\item [(a)] If $|\AAA^{BP}_{t-1}|\geq 1$ (i.e. if there is at least one active mark at the end of step $t-1$), we let $m^{BP}_{t}$ be the smallest element of $\AAA^{BP}_{t-1}$ and denote by $v_{t}$ the corresponding vertex. Note that $m^{BP}_1=J_0$ and $v_1=\rho$.
		\item [(b)] If $|\AAA^{BP}_{t-1}|=0$ and $\mathcal{E}^{BP}_{t-1}\neq [n]$ (so that there are still marks to be explored), \textcolor{black}{we start exploring} a \textit{new} mixed branching process tree defined as follows. We let $m^{BP}_{t}$ be a random mark selected from $[n]\setminus \mathcal{E}^{BP}_{t-1}$ (the set of unexplored marks at the end of step $t-1$) with probability $w_j/l'_n(t)$, where $l'_n(t)\coloneqq \sum_{i\in[n]\setminus\mathcal{E}_{t-1}^{BP}}w_i$, and we assign it to a vertex $v_t$ that constitutes the root of the new tree. \textcolor{black}{(Note that we have used the same notation for the sum of unexplored \textit{weights} and the sum of unexplored \textit{vertices} in the clusters exploration; however, this should not cause any confusion.)} Then we give to $v_t$ a $\text{Poisson}(w_{m^{BP}_{t}})$ number of children. 
		Iteratively, to the $i$-th individual in generation $g\geq 1$ (if any), we assign a random mark $J^g_i$ distributed as $M$ and a $\text{Poisson}(w_{J^g_i})$ number of children, say $Y^g_i$. 
		Marks are assigned independently to individuals in each generation, and they're also independent of the marks produced in previous generations; moreover, vertices produce offspring independently, 
		so that the $Y^g_i$ are independent random variables. 
		In particular we see that the $Y^g_i$ are i.i.d. since the (random) marks are all distributed as $M$, but the offspring of the root is not distributed like the $Y^g_i$ (even though it is independent of these random variables).
		\item [(c)] If $|\AAA^{BP}_{t-1}|=0$ and $\mathcal{E}^{BP}_{t-1}=[n]$, then all the marks are in status explored and we stop the procedure.
	\end{itemize}
	Denote by $J^{v_t}_1,\dots,J^{v_t}_{X_{v_t}}$ the marks of the children (if any) of vertex $v_t$, where $X_{v_t}$ denotes the number of children of $v_t$. Define $\mathcal{M}_t\coloneqq \{J^{v_t}_l:1\leq l\leq  X_{v_t}\}$, the collection of \textit{all} marks assigned to the children of $v_t$. Note that $|\mathcal{M}_t|= X_{v_t}$. 
	
	Moreover, we construct another set of marks $\widetilde{\mathcal{M}}^{}_t\subset \mathcal{M}_t$ as follows. If $X_{v_t}=0$ then we simply set $\widetilde{\mathcal{M}}^{}_t=\emptyset$, otherwise we define:
	\begin{itemize}
		\item $\mathcal{L}^{v_t}_0\coloneqq \emptyset$;
		\item $\mathcal{L}^{v_t}_l\coloneqq (J^{v_t}_1,\dots,J^{v_t}_l)$ for $1\leq l\leq X_{v_t}$,
	\end{itemize}
	and let (for $1\leq l\leq X_{v_t}$)
	\begin{equation}\label{assignment}
	J^{v_t}_l\in \widetilde{\mathcal{M}}^{}_t \Leftrightarrow J^{v_t}_l\notin (\mathcal{A}^{BP}_{t-1}\cup \{m^{BP}_t\})\cup \mathcal{E}^{BP}_{t-1}\cup \mathcal{L}^{v_t}_{l-1}.
	\end{equation}
	In words, the $l$-th mark $J^{v_t}_l$ is added to the set $\widetilde{\mathcal{M}}_t$ if, and only if, it did not appear at a previous step $i\leq t-1$ and it differs from its \textquotedblleft sister marks\textquotedblright $J^{v_t}_1,\dots,J^{v_t}_{l-1}$ (if any).
	
	Note that, if $|\AAA^{BP}_{t-1}|\geq 1$, then $\mathcal{A}^{BP}_{t-1}\cup \{m^{BP}_t\}=\mathcal{A}^{BP}_{t-1}$ as $m^{BP}_t$ is taken from $\mathcal{A}^{BP}_{t-1}$. On the other hand, if $|\mathcal{A}^{BP}_{t-1}|=0$, then according to \eqref{assignment} we (rightfully) do not include within $\widetilde{\mathcal{M}}^{}_t$ those marks assigned to the children of $v_t$ which are equal to $m^{BP}_t$, the mark of (their parent) $v_t$.
	
	By our construction, $\mathcal{M}^{}_t$ is the set of \textit{all} marks assigned to the children of $v_t$, while $\widetilde{\mathcal{M}}^{}_t$ is the set of \textit{distinct} marks of these offspring which also differ from all the marks that we have seen up to the end of step $t-1$.
	
	We declare active all marks in the set $\widetilde{\mathcal{M}}^{}_t$ and we eliminate from the tree all sub-trees rooted at those children of $v_t$ whose marks have \textit{not} been inserted into $\widetilde{\mathcal{M}}_t$. We conclude step $t$ by declaring explored the mark $m^{BP}_t$. Therefore we update $\mathcal{A}^{BP}_t=(\widetilde{\mathcal{M}}^{}_t\cup \mathcal{A}^{BP}_{t-1})\setminus \{m^{BP}_t\}$ and $\mathcal{E}^{BP}_t=\mathcal{E}^{BP}_{t-1}\cup \{m^{BP}_t\}$.
	Note that at each step we explore precisely \textit{one} mark. 
	
	Set $\tau_0\coloneqq 0$ and denote by $(\tau_i:i\geq 1)$ the ordered times (prior to the termination of the procedure that we have just described) at which the set of active marks becomes empty, so that $|\mathcal{A}^{BP}_{\tau_j}|=0$ for all $j$. Observe that
	\begin{equation}\label{rec2}
	|\mathcal{A}^{BP}_t|=
	\begin{cases}
	|\mathcal{A}^{BP}_{t-1}|+|\widetilde{\mathcal{M}}_t|-1, \quad & \text{ if } |\mathcal{A}^{BP}_{t-1}|\geq 1\\
	|\widetilde{\mathcal{M}}_t|, \quad & \text{ if } |\mathcal{A}^{BP}_{t-1}|=0.
	\end{cases}
	\end{equation}
	\textcolor{black}{By \eqref{rec2} we see that
		\begin{equation}
		|\mathcal{A}^{BP}_{\tau_{j-1}+t}|=1+\sum_{i=1}^{t}\big(|\widetilde{\mathcal{M}}^{}_{\tau_{j-1}+i}|-1\big), \hspace{0.2cm}1\leq t< \tau_j-\tau_{j-1}.
		\end{equation}}
	
	The following proposition establishes the connection between the clusters exploration in $NR_n(\textbf{w})$ and the exploration of the (reduced) branching process trees that we have just described. (Recall that $\mathcal{C}_j$ denote the $j$-th explored component in $NR_n(\textbf{w})$.)
	\textcolor{black}{\begin{prop}\label{newnewprop}
			For $i\in \{1,2,3\}$In distribution we have that $|\mathcal{C}_j|=\tau_j-\tau_{j-1}$ for all $j$, provided we use \textbf{Alg.i} and \textbf{Alg.i.BP} to explore $NR_n(\textbf{w})$ and the branching process tress, respectively.
	\end{prop}}
	The interested reader can find a proof of Proposition \ref{newnewprop} in the appendix at the end of the paper. We conclude this section with a few remarks concerning some of the quantities that we have just introduced.
	
	Let $W_n$ be a random variable with distribution function
	\begin{equation}\label{fn}
	F_n(x)\coloneqq \frac{1}{n}\sum_{i=1}^{n}\mathbbm{1}_{\{w_i\leq x\}}.
	\end{equation}
	Observe that, when $J_0$ is selected uniformly at random from $[n]$, then
	\[\mathbb{P}(w_{J_0}\leq x)=\sum_{i=1}^{n}\mathbbm{1}_{\{w_i\leq x\}}\mathbb{P}(J_0=i)=F_n(x),\]
	so that $w_{J_0}\overset{d}{=}W_n$.	Next we show that $w_M$ has the same law as the \textit{size-biased distribution} of $W_n$ (where $M$ is specified in (\ref{M})).
	\begin{defn}
		For a non-negative random variable $X$ with $\mathbb{E}(X)\in (0,\infty)$, define $X^*$ through
		\begin{equation}\label{sizebiased}
		\mathbb{P}(X^*\leq x)\coloneqq \frac{\mathbb{E}\big(X\mathbbm{1}_{\{X\leq x\}}\big)}{\mathbb{E}(X)}.
		\end{equation}
		We call $X^*$ the size-biased distribution of $X$.
	\end{defn}	
	Observe that, since $w_{J_0}\overset{d}{=}W_n$,
	\begin{equation*}
	\mathbb{E}\big(W_n\mathbbm{1}_{\{W_n\leq x\}}\big)=\mathbb{E}\big(w_{J_0}\mathbbm{1}_{\{w_{J_0}\leq x\}}\big)=\frac{1}{n}\sum_{i=1}^{n}w_i\mathbbm{1}_{\{w_i\leq x\}}.
	\end{equation*}
	Also, $\mathbb{E}(W_n)=\mathbb{E}(w_{J_0})=l_n/n$, and therefore 
	\begin{equation*}
	\mathbb{P}\left(W^*_n\leq x\right)=\frac{\mathbb{E}\big(W_n\mathbbm{1}_{\{W_n\leq x\}}\big)}{\mathbb{E}(W_n)}=\frac{1}{l_n}\sum_{i=1}^{n}w_i\mathbbm{1}_{\{w_i\leq x\}}
	=\sum_{i=1}^{n}\mathbbm{1}_{\{w_i\leq x\}}\mathbb{P}(M=i)=\mathbb{P}(w_M\leq x).
	\end{equation*}
	Thus $W^*_n\overset{d}{=}w_M$. 
	\begin{obs}\label{remcritparameter}
		Let us briefly explain why the parameter $\nu$ defined in (\ref{nu}) is the one characterizing the phase transition in the $NR_n(\textbf{w})$ random graph. 
		Let $W$ be a random variable with distribution function $F$, and suppose that $F$ satisfies either (\ref{six}) or (\ref{five}). Then $\E[W^2]<\infty$ and, by dominated convergence, it is possible to show that (as $n\rightarrow \infty$) $n^{-1}\sum_{i\in [n]}^{}w^2_i\rightarrow \E\left[W^2\right]$. Also, $n^{-1}\sum_{i\in [n]}^{}w_i\rightarrow \E(W)$ (so that in particular $l_n=\Theta(n)$) and therefore
		\[\nu_n\coloneqq \frac{\sum_{i\in [n]{}}^{}w^2_i}{l_n}=\frac{\sum_{i\in [n]{}}^{}w^2_i}{\sum_{i\in [n]}^{}w_i}\longrightarrow \frac{\E(W^2)}{\E(W)}=\nu.\]
		\textcolor{black}{Since the mean offspring distribution is $\nu_n=\mathbb{E}(w_M)=\E(W^*_n)$, we conclude that $\E(W^*_n)$ converges to $\nu$ as $n\rightarrow \infty$; that is, $\nu$ is the asymptotic mean offspring distribution of the branching processes whose trees have been used to \textit{approximate}} the clusters exploration in $NR_n(\textbf{w})$. The idea is that the \textit{reduced} trees, whose exploration is equivalent in distribution to the clusters exploration in the $NR_n(\textbf{w})$ model, are sufficiently \textit{close} to the original trees, so that their mean offspring distribution is roughly $\nu_n$. In turns, this tells us that the average number of newly discovered vertices at each step of the exploration in $NR_n(\textbf{w})$ is approximately $\nu_n\sim \nu$. Therefore, the random graph is expected to be critical precisely when $\nu=1$.
	\end{obs}
	
	\begin{obs}
		We also remark that, when the exponent $\tau$ characterising the power-law behaviour of the distribution function $F$ (which specifies the vertex weights) is such that $\tau \in (4,\infty)$ then, denoting by $W$ a random variable with distribution function $F$, we have $\mathbb{E}(W^3)<\infty$. On the other hand, if $\tau\in (3,4)$ then $\mathbb{E}(W^3)$ is not finite.
	\end{obs}

	\section{Proofs}
	In this section we are going to prove Theorems \ref{teo1taugreater4}, \ref{teo2tauin34} and \ref{teoexp}. Before proving these results, however, we list some useful facts in the next subsection.
	
	\subsection{Preliminaries}\label{subprel}
	The proofs of the next few results are postponed to Subsection \ref{subsectionauxiliary}. We start by establishing a simple lemma, which gives us information concerning the order of growth of the vertex weights.
	\textcolor{black}{\begin{lem}\label{maxw}
			Let $\tau>3$ and $c_F>0$. If (\ref{COND}) holds, then $w_i=\left(n c_F/i\right)^{1/(\tau-1)}$ for $1\leq i\leq n-1$. If $1-F(x)\leq c_Fx^{-(\tau-1)}$ for all $x\geq 0$, then
			$w_i\leq\left(n c_F/i\right)^{1/(\tau-1)}$; in particular, 
			$w_1=O\left(n^{1/(\tau-1)}\right)$.
	\end{lem}}
	The next result provides bounds on $|\nu_n-1|$, the distance between the mean offspring distribution of the branching processes that we use to approximate the clusters exploration in $NR_n(\textbf{w})$ and the critical value $\nu=1$. In particular, the next result quantifies the rate of convergence of $\nu_n$ to $1$ and it also provides information concerning the second moment of $W^*_n$ (the size-biased distribution of $W_n$).
	\begin{prop}\label{ordersize}
		Let $W$ be a random variable with distribution $F$, let $W_n$ be a random variable with distribution $F_n$ as in \eqref{fn}, and let $W_n^*$ be its size biased distribution. Suppose that $1-F(x)\leq c_F x^{-(\tau-1)}$ for all $x\geq 0$, where $\tau >3$ and $c_F$ is a positive constant. Then, for all large enough $n$, we have that
		\begin{equation}\label{nun1}
		|\nu_n-1|\leq C_1n^{-\frac{\tau-3}{\tau-1}}
		\end{equation}
		for some finite constant $C_1>0$ which depends on $c_F$ and $\tau$. In addition, if $\tau>4$ then, for all large enough $n$, we have that
		\begin{equation}\label{Ewnstart}
		\bigg|\E((W_n^*)^2)-\frac{\E(W^3)}{\E(W)}\bigg|\leq C_2n^{-\frac{\tau-4}{\tau-1}}
		\end{equation}
		for some finite constant $C_2>0$ which depends on $c_F$ and $\tau$.
	\end{prop}
	Next we introduce a stochastic domination result (the counterpart of Lemma 5 in \cite{NP} in this inhomogeneous setting) which involves a random walk that later on we will use to dominate the process arising from the exploration of the reduced trees generated by the mixed Poisson branching processes of Section 4.
	
	Let $(\Upsilon_i)_{i\geq 1}$ be a sequence of independent random variables, such that each $\Upsilon_i$ has a mixed Poisson distribution with random parameter $w_{M_i}$, where $(M_i)_{i\geq 1}$ is a sequence of independent random variables, all distributed as $M$ in \eqref{M}, with $w_i$ as in \eqref{wj}. Set $S_0:=1$ and define, for $t\in \mathbb{N}_0$,
	\begin{equation}\label{st}
	S_t=1+\sum_{i=1}^{t}(\Upsilon_i-1).    
	\end{equation}
	Given any $H\in \mathbb{N}$, we set
	\begin{equation}\label{gamma}
	\gamma \coloneqq  \inf\{t\geq 1:S_t=0\text{ or }S_t\geq H\}.
	\end{equation}
	The next result, which is the counterpart of Lemma 5 in \cite{NP}, states that the (conditional) law of the overshoot $S_{\gamma}-H$, given $S_{\gamma}\geq H$, is stochastically dominated by the $\text{Poisson}(w_1)$ distribution.
	\begin{lem}\label{OurLemma6}
		Let $H\in \mathbb{N}$, and let $S_t$ and $\gamma$ be as above. Let $Y_{w_1}$ be a Poisson random variable with mean $w_1$, and let $\Sigma \subset \mathbb{N}$ be a set of positive integers. Then, for any $k\geq 1$, we have that
		\begin{equation*}
		\mathbb{P}(S_{\gamma}-H\geq k|S_{\gamma}\geq H, \gamma \in \Sigma)\leq 
		\mathbb{P}(Y_{w_1}\geq k).
		\end{equation*}
	\end{lem}
	The following corollary is straightforward.
	\begin{cor}\label{OurCorollary6}
		If $Y_{w_1}$ is a Poisson random variable with mean $w_1$ and $f$ is an increasing real function, then with the notation of the previous lemma we have
		\begin{equation*}
		\mathbb{E}(f(S_{\gamma}-H)|S_{\gamma}\geq H, \gamma \in \Sigma)\leq \mathbb{E}(f(Y_{w_1})).
		\end{equation*}
	\end{cor}
	We conclude by recalling a basic result, the Optional Stopping Theorem, which we repeatedly use in the sequel. Its proof can be found in any advanced probability textbook.
	\begin{thm}\label{OptStopTh}
		Let $(X_i)_{i\in \mathbb{N}_0}$ be a martingale and let $\tau_1,\tau_2$ be stopping times with $0\leq \tau_1\leq \tau_2$. Suppose that $\tau_2$ is bounded. Then
		\begin{equation}\label{OST}
		\E(X_{\tau_1})=\E(X_{\tau_2}).
		\end{equation}
		If $(X_i)_{i\in \mathbb{N}_0}$ is a submartingale then \eqref{OST} has to be changed with $\E(X_{\tau_1})\leq\E(X_{\tau_2})$; if $(X_i)_{i\in \mathbb{N}_0}$ is a supermartingale, then \eqref{OST} has to be changed with $\E(X_{\tau_1})\geq\E(X_{\tau_2})$.
	\end{thm}

	\subsection{Proof of Theorems \ref{teo1taugreater4}, \ref{teo2tauin34} and \ref{teoexp} -- the probability of large maximal components}
	To prove the results of this subsection, we explore clusters and branching process trees by means of \textbf{Alg.1} and \textbf{Alg.1.BP} respectively; that is, we sample $V_n$ and $J_0$ uniformly at random from $[n]$.
	
	The upper bounds for the probabilities of observing maximal components containing more than $An^{2/3}$ and $An^{\frac{\tau-2}{\tau-1}}$ vertices stated in Theorems \ref{teo1taugreater4} and \ref{teo2tauin34}, respectively, are proved through Lemmas \ref{upperbound} and \ref{lem45} below, which are unaffected by the specific value of $\tau$ (the exponent characterising the power law decay of the distribution function $F$ which specifies the vertex weights through \eqref{wj}).
	
	Specifically, with the first lemma we obtain an upper bound for the probability that $|\mathcal{C}(V_n)|$ is larger than $k$, and then we use the second lemma to control an expected value which appears in our upper bound for
	\begin{equation}\label{probunifvertex}
	\mathbb{P}(|\mathcal{C}(V_n)|>k).
	\end{equation}
	The upper bounds for the probabilities involving $|\mathcal{C}_{\max}|$ will be deduced from our upper bounds on \eqref{probunifvertex} by means of a standard argument, which consists in bounding the probability that $|\CC_{\max}|>k$ by the probability that there are more than $k$ vertices lying in components containing at least $k$ nodes, and then using Markov's inequality to bound the latter probability.
	
	As a first step toward obtaining an upper bound for \eqref{probunifvertex} we show how such probability can be bounded from above by the probability that a random walk stays positive for $k$ steps.
	
	To this end note that, recalling the algorithmic procedure \textbf{Alg.1.BP} to explore the branching process trees of Section 4, we have $\widetilde{\mathcal{M}}_t\subset \mathcal{M}_t$ (because $\mathcal{M}_t$ is formed by \textit{all} the marks of the children of $v_t$, whereas $\widetilde{\mathcal{M}}^{}_t$ only contains those marks which did not appear at earlier steps). This implies that $|\widetilde{\mathcal{M}}^{}_t|\leq |\mathcal{M}^{}_t|$ for all $t$ and hence in particular
	\begin{multline}\label{asa}
	\mathbb{P}(|\mathcal{C}(V_n)|> k)=\mathbb{P}\Big(1+\sum_{i=1}^{t}\big(|\widetilde{\mathcal{M}}^{}_i|-1\big)>0\hspace{0.15cm}\forall t\in [k]\Big)\\
	\leq \mathbb{P}\Big(1+\sum_{i=1}^{t}\big(|\mathcal{M}^{}_i|-1\big)>0\hspace{0.15cm}\forall  t\in [k]\Big).
	\end{multline}
	
	Recall that the $\left|\mathcal{M}^{}_i\right|=X_{v_i}$ 
	are  independent random variables but they are not identically distributed. Indeed, $|\mathcal{M}_1|$ has a $\text{Poisson}(w_{J_0})$ distribution, with $J_0$ uniformly distributed on $[n]$, whereas $|\mathcal{M}_i|$ (for $2\leq i\leq k$), on the event appearing in \eqref{asa}, has a mixed Poisson distribution with random parameter $w_{M_i}$, 
	where the marks $M_i$ ($i\geq 2$) are independent identically distributed random variables with distribution as $M$ in \eqref{M}. 

	Thus, in order to obtain an upper bound for \eqref{asa} involving a sequence of i.i.d. random variables, we 
	need to substitute $|\mathcal{M}^{}_1|$ with an independent mixed Poisson random variable with random parameter $w_{M_1}$, where $M_1$ is distributed as $M$ and is independent of $(M_i)_{i\geq 2}$.
	
	To achieve this, let's recall that if $Z$ is any random variable and $f,g$ are arbitrary increasing functions, then $\E(f(Z)g(Z))\geq \E(f(Z))\E(g(Z))$, see for instance Lemma 2.21 in \cite{Ross2}. Therefore, if $Z$ is a non-negative random variable with finite positive mean (so that we can define its size-biased distribution), taking the increasing functions $f(z)\coloneqq z$ and $g(z)\coloneqq \mathbbm{1}_{\{z> x\}}$, we obtain $\E\big(Z\mathbbm{1}_{\{Z> x\}}\big)\geq \E(Z)\mathbb{P}(Z>x)$, that is
	\[\mathbb{P}(Z^*>x)=\frac{\E\big(Z\mathbbm{1}_{\{Z> x\}}\big)}{\E(Z)}\geq \mathbb{P}(Z>x).\]
	Therefore $Z\preccurlyeq Z^*$, i.e. the random variable $Z^*$ stochastically dominates $Z$. Consequently, since $w_{J_0}\overset{d}{=}W_n$ and $w_M\overset{d}{=}W^*_n$ (recall the discussion before Remark \ref{remcritparameter}), then $W_n\preccurlyeq W^*_n$, and we obtain that the random variable $\left|\mathcal{M}^{}_1\right|$, which has the $\text{Poisson}(w_{J_0})$ distribution, is stochastically dominated by a mixed Poisson random variable with random parameter $w_{M_1}$.
	
	Now let $(\Upsilon_i)_{i\geq 1}$ be a sequence of independent random variables where \textit{each} $\Upsilon_i$ has a mixed Poisson distribution with random parameter $w_{M_i}$, with $(M_i)_{i\geq 1}$ an i.i.d. sequence of random variables all distributed as $M$ in \eqref{M}. From the previous paragraph we know that $\left|\mathcal{M}^{}_1\right|\preccurlyeq \Upsilon_1$, 
	while $\left|\mathcal{M}^{}_i\right|$ and $ \Upsilon_i$ are equal in distribution for $i\geq 2$. 
	Now consider the process defined in \eqref{st}, where $S_0\coloneqq 1$ and $S_i\coloneqq S_{i-1}+\Upsilon_i-1$ for $i\geq 1$, so that $S_t=1+\sum_{i=1}^{t}(\Upsilon_i-1)$, $t\in \mathbb{N}_0$. Since $(|\mathcal{M}_i|)_i$ and $(\Upsilon_i)_i$ are sequence of independent random variables we obtain
	\begin{equation}\label{imp}
	\mathbb{P}\Big(1+\sum_{i=1}^{t}\big(|\mathcal{M}^{}_i|-1\big)>0\hspace{0.15cm}\forall  t\in [k]\Big)
	\leq \mathbb{P}\big(S_t>0\hspace{0.15cm}\forall \text{  }t\in [k]\big).
	\end{equation}

	
	It follows from \eqref{asa} and \eqref{imp} that, in order to obtain an upper bound for $\mathbb{P}(|\CC(V_n)|\geq k)$ which, as we said earlier, we subsequently use to derive our upper bounds for the probabilities of observing unusually large maximal clusters in both regimes $\tau\in (3,4)$ and $\tau>4$, we need to estimate the probability on the right-hand side of \eqref{imp}, i.e. the probability that a $\mathbb{Z}$-valued random walk stays positive for $k\in \mathbb{N}$ steps.
	
	\begin{lem}\label{upperbound}
		Let $k\geq 1$, and let $H,H'$ be positive integers with $H'\geq k$. Consider $S_t$ as in (\ref{st}) and define
		\begin{equation*}
		\gamma \coloneqq \left\{ \begin{aligned}
		& \inf\{t< H':S_t=0\text{ or }S_t\geq H\} &&  \text{ if }\{t< H':S_t=0\text{ or }S_t\geq H\}\neq \emptyset,\\
		& H' && \text{ if }\{t< H':S_t=0\text{ or }S_t\geq H\}=\emptyset.\\
		\end{aligned}
		\right.
		\end{equation*}
		Then, setting $\gamma^*\coloneqq \gamma \wedge k$, we have 
		\begin{equation}\label{mainboundforprobwelike}
		\Prob(|\CC(V_n)|>k)\leq\Prob(S_{\gamma^*}>0)\leq \frac{1-(1-\nu_n)\mathbb{E}(\gamma)}{H}+\frac{\mathbb{E}(\gamma)}{k}.
		\end{equation}
	\end{lem}
	
	\begin{proof}
		From \eqref{asa} and \eqref{imp} we can bound
		\begin{align}\label{tree}
		\nonumber\mathbb{P}(|\CC(V_n)|>k)&\leq \mathbb{P}(S_t>0\hspace{0.15cm}\forall t\in [k])\\
		\nonumber&\leq \mathbb{P}(S_{\gamma^*}>0)+\mathbb{P}(S_t>0\hspace{0.15cm}\forall t\in [k],S_{\gamma^*}\leq 0)\\
		\nonumber&=\mathbb{P}(S_{\gamma^*}>0)\\
		\nonumber&\leq \mathbb{P}(S_{\gamma^*}>0, \gamma<k)+\mathbb{P}(\gamma\geq k)\\
		&\leq \mathbb{P}(S_{\gamma}>0, \gamma<k)+\frac{\mathbb{E}(\gamma)}{k}.
		\end{align}
		Observe that, if $\gamma<k$, then $\gamma<H'$ (since $H'\geq k$). If this happens and $S_{\gamma}>0$, then we must have $S_{\gamma}\geq H$. Therefore we can bound
		\begin{equation}\label{doubtree}
		\mathbb{P}(S_{\gamma}>0, \gamma<k)\leq \mathbb{P}(S_{\gamma}\geq H,\gamma<H')\leq \mathbb{P}(S_{\gamma}\geq H).
		\end{equation}
		Thus combining (\ref{tree}) and (\ref{doubtree}) we arrive at
		\begin{align}\label{kokko}
		\mathbb{P}(|\CC(V_n)|>k)&\leq \mathbb{P}(S_{\gamma}\geq H)+\frac{\mathbb{E}(\gamma)}{k}.
		\end{align}
		The probability on the right-hand side of \eqref{kokko} can be bounded from above using Markov's inequality; this is possible because $S_{\gamma}$ is always non-negative. (Indeed, if $\gamma<H'$ then $S_{\gamma}\in \{0,H,H+1,\dots,\}$ and hence, in particular, $S_{\gamma}\geq 0$. If $\gamma=H'$, then $1\leq S_{H'-1}\leq H-1$ and hence $S_{H'}=S_{H'-1}+\Upsilon_{H'}-1\geq \Upsilon_{H'}\geq 0$. Thus $S_{\gamma}\geq 0$ always.) Consequently we can apply Markov's inequality to conclude that
		\begin{equation}\label{porco}
		\mathbb{P}(S_{\gamma}\geq H)\leq \frac{\mathbb{E}(S_{\gamma})}{H}.
		\end{equation}
		Recall that $\nu_n=\mathbb{E}(W^*_n)=\mathbb{E}(\Upsilon_1)$. Since $S_t+t(1-\nu_n)$ is a martingale (with respect to the filtration formed by the $\sigma$- fields $\mathcal{F}_t \coloneqq  \sigma(\Upsilon_i, i \leq t)$) and $\gamma\leq H'$ is a bounded stopping time, we can apply Theorem \ref{OptStopTh} (with $\tau_1=0$ and $\tau_2=\gamma$) to obtain $1=\mathbb{E}(S_{\gamma})+\mathbb{E}(\gamma)(1-\nu_n)$, or equivalently $\mathbb{E}(S_{\gamma})=1-\mathbb{E}(\gamma)(1-\nu_n)$. This shows that
		\begin{equation}\label{deus}
		\mathbb{P}(S_{\gamma}\geq H)\leq \frac{1-(1-\nu_n)\mathbb{E}(\gamma)}{H};
		\end{equation}
		substituting this bound into \eqref{kokko} yields the desired result.
	\end{proof}

	In order to obtain an upper bound for the expression on the right-hand side of \eqref{mainboundforprobwelike} (and so for the probability that $\mathcal{C}(V_n)$ contains more than $k$ nodes) we need to evaluate $\mathbb{E}(\gamma)$. This is achieved by means of the following
	\begin{lem}\label{lem45}
		Consider $S_t$ and $\gamma$ as in Lemma \ref{upperbound}. Define $b_H\coloneqq 2H^2\vee (\mathbb{E}((W^*_n)^2)+1-\nu_n)$. If $1-\nu_n>0$ we have
		\begin{equation}\label{oneminuspos}
		\mathbb{E}(\gamma)\mathbb{E}((W^*_n)^2)\bigg(1-\frac{1-\nu_n}{\mathbb{E}((W^*_n)^2)}H-\frac{2b_H}{\mathbb{E}((W^*_n)^2)H'}\bigg)\leq H+3w_1+\frac{w^2_1}{H}.
		\end{equation}
		On the other hand, if $1-\nu_n\leq 0$ we have
		\begin{equation}\label{oneminusneg}
		\mathbb{E}(\gamma)\mathbb{E}((W^*_n)^2)\bigg(1-\frac{\nu_n-1}{\mathbb{E}((W^*_n)^2)}\left[H+3w_1+\frac{w^2_1}{H}-1\right]-\frac{2b_H}{\mathbb{E}((W^*_n)^2)H'}\bigg)\\
		\leq H+3w_1+\frac{w^2_1}{H}.
		\end{equation}
	\end{lem}
	\begin{proof}
		Observe that the process defined by
		\begin{equation}
		M_t=S_t^2+t[\nu_n-1-\E((W_n^*)^2)]-2(\nu_n-1)\sum_{k=0}^{t-1}S_k
		\end{equation}
		is a martingale (as before, with respect to the filtration formed by the $\sigma$-fields $\mathcal{F}_t \coloneqq  \sigma(\Upsilon_i, i \leq t)$). This easily follows from the fact that
		\begin{align*}		\mathbb{E}\big(S^2_{t+1}|\mathcal{F}_t\big)&=\mathbb{E}\big(S^2_{t}+2S_t(\Upsilon_{t+1}-1)+(\Upsilon_{t+1}-1)^2|\mathcal{F}_t\big)\\
		&=S^2_t+2S_t(\nu_n-1)+\nu_n+\mathbb{E}((W^*_n)^2)-2\nu_n+1\\
		&=\textcolor{black}{S^2_t+2S_t(\nu_n-1)-(1-\nu_n-\mathbb{E}((W^*_n)^2))}.
		\end{align*}
		Then, by Theorem \ref{OptStopTh} with $\tau_1\coloneqq0$ and $\tau_2=\gamma\leq H'$, we obtain
		\begin{equation}\label{Mart}
		\nonumber1=\mathbb{E}(S^2_{\gamma})-\mathbb{E}(\gamma)\big(\mathbb{E}((W^*_n)^2)+1-\nu_n\big)-2(\nu_n-1)\mathbb{E}\Big(\sum_{k=0}^{\gamma-1}S_k\Big),
		\end{equation}
		from which we arrive at
		\begin{equation}\label{ppppppppp}
		\mathbb{E}(\gamma)\big(\mathbb{E}((W^*_n)^2)+1-\nu_n\big)+2(\nu_n-1)\mathbb{E}\Big[\sum_{k=0}^{\gamma-1}S_k\Big]\leq \mathbb{E}(S^2_{\gamma}).
		\end{equation}
		Next we bound from above the expected value of $S^2_{\gamma}$. We have that
		\begin{equation}\label{stAr}
		\mathbb{E}(S^2_{\gamma})=\mathbb{E}\big(S^2_{\gamma}\mathbbm{1}_{\{\gamma<H'\}}\big)+\mathbb{E}\big(S^2_{\gamma}\mathbbm{1}_{\{\gamma= H'\}}\big).
		\end{equation}
		Since $S^2_{t}=S^2_{t-1}+2S_{t-1}(\Upsilon_{t}-1)+(\Upsilon_{t}-1)^2$ for all $t\geq 1$, we can write
		\begin{equation*}
		\mathbb{E}\big(S^2_{\gamma}\mathbbm{1}_{\{\gamma= H'\}}\big)=\mathbb{E}\big(S^2_{H'-1}\mathbbm{1}_{\{\gamma= H'\}}\big)\\
		+2\mathbb{E}\big(S_{H'-1}(\Upsilon_{H'}-1)\mathbbm{1}_{\{\gamma= H'\}}\big) +\mathbb{E}\big((\Upsilon_{H'}-1)^2\mathbbm{1}_{\{\gamma= H'\}}\big).
		\end{equation*}
		On the event $\{\gamma=H'\}$ (which belongs to the $\sigma$-algebra generated by the the first $H'-1$ random variables $\Upsilon_i$) we have that $1\leq S_{H'-1}<H$ and hence $S^2_{H'-1}<H^2$. Moreover, $S_{H'-1}\mathbbm{1}_{\{\gamma= H'\}}$ and $\Upsilon_{H'}-1$  are independent, as well as $(\Upsilon_{H'}-1)^2$ and $\mathbbm{1}_{\{\gamma= H'\}}$. Therefore, when $n$ is large enough we can write (using Markov's inequality)
		\begin{align}\label{stAAr}
		\nonumber\mathbb{E}\big(S^2_{\gamma}\mathbbm{1}_{\{\gamma= H'\}}\big)&\leq H^2\mathbb{P}(\gamma=H')+2\mathbb{E}\big(S_{H'-1}\mathbbm{1}_{\{\gamma= H'\}}\big)(\nu_n-1)\\	\nonumber&\hspace{0.8cm}+\mathbb{E}\big((\Upsilon_{H'}-1)^2\big)\mathbb{P}(\gamma=H')\\
		\nonumber&\leq H^2\frac{\mathbb{E}(\gamma)}{H'}+2H|1-\nu_n|\frac{\mathbb{E}(\gamma)}{H'}+\mathbb{E}\big((\Upsilon_{H'}-1)^2\big)\frac{\mathbb{E}(\gamma)}{H'}\\
		\nonumber&\leq 2\frac{\mathbb{E}(\gamma)}{H'}(2H^2\vee (\mathbb{E}((W^*_n)^2)+1-\nu_n))\\
		&=2\frac{\mathbb{E}(\gamma)}{H'}b_H,
		\end{align}
		where the last inequality follows from the facts that
		\begin{equation*}
		\mathbb{E}\big((\Upsilon_{H'}-1)^2\big)=\mathbb{E}\big(w_M+w^2_M\big)+1-2\nu_n=\mathbb{E}((W^*_n)^2)+1-\nu_n, \text{ }H|1-\nu_n|\leq H
		\end{equation*}
		and we set $b_H\coloneqq 2H^2\vee (\mathbb{E}((W^*_n)^2)+1-\nu_n)$. Next we consider the term $\mathbb{E}\big(S^2_{\gamma}\mathbbm{1}_{\{\gamma<H'\}}\big)$. Note that on the event $\{\gamma<H'\}$ we have either $S_{\gamma}=0$ or $S_{\gamma}\geq H$. Therefore we can write
		\begin{equation*}
		\mathbb{E}\big(S^2_{\gamma}\mathbbm{1}_{\{\gamma<H'\}}\big)=\mathbb{E}\big(S^2_{\gamma}\mathbbm{1}_{\{\gamma<H'\}}\mathbbm{1}_{\{S_{\gamma}\geq H\}}\big)
		\leq \mathbb{E}\big(S^2_{\gamma}\mathbbm{1}_{\{S_{\gamma}\geq H\}}\big)=\mathbb{P}(S_{\gamma}\geq H)\mathbb{E}\big(S^2_{\gamma}|S_{\gamma}\geq H\big).
		\end{equation*}
		Now, setting $f(x)\coloneqq 2Hx+x^2$ (which is increasing for $x \geq 0$) and decomposing $S^2_{\gamma}=2H(S_{\gamma}-H)+(S_{\gamma}-H)^2+H^2=f(S_{\gamma}-H)+H^2$, applying Corollary \ref{OurCorollary6} we obtain 
		\begin{equation*}
		\E(f(S_{\gamma}-H)| S_{\gamma}\geq H)\leq\E(2HY_{w_1} +Y_{w_1}^2)=2Hw_1+w_1+w_1^2.
		\end{equation*}
		Therefore, since $H\geq 1$, we arrive at
		\begin{align}\label{condmean}
		\E(S_{\gamma}^2| S_{\gamma}\geq H) \leq H^2+2Hw_1 +w_1+ w_1^2
		\leq  H^2+3Hw_1+w_1^2.
		\end{align}
		Thus, using \eqref{deus}, we obtain
		\begin{multline}\label{stAAAr}
		\mathbb{E}\big(S^2_{\gamma}\mathbbm{1}_{\{\gamma<H'\}}\big)\leq \mathbb{P}(S_{\gamma}\geq H)\big(H^2+3Hw_1+w_1^2\big)\\
		\leq \frac{1-(1-\nu_n)\mathbb{E}(\gamma)}{H}\big(H^2+3Hw_1+w_1^2\big)\\
		=[1-(1-\nu_n)\mathbb{E}(\gamma)]\big(H+3w_1+w^2_1/H\big).
		\end{multline}
		Consequently, combining (\ref{stAr}), (\ref{stAAr}) and (\ref{stAAAr}) we arrive at
		\begin{equation*}
		\mathbb{E}(S^2_{\gamma})\leq [1-(1-\nu_n)\mathbb{E}(\gamma)]\big(H+3w_1+w^2_1/H\big)+2\frac{\mathbb{E}(\gamma)}{H'}b_H.
		\end{equation*}
		Therefore using \eqref{ppppppppp} we obtain
		\begin{multline}\label{rearr}
		\mathbb{E}(\gamma)\big(\mathbb{E}((W^*_n)^2)+1-\nu_n\big)+2(\nu_n-1)\mathbb{E}\big(\sum_{k=0}^{\gamma-1}S_k\big)\\
		\leq [1-(1-\nu_n)\mathbb{E}(\gamma)]\big(H+3w_1+w^2_1/H\big)+2\frac{\mathbb{E}(\gamma)}{H'}b_H.
		\end{multline}
		Now observe that, by definition of $\gamma$, we have $1\leq S_{k}\leq H-1$ for $1\leq k\leq \gamma -1$ and so
		\begin{equation}\label{impineq}
		\gamma\leq \sum_{k=0}^{\gamma-1}S_k\leq \gamma(H-1).
		\end{equation}
		If $1-\nu_n>0$ then, using \eqref{impineq} and rearranging the terms in \eqref{rearr} we obtain
		\begin{equation*}
		\nonumber\mathbb{E}(\gamma)\mathbb{E}((W^*_n)^2)\Big(1-\frac{1-\nu_n}{\mathbb{E}((W^*_n)^2)}\Big[H-3(w_1+1)-\frac{w^2_1}{H}\Big]-\frac{2b_H}{\mathbb{E}((W^*_n)^2)H'}\Big)
		\leq H+3w_1+\frac{w^2_1}{H}.
		\end{equation*}
		Since trivially 
		\begin{equation*}
		\frac{1-\nu_n}{\mathbb{E}((W^*_n)^2)}\Big[H-3(w_1+1)-\frac{w^2_1}{H}\Big]< \frac{1-\nu_n}{\mathbb{E}((W^*_n)^2)}H,
		\end{equation*}
		we arrive at
		\begin{equation}
		\nonumber\mathbb{E}(\gamma)\mathbb{E}((W^*_n)^2)\Big(1-\frac{1-\nu_n}{\mathbb{E}((W^*_n)^2)}H-\frac{2b_H}{\mathbb{E}((W^*_n)^2)H'}\Big)\leq H+3w_1+\frac{w^2_1}{H}.
		\end{equation}
		On the other hand, if $1-\nu_n\leq 0$, using \eqref{impineq} and rearranging the terms in \eqref{rearr} we obtain
		\begin{multline*}
		\nonumber\mathbb{E}(\gamma)\mathbb{E}((W^*_n)^2)\Big(1-\frac{\nu_n-1}{\mathbb{E}((W^*_n)^2)}\Big[H+3w_1+\frac{w^2_1}{H}-1\Big]-\frac{2b_H}{\mathbb{E}((W^*_n)^2)H'}\Big)
		\leq H+3w_1+\frac{w^2_1}{H},
		\end{multline*}
		completing the proof.
	\end{proof}

	\subsubsection{Proof of Theorem \ref{teo1taugreater4} (case $\tau>4$) -- $|\CC_{\max}|$ is unlikely to be larger than $An^{2/3}$}
	Note that, since	
	\begin{align*}
	\frac{|1-\nu_n|}{\mathbb{E}((W^*_n)^2)}H\leq \frac{|1-\nu_n|}{\mathbb{E}((W^*_n)^2)}\Big[H+3w_1+\frac{w^2_1}{H}-1\Big],
	\end{align*}
	we can use Lemma \ref{lem45} to bound
	\begin{multline}\label{ddddddaaaaaaa}
	\mathbb{E}(\gamma)\mathbb{E}((W^*_n)^2)\Big(1-\frac{|1-\nu_n|}{\mathbb{E}((W^*_n)^2)}\Big[H+3w_1+\frac{w^2_1}{H}-1\Big]-\frac{2b_H}{\mathbb{E}((W^*_n)^2)H'}\Big)
	\leq H+3w_1+\frac{w^2_1}{H},
	\end{multline}
	and this holds independently of the sign of $1-\nu_n$. Let $H=\lfloor n^{1/3}\rfloor$ and observe that, since $w_1=O(n^{1/(\tau-1)})$ (see Lemma \ref{maxw}) then, as $n\rightarrow \infty$, we obtain $w_1/H \asymp n^{-\frac{\tau-4}{3(\tau-1)}}\ll 1$ and hence, for all sufficiently large $n$,
	\begin{equation}\label{comportamento}
	H+3w_1+\frac{w^2_1}{H}= O(H).
	\end{equation}
	Moreover, we know from Proposition \ref{ordersize} that $|1-\nu_n|=O(n^{-\frac{\tau-3}{\tau-1}})$ and hence $|1-\nu_n|H=O(n^{-2\frac{\tau-4}{3(\tau-1)}})=o(1)$. Furthermore, again from Proposition \ref{ordersize}, we know that
	\begin{equation}\label{barabba}
	\Big|\mathbb{E}((W^*_n)^2)-\frac{\mathbb{E}(W^3)}{\mathbb{E}(W)}\Big|=O(n^{-\frac{\tau-4}{\tau-1}}).
	\end{equation}
	Taking $H'$ in such a way that $H^2=o(H')$ (whence $b_H/H'=o(1)$) we thus see that the expression within round brackets in (\ref{ddddddaaaaaaa}) is strictly positive for all sufficiently large $n$. Therefore we can write (when $n$ is large enough)
	\begin{equation}\label{daje}
	\mathbb{E}(\gamma)\leq \frac{H+3w_1+\frac{w^2_1}{H}}{\mathbb{E}((W^*_n)^2)}\Big(1-O(n^{-2\frac{\tau-4}{3(\tau-1)}})-O\Big(\frac{b_H}{H'}\Big)\Big)^{-1}.
	\end{equation}
	Using \eqref{comportamento} and \eqref{barabba}, together with the inequality $\sum_{k=2}^{\infty}x^k\leq 2x^2$ (which is valid for all $x\in [0,1/2]$), it is not difficult to show that the quantity which appears on the right-hand side of \eqref{daje} is bounded from above (for all large enough $n$) by $4H\mathbb{E}(W)/\mathbb{E}(W^3)$. Therefore we obtain
	\begin{equation}\label{boundmeangamma}
	\mathbb{E}(\gamma)\leq 4H\frac{\mathbb{E}(W)}{\mathbb{E}(W^3)}.
	\end{equation}
	Thus using Lemma \ref{upperbound} we can bound
	\begin{align}\label{zizza}
	\mathbb{P}(|\CC(V_n)|>k)\leq \frac{1}{H}+\frac{|1-\nu_n|\mathbb{E}(\gamma)}{H}+\frac{\mathbb{E}(\gamma)}{k}.
	\end{align}
	Since $|1-\nu_n|H\ll 1$, substituting into \eqref{zizza}  the bound for $\mathbb{E}(\gamma)$ stated in \eqref{boundmeangamma} we obtain that (for all large enough $n$)
	\begin{equation}\label{44Prime}
	\mathbb{P}(|\CC(V_n)|>k)\leq \frac{2}{H}+4k^{-1}H\frac{\mathbb{E}(W)}{\mathbb{E}(W^3)}.
	\end{equation}
	Finally, denoting by $N_k\coloneqq \sum_{i=1}^{n}\mathbbm{1}_{\{|\mathcal{C}(i)|>k\}}$ the number of vertices contained in components formed by more than $k$ nodes, using Markov's inequality we obtain
	\begin{equation*}
	\Prob(|\CC_{\max}|>k)\leq \Prob(N_k>k)\leq \frac{\E(N_k)}{k}\leq \frac{n\Prob(|\CC(V_n)|>k)}{k}
	\leq 2\frac{n}{Hk}+4\frac{nH}{k^2}\frac{\mathbb{E}(W)}{\mathbb{E}(W^3)}.
	\end{equation*}
	Taking $k=\lfloor A n^{2/3} \rfloor$ and recalling the definition of $H$ we see that there is a finite constant $c_1>0$ (which depends on $c_F$ and $\tau$) such that 
	\begin{equation}
	\Prob(|\CC_{\max}|>k)\leq \frac{c_1}{A}
	\end{equation}
	for all large enough $n$, which concludes the proof since $A n^{2/3}\geq \lfloor An^{2/3}\rfloor=k$.

	\subsubsection{Proof of Theorem \ref{teo2tauin34} (case $\tau \in (3,4)$) -- $|\CC_{\max}|$ is unlikely to be larger than $An^{\frac{\tau-2}{\tau-1}}$}
	\textcolor{black}{Before starting with the actual proof we need a simple result, whose proof is postponed to Subsection \ref{subsectionauxiliary}, which guarantees that, when $F$ (i.e. the distribution function determining the vertex weights) satisfies (\ref{COND}) for some $\tau\in (3,4)$ and $c_F>0$, then $1-\nu_n>0$ for all sufficiently large $n$. 
		\begin{lem}\label{lempositive}
			Suppose that there exist $\tau> 3$ and $c_F>0$ such that (\ref{COND}) holds. Then, for all sufficiently large $n$, we have that $1-\nu_n>\frac{n^{-\frac{\tau-3}{\tau-1}}}{\tau-1} >0$.
		\end{lem}
	}
	We can now proceed with the proof of Theorem \ref{teo2tauin34}. Since $1-\nu_n> 0$ for all large enough $n$, it follows from Lemma \ref{lem45} that
	\begin{equation}\label{bubbe}
	\mathbb{E}(\gamma)\mathbb{E}((W^*_n)^2)\Big(1-\frac{1-\nu_n}{\mathbb{E}((W^*_n)^2)}H-\frac{2b_H}{\mathbb{E}((W^*_n)^2)H'}\Big)\leq H+3w_1+\frac{w^2_1}{H}.
	\end{equation}
	Let $H\coloneqq \lfloor \delta n^{1/(\tau-1)}\rfloor $, where $\delta\in (0,1)$ is some constant that we specify later. From Proposition \ref{ordersize} we know that, for all large enough $n$, $|\nu_n-1|\leq C_1n^{-\frac{\tau-3}{\tau-1}}$ for some positive constant $C_1$ which depends on $c_F$ and $\tau$. Since $1-\nu_n>0$ for all large enough $n$, we obtain 
	\begin{equation}\label{sPi}
	(1-\nu_n)H=|1-\nu_n|H\leq \delta C_1n^{\frac{4-\tau}{\tau-1}}.
	\end{equation}
	Next, we bound the second moment of $W^*_n$. Since
	\begin{align*}
	\nonumber\mathbb{E}((W^*_n)^2) =l^{-1}_n \sum_{j=1}^{n}w^3_j=\frac{(c_Fn)^{3/(\tau-1)}}{l_n}\sum_{j=1}^{n}j^{-3/(\tau-1)},
	\end{align*}
	then, using the fact that
	\begin{equation*}
	\int_{1}^{b+1}x^{-r} \mathrm{d}x\leq\sum_{i=1}^{b}i^{-r}\leq 1+\int_{1}^{b}x^{-r} \mathrm{d}x \hspace{0.4cm}(r>0),
	\end{equation*}
	we can write $0.5(\tau-1)(\tau-2)^{-1}c^{1/(\tau-1)}_F n\leq l_n\leq 2(\tau-1)(\tau-2)^{-1}c^{1/(\tau-1)}_F n$ and
	\begin{equation}\label{46Prime}
	C_2n^{\frac{4-\tau}{\tau-1}}\leq\mathbb{E}((W^*_n)^2)\leq C_3n^{\frac{4-\tau}{\tau-1}}
	\end{equation}
	for all large $n$ and for some finite constants $0<C_2<C_3$ which depend on $c_F$ and $\tau$. 
	(We remark that here we do not need an upper bound for $\mathbb{E}((W^*_n)^2)$, but only a lower bound. However, later on we will need the upper bound too; we decided to state both bounds here for referencing purposes). Therefore, combining (\ref{sPi}) and (\ref{46Prime}) we obtain
	\[\frac{1-\nu_n}{\mathbb{E}((W^*_n)^2)}H\leq \delta \frac{C_1}{C_2}\]
	and the quantity on the right-hand side of the last inequality can be made at most $1/2$ by choosing $\delta\leq C_2/2C_1$. Since the term $2b_H/\mathbb{E}((W^*_n)^2)H'$ can be made as small as we like by choosing a proper value of $H'$ (in particular, a value $H'>n^{\frac{\tau-2}{\tau-1}}$ would do the job) we conclude that, for all large enough $n$ (and taking $\delta\leq  C_2/2C_1$)
	\begin{align}\label{malle}
	\mathbb{E}((W^*_n)^2)\Big(1-\frac{1-\nu_n}{\mathbb{E}((W^*_n)^2)}H-\frac{2b_H}{\mathbb{E}((W^*_n)^2)H'}\Big)\geq C_4n^{\frac{4-\tau}{\tau-1}},
	\end{align}
	for some positive constant $C_4$ which depends on $c_F$ and $\tau$. Since $w_1=O(n^{1/(\tau-1)})$, we can combine (\ref{bubbe}) and (\ref{malle}) together to obtain that $\mathbb{E}(\gamma)\leq Cn^{\frac{\tau-3}{\tau-1}}$ for some finite constant $C>0$ which depends $c_F$ and $\tau$. Using Lemma \ref{upperbound} together with our previous estimate on $\mathbb{E}(\gamma)$ we arrive at
	\begin{equation}\label{48}
	\mathbb{P}(|\CC(V_n)|>k)\leq \frac{1-\mathbb{E}(\gamma)(1-\nu_n)}{H}+\frac{\mathbb{E}(\gamma)}{k}\leq \frac{1}{H}+\frac{\mathbb{E}(\gamma)}{k}\leq \frac{1}{H}+C\frac{n^{\frac{\tau-3}{\tau-1}}}{k}.
	\end{equation}
	Proceeding as in the proof of Theorem \ref{teo1taugreater4} we obtain
	\begin{equation*}
	\mathbb{P}(|\CC_{\max}|>k)\leq \frac{n}{kH}+\frac{n Cn^{\frac{\tau-3}{\tau-1}}}{k^2}=\frac{n}{kH}+\frac{Cn^{2\frac{\tau-2}{\tau-1}}}{ k^2}.
	\end{equation*}
	Taking $k=\lfloor An^{\frac{\tau-2}{\tau-1}}\rfloor$ and recalling the definition of $H$ we finally conclude that
	\begin{equation*}
	\mathbb{P}(|\CC_{\max}|>k)\leq \frac{c_3}{A}
	\end{equation*}
	for some finite constant $c_3>0$ that depends on $c_F$ and $\tau$. Since $A n^{\frac{\tau-2}{\tau-1}}\geq k$, the desired result follows.

	\subsubsection{Proof of Theorem \ref{teoexp} -- The exponential upper bound}
	Here we wish to improve the polynomial upper bounds stated in Theorems \ref{teo1taugreater4} and \ref{teo2tauin34}. 
	
	\textcolor{black}{In what follows we work under the assumption that (\ref{COND}) is satisfied for some $\tau>3$ and $c_F>0$. }
	
	Let $Q_t\coloneqq 1+\sum_{i=1}^{t}(|\mathcal{M}_i|-1)$ for $t\leq \beta$, with $\beta$ defined as $\gamma$ in Lemma \ref{upperbound} but using $Q_t$ instead $S_t$, and set $\beta^*=\beta \wedge k$, \textcolor{black}{where $k$ will be chosen later (and its actual value will depend on the range of $tau$).} Note that $|\mathcal{A}^{BP}_{t}|\leq Q_t$ for $t\leq \beta$, and so in particular $|\mathcal{A}^{BP}_{\beta*}|\leq Q_{\beta*}$.
	
	Define, for $t\leq T\ll n,$
	\begin{equation}
	Z_t\coloneqq \sum_{j=1}^{t}\big(|\widetilde{\mathcal{M}}_{\beta^*+j}|-1\big) 
	\end{equation}
	and observe that, if $|\AAA^{BP}_{\beta^*+j}|>0$ for all \textcolor{black}{$1\leq j\leq t$}, then 
	\begin{align*}
	Z_j\overset{def}{=}\sum_{i=1}^{j}\big(|\widetilde{\mathcal{M}}_{\beta^*+i}|-1\big)=\sum_{h=\beta^*+1}^{\beta^*+j}\big(|\widetilde{\mathcal{M}}_{h}|-1\big)=|\AAA^{BP}_{\beta^*+j}|-|\AAA^{BP}_{\beta^*}|.
	\end{align*}
	Thus, setting $P\coloneqq \mathbb{P}(Q_{\beta^*}>0)$ and using the law of total expectation we obtain
	\begin{multline}\label{49Prime}
	\mathbb{P}\big(|\AAA^{BP}_{\beta^*+j}|>0 \hspace{0.15cm}\forall j\in [t]|Q_{\beta^*}>0\big)\\
	=P^{-1}\mathbb{E}\Big(\mathbb{E}_{Q_{\beta^*}}\big(\mathbbm{1}_{\{|\AAA^{BP}_{\beta^*+j}|>0 \hspace{0.15cm}\forall j\in [t]\}}\mathbbm{1}_{\{Q_{\beta^*}>0\}}\big)\Big)\\
	\leq P^{-1}\mathbb{E}\big(\mathbbm{1}_{\{Q_{\beta^*}>0\}}\mathbb{P}_{Q_{\beta^*}}(Z_t>-|\AAA^{BP}_{\beta^*}|)\big),
	\end{multline}
	where we denote by $\mathbb{P}_{Q_{\beta^*}}(\cdot)$ the probability measure $\mathbb{P}(\cdot |Q_{\beta^*}>0)$ and we write $\mathbb{E}_{Q_{\beta^*}}(\cdot)$ for the expectation operator with respect to $\mathbb{P}_{Q_{\beta^*}}(\cdot)$. Since $|\mathcal{A}^{BP}_{\beta^*}|\leq Q_{\beta^*}$, using Markov's inequality we obtain (for any $r>0$)
	\begin{align}\label{hjh}
	\nonumber P^{-1}\mathbb{E}\big(\mathbbm{1}_{\{Q_{\beta^*}>0\}}\mathbb{P}_{Q_{\beta^*}}(Z_t>-|\AAA^{BP}_{\beta^*}|)\big)&\leq P^{-1}\mathbb{E}\big(\mathbbm{1}_{\{Q_{\beta^*}>0\}}\mathbb{P}_{Q_{\beta^*}}(Z_t>-Q_{\beta^*})\big)\\
	\nonumber &\leq P^{-1}\mathbb{E}\big(\mathbbm{1}_{\{Q_{\beta^*}>0\}}\mathbb{P}_{Q_{\beta^*}}(e^{rZ_t}>e^{-rQ_{\beta^*}})\big)\\
	\nonumber &\leq P^{-1}\mathbb{E}\big(\mathbbm{1}_{\{Q_{\beta^*}>0\}}e^{rQ_{\beta^*}}\mathbb{E}_{Q_{\beta^*}}(e^{rZ_t})\big)\\
	&= \textcolor{black}{\mathbb{E}_{Q_{\beta^*}}(e^{rZ_t}) \mathbb{E}_{Q_{\beta^*}}(e^{rQ_{\beta^*}}). }
	\end{align}
	With the next lemma \textcolor{black}{(whose proof is given in Subsection \ref{subsectionauxiliary})} we establish an upper bound for the first expectation in \eqref{hjh}, i.e. the $\mathbb{E}_{Q_{\beta*}}(\cdot)$-expectation of $e^{rZ_t}$.
	\begin{lem}\label{LemUmbi}
		Let $r\leq 1/w_1$ and suppose that $t\leq T\ll n$. Then, for all large enough $n$, we have that
		\begin{multline}\label{50Prime}
		\mathbb{E}_{Q_{\beta^*}}(e^{rZ_t})=\mathbb{E}_{Q_{\beta^*}}\big(e^{r\sum_{j=1}^{t}(|\widetilde{\mathcal{M}}_{\beta*+j}-1)}\big)\\
		\leq 2\exp\Big\{r^2t\mathbb{E}((W^*_n)^2)(1+c'/w_1)-r\nu_n\frac{\tau-2}{\tau-1}\frac{t^2}{2n}\Big\}\cdot\\
		\cdot \exp\Big\{rt(\nu_n-1)+ \bar{c}\frac{T^3}{n^2w_1}+3r^2t\nu_n \Big\}
		\end{multline}
		for some finite constants $c',\bar{c}>0$.
	\end{lem}
	Consequently, taking $r=1/w_1\asymp n^{-1/(\tau-1)}\ll 1$ throughout \textcolor{black}{and recalling the definition of $\mathbb{E}_{Q_{\beta^*}}(\cdot)$}, we see that the expression in \eqref{hjh} is at most
	\begin{multline}\label{jgj}
	2\exp\Big\{r^2t\mathbb{E}((W^*_n)^2)(1+c'/w_1)-r\nu_n\frac{\tau-2}{\tau-1}\frac{t^2}{2n}\Big\}\cdot\\
	\cdot \exp\Big\{rt(\nu_n-1)+ \bar{c}\big(\frac{T^3}{n^2w_1}\big)+3r^2t\nu_n \Big\}\mathbb{E}\big(e^{rQ_{\beta^*}}|Q_{\beta^*}>0\big).
	\end{multline}
	To bound the second expectation in \eqref{hjh} we argue as follows. Lemma \ref{OurLemma6} states that the (conditional) law of the \textit{overshoot} $S_{\gamma}-H$, given $S_{\gamma}\geq H$, is stochastically dominated by the $\text{Poisson}(w_1)$ distribution. This result also holds for the overshoot $Q_{\beta}-H$. (To see this, it is enough to follow the proof of Lemma \ref{OurLemma6} using $Q_{\beta}=Q_{\beta-1}+|\mathcal{M}_{\beta}|-1$ in place of $S_{\gamma}=S_{\gamma-1}+|\Upsilon_{\gamma}|-1$ together with the fact that $|\mathcal{M}_1|\preccurlyeq \Upsilon_1$.) Clearly Corollary \ref{OurCorollary6} holds too, so that
	\begin{align}\label{opo}
	\nonumber\mathbb{E}\big(e^{rQ_{\beta^*}}|Q_{\beta}>0,\beta \in [k] \big)&=\mathbb{E}\big(e^{rQ_{\beta}}|Q_{\beta}>0,\beta \in [k] \big)\\
	\nonumber&=e^{rH}\mathbb{E}\big(e^{r(Q_{\beta}-H)}|Q_{\beta}>0,\beta \in [k] \big)\\
	\nonumber&=e^{rH}\mathbb{E}\big(e^{r(Q_{\beta}-H)}|Q_{\beta}\geq H,\beta \in [k] \big)\\
	\nonumber&\leq e^{rH}\mathbb{E}\big(e^{rY_{w_1}}\big)\\
	&= e^{rH}e^{w_1(e^r-1)}.
	\end{align}
	Since $r\ll 1$ we can bound  $e^r-1\leq r+r^2$ and hence the expression in \eqref{opo} is at most $e^{w_1r+w_1r^2+rH}$. Since $\{\beta>k\}$ and $\{Q_{\beta}>0,\beta \in [k]\}$ are disjoint events whose union is $\{Q_{\beta^*}>0\}$, and because the (conditional) expected value of $e^{rQ_{\beta^*}}$ given $\beta > k$ is at most $e^{rH}$ \textcolor{black}{(since in this case $Q_{\beta^*}=Q_{k}<H$)}, we conclude that (as $r=1/w_1$)
	\begin{align}\label{51Prime}
	\mathbb{E}\big(e^{rQ_{\beta^*}}|Q_{\beta^*}>0\big)\leq e^{rH}+e^{w_1r+w_1r^2+rH}\leq (1+e^{1+1/w_1})e^{rH}\leq 9e^{rH}
	\end{align}
	provided $n$ is large enough. Therefore, combining \eqref{49Prime}, \eqref{50Prime} and \eqref{51Prime}  we arrive at
	\begin{multline}\label{51Prime2}
	\mathbb{P}\big(|\AAA^{BP}_{\beta^*+j}|>0 \hspace{0.15cm}\forall j\in [t]|Q_{\beta^*}>0\big)\\
	\leq 10\exp\Big\{r^2t\mathbb{E}((W^*_n)^2)(1+c'/w_1)-r\nu_n\frac{\tau-2}{\tau-1}\frac{t^2}{2n}\Big\}\cdot\\
	\cdot \exp\Big\{rt(\nu_n-1)+ \bar{c}\big(\frac{T^3}{n^2w_1}\big)+3r^2t\nu_n \Big\} e^{rH}.
	\end{multline}
	Observe that, taking $H\asymp n^{1/3}$ for $\tau>4$ and $H\asymp n^{1/(\tau-1)}$ for $\tau\in (3,4)$ we see that $e^{rH}=O(n^{\frac{\tau-4}{3(\tau-1)}})$ when $\tau>4$, whereas $e^{rH}=O(1)$ if $3<\tau<4$. Moreover, $T\sim An^{2/3}$ for $\tau>4$ while $T\sim An^{\frac{\tau-2}{\tau-1}}$ for $3<\tau<4$ and consequently we can write
	$$
	\bar{c}\Big(\frac{T^3}{n^2w_1}\Big)=
	\begin{cases}
	O(1), & \text{ if } \tau>4 \text{ and } A=O\big(n^{\frac{(\tau-4)\wedge 1}{3(\tau-1)}}\big)\\
	O(1), & \text{ if } 3<\tau<4 \text{ and } A=O\big(n^{\frac{5-\tau}{3(\tau-1)}}\big)\\
	\end{cases}
	$$
	Using these estimates in \eqref{51Prime2} and recalling that $r=1/w_1\asymp n^{-1/(\tau-1)}$ we see that, for all large enough $n$, 
	\begin{multline*}\label{kklll}
	\nonumber \mathbb{P}\big(|\AAA^{BP}_{\beta^*+j}|>0 \hspace{0.15cm}\forall j\in [t]|Q_{\beta^*}>0\big)\leq\\
	C'\exp\Big\{r^2t\Big(\mathbb{E}((W^*_n)^2)+3\nu_n + c\frac{\mathbb{E}((W^*_n)^2)}{n^{1/(\tau-1)}}\Big)-\frac{\tau-2}{\tau-1}\frac{t^2r\nu_n}{2n}+rt(\nu_n-1)+rH\Big\}
	\end{multline*}
	for some finite constant $C'>0$. Now, setting
	\begin{equation*}
	f(r,t)\coloneqq
	r^2t\Big(\mathbb{E}((W^*_n)^2)+3\nu_n + c\frac{\mathbb{E}((W^*_n)^2)}{n^{1/(\tau-1)}}\Big)-r\nu_n\frac{\tau-2}{\tau-1}\frac{t^2}{2n}+rt(\nu_n-1)+rH,
	\end{equation*}
	we see that the derivative (with respect to $r$) of $f(r,t)$ vanishes if, and only if, 
	\begin{align*}
	r=r_0\coloneqq \frac{\nu_n\frac{\tau-2}{\tau-1}\frac{t^2}{2n}-t(\nu_n-1)-H}{2t\Big(\mathbb{E}((W^*_n)^2)+3\nu_n + c\frac{\mathbb{E}((W^*_n)^2)}{n^{1/(\tau-1)}}\Big)}.
	\end{align*}
	Since the second derivative of $f(r,t)$ with respect to $r$ is always positive, the value $r_0$ indeed minimizes $f(r,t)$. 
	
	Therefore
	\begin{align*}
	f(r,t)\geq f(r_0,t)=-\frac{\Big(\nu_n\frac{\tau-2}{\tau-1}\frac{t^2}{2n}-t(\nu_n-1)-H\Big)^2}{4t\Big(\mathbb{E}((W^*_n)^2)+3\nu_n + c\frac{\mathbb{E}((W^*_n)^2)}{n^{1/(\tau-1)}}\Big)} 
	\end{align*}
	Let $t=T-k$. When $\tau>4$ take $k=H^2$ where $H=\lfloor{n^{1/3}\rfloor}$ and $T=\lfloor{An^{2/3}\rfloor}$, with $A=O\big(n^{\frac{(\tau-4)\wedge 1}{3(\tau-1)}}\big)$. Using (\ref{nun1}), (\ref{Ewnstart}) and since $n^{2/3}(A-2)<t<n^{2/3}(A-1/2)$ and $\nu_n\frac{\tau-2}{\tau-1}\frac{t^2}{2n}-t(\nu_n-1)-H>0$ for $A>2$,
	then we obtain
	\begin{multline*}
	\frac{\Big(\nu_n\frac{\tau-2}{\tau-1}\frac{t^2}{2n}-t(\nu_n-1)-H\Big)^2}{4t\Big(\mathbb{E}((W^*_n)^2)+3\nu_n + c\frac{\mathbb{E}((W^*_n)^2)}{n^{1/(\tau-1)}}\Big)}\\ 
	\geq \frac{\Big((1-C_1n^{-\frac{\tau-3}{\tau-1}})(\frac{\tau-2}{\tau-1})\frac{n^{4/3}(A-2)^2}{2n}-n^{2/3}(A-1/2)C_1n^{-\frac{\tau-3}{\tau-1}}-n^{1/3}\Big)^2}{4n^{2/3}(A-1/2)\Big(\frac{\mathbb{E}((W^3))}{\mathbb{E}((W))}+4\Big)}
	\end{multline*}
	\textcolor{black}{which, for $n$ large enough,  is greater than 
		\[\frac{\Big(\big(\frac{\tau-2}{\tau-1}\big)\frac{(A-2)^2}{2}-2\Big)^2}{4(A-1/2)\Big(\frac{\mathbb{E}(W^3)}{\mathbb{E}(W)}+4\Big)}.\]}
	Since 
	\begin{equation*}
	\Big(\Big(\frac{\tau-2}{\tau-1}\Big)\frac{(A-2)^2}{2}-2\Big)^2> \Big(\Big(\frac{\tau-2}{\tau-1}\Big)\frac{A}{2}\Big(\frac{A}{2}-2\Big)\Big)^2
	\end{equation*}
	for $A> (4+\sqrt{32})/2$, and $(A/2-2)/(A-1/2)>1/4$ for $A>8$, when $n$ is large enough and $A>8$ we obtain
	\begin{equation*}
	e^{f(r_0,T-k)}\leq \exp\bigg\{-\frac{\big(\frac{\tau-2}{\tau-1}\big)^2 A^2(A-4)}{128\big(\frac{\mathbb{E}(W^3)}{\mathbb{E}(W)}+4\big)}\bigg\}.
	\end{equation*}
	When $\tau\in (3,4)$, let $k= H^{\tau-2}$ where $H=\lfloor{ n^{1/(\tau-1)}\rfloor}$ and $T=\lfloor{ An^{\frac{\tau-2}{\tau-1}}\rfloor}$, with $A=O\big(n^{\frac{5-\tau}{3(\tau-1)}}\big)$. Using \eqref{nun1}, \eqref{46Prime}, \eqref{lempositive} and since $n^{\frac{\tau-2}{\tau-1}}(A-2)<t<n^{\frac{\tau-2}{\tau-1}}(A-1/2)$, for $A>2$ we obtain, expanding the squared term at the numerator, 
	\begin{multline*}
	\frac{\Big(\nu_n\frac{\tau-2}{\tau-1}\frac{t^2}{2n}-t(\nu_n-1)-H\Big)^2}{4t\Big(\mathbb{E}((W^*_n)^2)+3\nu_n + c\frac{\mathbb{E}((W^*_n)^2)}{n^{1/(\tau-1)}}\Big)} \\
	\geq \frac{\nu_n^2\big(\frac{\tau-2}{\tau-1}\big)^2\frac{t^4}{4n^2}+\nu_n\big(\frac{\tau-2}{\tau-1}\big)\frac{t^3}{n}(1-\nu_n)+H^2+t^2(1-\nu_n)^2-2t(1-\nu_n)H-2\nu_n\big(\frac{\tau-2}{\tau-1}\big)\frac{t^2}{2n}H}{4n^{\frac{\tau-2}{\tau-1}}(A-1/2)(C_3n^{\frac{4-\tau}{\tau-1}}+4)},
	\end{multline*}
	which for $n$ large enough and $A>8$ is greater than $\frac{A(A-2\tau)}{4(\tau-1)^2(A-1/2)(C_3+1)}$. Consequently we obtain, for  $8<A=O\big(n^{\frac{5-\tau}{3(\tau-1)}}\big)$ and for all large enough $n$, 
	\begin{equation*}
	\frac{\Big(\nu_n\frac{\tau-2}{\tau-1}\frac{t^2}{2n}-t(\nu_n-1)-H\Big)^2}{4t\Big(\mathbb{E}((W^*_n)^2)+3\nu_n + c\frac{\mathbb{E}((W^*_n)^2)}{n^{1/(\tau-1)}}\Big)}>\frac{A-2\tau}{4(\tau-1)^2(C_3+1)},
	\end{equation*}
	Thus we can bound 
	\begin{equation*}
	e^{f(r_0,T-k)}\leq \exp\Big\{-\frac{A-2\tau}{4(\tau-1)^2(C_3+1)}\Big\}.
	\end{equation*}
	Therefore, for $\tau>4$ we arrive at 
	\begin{equation}\label{Ptau4}
	\mathbb{P}\big(|\AAA^{BP}_{\beta^*+j}|>0 \hspace{0.15cm}\forall j\in [T-k]|Q_{\beta^*}>0\big)\\
	\leq C'\exp\Big\{-\frac{\big(\frac{\tau-2}{\tau-1}\big)^2A^2(A-4)}{128\big(\frac{\mathbb{E}(W^3)}{\mathbb{E}(W)}+4\big)}\Big\},
	\end{equation}
	whereas for $3<\tau<4$ we have
	\begin{align}\label{Ptau34}
	\mathbb{P}\Big(|\AAA^{BP}_{\beta^*+j}|>0 
	\hspace{0.15cm}\forall j\in [T-k]|Q_{\beta^*}>0\Big)
	\leq C''\exp\Big\{-\frac{A-2\tau}{4(\tau-1)^2(C_3+1)}\Big\}.
	\end{align}
	Note that, since $|\AAA^{BP}_{t}|\leq  Q_{t}$ for \textcolor{black}{ $0<t\leq \beta^*$, when $T>k$ ($\geq \beta^*$)} we obtain
	\begin{multline*}
	\mathbb{P}\big(|\CC(V_n)|>T)\leq \mathbb{P}(|\AAA^{BP}_{\beta^*}|>0, |\AAA^{BP}_{\beta^*+j}|>0, \forall j\in [T-k]\big)\\
	\leq \mathbb{P}(Q_{\beta^*}>0, |\AAA^{BP}_{\beta^*+j}|>0, \forall j\in [T-k]\big)\\
	=\mathbb{P}(Q_{\beta^*}>0)\mathbb{P}\big(|\AAA^{BP}_{\beta^*+j}|>0 \hspace{0.15cm}\forall j\in [T-k]|Q_{\beta^*}>0\big),
	\end{multline*}
	\textcolor{black}{and the second probability on the right-hand side of the last expression is bounded from above in (\ref{Ptau4}) and (\ref{Ptau34}) for the cases $\tau>4$ and $\tau\in(3,4)$, respectively.}
	
	To complete the proof, we thus need an upper bound for $\mathbb{P}(Q_{\beta^*}>0)$. To this end, we use Lemma \ref{upperbound}, in which we have established an upper bound for $\mathbb{P}(S_{\gamma^*}>0)$ with $S_t$ being a random walk with independent increments having distribution $\text{Poi}(w_{M_i})-1$, where the random variables $M_i$ are independent with distribution $M$ as in \eqref{M}. In particular, we now construct such a process $S_t$ \textit{starting} from the random variables $|\mathcal{M}_i|$, in such a way that $\mathbb{P}(Q_{\beta^*}>0)\leq \mathbb{P}(S_{\gamma^*}>0)$. To this end, recall that $|\mathcal{M}_1|$ is a random variable with the mixed $\text{Poi}(w_{J_0})$ distribution \textcolor{black}{(where $J_0$ is uniformly distributed on $[n]$)}. Thanks to our discussion prior to the statement of Lemma \ref{upperbound} we know that, if $Y_1$ is a random variable with the $\text{Poi}(w_{M})$ distribution, then there is a coupling $(D_1,\widehat{Y}_1)$ of $|\mathcal{M}_1|$ and $Y_1$ such that $D_1\leq \widehat{Y}_1$ almost surely. For $i\geq 2$, let $\widehat{|\mathcal{M}|}_i$ be independent copies of the $|\mathcal{M}_i|$, defined on the \textit{same} probability space where both $D_1$ and $\widehat{Y}_1$ are defined. Set $\widehat{Q}_0\coloneqq 1$ and $\widehat{Q}_i\coloneqq \widehat{Q}_{i-1}+\widehat{D}_i-1$ for $i\geq 1$, where $\widehat{D}_1\coloneqq D_1$ and $\widehat{D}_i\coloneqq \widehat{|\mathcal{M}|}_i$ for $i\geq 2$. Moreover, we set $S_0\coloneqq 1$ and $S_i\coloneqq S_{i-1}+\Upsilon_i-1$ for $i\geq 1$, where $\Upsilon_1\coloneqq \widehat{Y}_1$ and $\Upsilon_i\coloneqq \widehat{|\mathcal{M}|}_i$ for $i\geq 2$. Define $\gamma$ to be the first time $t\geq 1$ at which either $S_t=0$ or $S_t\geq H$, and similarly define $\widehat{\beta}$ to be the first time $t\geq 1$ at which either $\widehat{Q}_t=0$ or $\widehat{Q}_t\geq H$. Let $\gamma^*\coloneqq \gamma \wedge k$ and $\widehat{\beta}^*\coloneqq \widehat{\beta}\wedge k$. Note that, almost surely, $\widehat{Q}_t\leq S_t$ for every $t\in \mathbb{N}_0$, because $S_1=\Upsilon_1=\widehat{Y}_1\geq D_1=\widehat{D}_1=\widehat{Q}_1$ almost surely and for $t\geq 2$ we have that
	\[S_t=\Upsilon_1+\sum_{i=2}^{t}(\Upsilon_i-1)\geq \widehat{D}_1+ \sum_{i=2}^{t}(\Upsilon_i-1)=\widehat{D}_1+ \sum_{i=2}^{t}(\widehat{D}_i-1)=\widehat{Q}_t.\]
	Moreover, $Q_{\beta^*}$ has the same distribution as $\widehat{Q}_{\widehat{\beta}^*}$. Therefore we can write
	\begin{align}\label{starlab}
	\nonumber\mathbb{P}(Q_{\beta^*}>0)=\mathbb{P}(\widehat{Q}_{\widehat{\beta}^*}>0)&=\mathbb{P}(\widehat{Q}_{k}>0, \widehat{\beta}>k)+\mathbb{P}(\widehat{Q}_{\widehat{\beta}}>0, \widehat{\beta}\leq k)\\
	&=\mathbb{P}(\widehat{\beta}>k)+\mathbb{P}(\widehat{Q}_{\widehat{\beta}}>0, \widehat{\beta}\leq k).
	\end{align}
	We claim that 
	\begin{equation}\label{claimq1}
	\mathbb{P}(\widehat{\beta}>k)=\mathbb{P}(\widehat{\beta}>k,S_{\gamma^*}>0).
	\end{equation}
	To see this, suppose that $\widehat{\beta}>k$ and $S_{\gamma^*}=0$. Since $\widehat{\beta}>k$, then $\widehat{Q}_t\in (0,H)\cap \mathbb{N}$ for all $t\leq k$ and so in particular $S_t>0$ for all $t\leq k$. If $\gamma\leq k$ then we obtain $0<S_{\gamma}=S_{\gamma^*}=0$, a contradiction. Similarly, if $\gamma>k$ then we obtain $0<S_k=S_{\gamma^*}=0$, also a contradiction. Therefore  $\mathbb{P}(\widehat{\beta}>k,S_{\gamma^*}=0)=0$,  proving the claim. Next we claim that
	\begin{equation}\label{claimq2}
	\mathbb{P}(\widehat{Q}_{\widehat{\beta}}>0, \widehat{\beta}\leq k)=\mathbb{P}(\widehat{Q}_{\widehat{\beta}}>0, \widehat{\beta}\leq k, S_{\gamma^*}>0).
	\end{equation}
	To see this, suppose that \textcolor{black}{$\widehat{Q}_{\widehat{\beta}}>0, \widehat{\beta}\leq k$ and $S_{\gamma^*}=0$. Since $\widehat{Q}_{\widehat{\beta}}>0$, by definition of $\widehat{\beta}$ we must have $\widehat{Q}_{\widehat{\beta}}\geq H$}. This implies that $S_{\widehat{\beta}}\geq H$, and hence $\gamma\leq \widehat{\beta}$. Also, since $\widehat{Q}_t>0$ for all $t<\widehat{\beta}$ and $\widehat{Q}_{\widehat{\beta}}>0$, we must have $S_{\gamma}\geq \widehat{Q}_{\gamma}>0$ and so (by definition of $\gamma$) we get $S_{\gamma}\geq H$. Since $\gamma\leq \widehat{\beta}$ and $\widehat{\beta}\leq k$, it follows that $\gamma \leq k$. Therefore $S_{\gamma^*}=S_{\gamma}\geq H$, which contradicts the initial assumption that $S_{\gamma^*}=0$, thus proving the claim. It follows from (\ref{starlab}), \eqref{claimq1} and \eqref{claimq2} that
	\begin{multline*}
	\mathbb{P}(Q_{\beta^*}>0)=\mathbb{P}(\widehat{\beta}>k,S_{\gamma^*}>0)+\mathbb{P}(\widehat{Q}_{\widehat{\beta}}>0, \widehat{\beta}\leq k, S_{\gamma^*}>0)\\
	\leq \mathbb{P}(\widehat{\beta}>k,S_{\gamma^*}>0)+\mathbb{P}(\widehat{\beta}\leq k, S_{\gamma^*}>0)=\mathbb{P}(S_{\gamma^*}>0).
	\end{multline*}
	By Lemma \ref{upperbound}, \eqref{44Prime} and
	taking $k=H^2$ when $\tau>4$ we obtain that $\mathbb{P}(Q_{\beta^*}>0)= O(1/H)$. Consequently there is a finite constant $C>0$ \textcolor{black}{(which depends on $c_F$ and $\tau$} such that
	
	\begin{equation*}
	\mathbb{P}(|\CC_{\max}|>T)\leq \frac{nC}{TH}\exp\bigg\{-\frac{\big(\frac{\tau-2}{\tau-1}\big)^2 A^2(A-4)}{128\big(\frac{\mathbb{E}(W^3)}{\mathbb{E}(W)}+4\big)}\bigg\}\\
	\leq \frac{c_5}{A}\exp\bigg\{-\frac{\big(\frac{\tau-2}{\tau-1}\big)^2 A^2(A-4)}{128\big(\frac{\mathbb{E}(W^3)}{\mathbb{E}(W)}+4\big)}\bigg\}
	\end{equation*}
	for some finite constant $c_5>0$ \textcolor{black}{(which depends on $c_F$ and $\tau$)}. \textcolor{black}{Similarly, by 
		Lemma \ref{upperbound}, \eqref{48}, and 
		taking $k= H^{\tau-2}\asymp n^{\frac{\tau-2}{\tau-1}}$ when $\tau\in (3,4)$ we obtain $\mathbb{P}(Q_{\beta^*}>0)= O(1/H)$. Consequently there are finite constants $c_6,c_7>0$ which depend on $c_F$ and $\tau$ such that
		\begin{equation*}
		\mathbb{P}(|\CC_{\max}|>T)\leq \frac{c_6}{A}\exp\Big\{- \frac{A-2\tau}{4(\tau-1)^2(C_3+1)}\Big\},
		\end{equation*}
		completing the proof of the theorem.}

	\subsection{Proof of Theorems \ref{teo1taugreater4} and \ref{teo2tauin34} --  the probability of small maximal components}
	To prove the results of this section we use \textbf{Alg.2} and \textbf{Alg.2.BP} to establish the bound for the case $\tau>4$, whereas we use \textbf{Alg.3} and \textbf{Alg.3.BP} to handle the case $\tau\in (3,4)$. That is, when $\tau>4$ we start the exploration process from a node (resp. mark) selected with probability proportional to its weight, i.e. $V_n=i$ (resp. $J_0=i$) with probability $w_i/l_n$ for $i\in [n]$, whereas when $\tau\in (3,4)$ we (deterministically) start the procedure from vertex $V_n=1$ (resp. mark $J_0=1$). In a moment we will explain why it is actually useful to start the exploration processes in different ways for the two regimes $\tau>4$ and $\tau\in (3,4)$.
	
	Recall that our goal here is to show that, when $\tau>4$, a largest component is unlikely to contain less than $n^{2/3}/A$ vertices; similarly we prove that, if $3<\tau<4$, then a largest component is unlikely to contain less than $n^{\frac{\tau-2}{\tau-1}}/A$ nodes.
	
	Let $T_2=T_2(n)\in \mathbb{N}$. By Proposition \ref{newnewprop}, independently of the way we choose the vertex from which to start the exploration process, we can write
	\begin{equation}\label{impidforlower}
	\mathbb{P}\left(|\CC_{\max}|<T_2\right)=\mathbb{P}\left(\tau_j-\tau_{j-1}<T_2\hspace{0.15cm}\forall j\right),
	\end{equation}
	where $\tau_0=0$ and $(\tau_j:j\geq 1)$ are the ordered times \textcolor{black}{(prior to the termination of the procedure)} at which the set of active marks becomes empty.

	Let $T_1=T_1(n)\in \mathbb{N}$. Following  \cite{NP,NCHperc}, 
	the idea is to prove that, with sufficiently high probability, 
	the process $|\AAA^{BP}_t|$ reaches some (high) level $h=h(n)$ before time $T_1$ and then it remains positive for at least $T_2$ steps.
	
	Intuitively, if we want this strategy to be successful, we need $h$ to be substantially larger than $\sqrt{T_2}$, so that for \textcolor{black}{the process of active marks} 
	(which, in some sense, it behaves like a mean-zero, integer-valued random walk) started at height $h$ it becomes indeed likely to remain positive for $T_2$ steps. It is at this stage that it becomes useful to work with the two procedures \textbf{Alg.2.BP} and \textbf{Alg.3.BP} for the cases $\tau>4$ and $\tau\in (3,4)$, respectively. 
	
	Indeed, let us start by considering the case $\tau\in (3,4)$. In this regime, the mark $J_0$ from which we start the exploration process is (deterministically) chosen to be vertex $1$. By Lemma \ref{maxw}, $w_i=(nc_F/i)^{1/(\tau-1)}$ and hence, since $l_n\asymp n$, it follows that at the end of the first step in the procedure we expect to have approximately 
	\begin{equation*}
	\sum_{j=2}^{n}\Big(1-e^{-w_1 w_j/l_n}\Big)\approx w_1\sum_{j=2}^{n}\frac{w_j}{l_n}= w_1\Big(1 - \frac{w_1}{l_n}\Big)\asymp n^{\frac{1}{\tau-1}}(1-o(1))\sim n^{\frac{1}{\tau-1}} 
	\end{equation*}
	active marks (which correspond to the nodes directly connected to $V_n=1$). In this regime (i.e. when $\tau\in (3,4)$) we have that
	\[n^{\frac{1}{\tau-1}}\gg n^{\frac{\tau-2}{2(\tau-1)}}=\sqrt{n^{\frac{\tau-2}{\tau-1}}}\]
	and therefore, taking $h=n^{\frac{1}{\tau-1}}$ and $T_2=\lceil{n^{\frac{\tau-2}{\tau-1}}/A\rceil}$, we do have that $h$ is much larger than $\sqrt{T_2}$. This means that, after \textit{one step} only, our process already reached a height which is sufficient to guarantee that it will remain positive for $T_2$ steps. 
	
	In other words, taking $T_1=2$ and $h,T_2$ as above, we can indeed show that our process reaches level $h$ at time $t=1<T_1$ and then remains positive for $T_2$ steps. This approach, however, can't work for the case $\tau>4$ (unless we make unpleasant assumptions on $A$ of the type $A=A(n)\geq n^{\frac{2(\tau-4)}{3(\tau-1)}}$). Indeed, when $\tau>4$, since $n^{\frac{1}{\tau-1}}\ll n^{1/3}=\sqrt{n^{2/3}}$, it becomes unlikely that our process remains positive for $T_2= \lceil n^{2/3}/A \rceil$ steps after having reached height $h\asymp n^{\frac{1}{\tau-1}}$ in one step. 
	
	In other words, when $\tau>4$ it is not sufficient to analyse the component of vertex $1$ to draw conclusions on $|\CC_{\max}|$; to do this, we need to explore the components of multiple vertices and, in this setting, it is convenient that the nodes from which we start exploring new components are selected from the set of unexplored nodes with probability proportional to their weights. We then need to perform two separate analysis for the cases $\tau>4$ and $\tau\in (3,4)$.

	In particular, following our previous discussion, we let $h\in \mathbb{N}$ be some positive integer and bound, for the case $\tau>4$,
	\begin{align}\label{twoprob}
	\nonumber\mathbb{P}(|\mathcal{C}_{\max}|<T_2)&=\mathbb{P}\left(\tau_j-\tau_{j-1}<T_2\hspace{0.15cm}\forall j\right)\\
	\nonumber&\leq \mathbb{P}\left(|\AAA^{BP}_t|<h\hspace{0.15cm}\forall t\in [T_1-1]\right)\\
	&+\mathbb{P}\left(\tau_j-\tau_{j-1}<T_2\hspace{0.15cm}\forall j,\exists t\in [T_1-1]:|\AAA^{BP}_t|\geq h\right),
	\end{align}
	while for the case $\tau \in (3,4)$ we write
	\begin{align}\label{twoprob34}
	\nonumber\mathbb{P}(|\CC_{\max}|<T_2)&\leq \mathbb{P}(|\mathcal{C}_1|<T_2)=\mathbb{P}(\tau_1<T_2)\\
	&\leq \mathbb{P}(|\mathcal{A}^{BP}_1|< h)+\mathbb{P}(\tau_1<T_2, |\mathcal{A}^{BP}_1|\geq h),
	\end{align}
	where we recall that $\mathcal{C}_1$ is the first component to be explored in $NR_n(\textbf{w})$ (the component of node $1$).

	The probabilities on the right-hand sides of \eqref{twoprob} and \eqref{twoprob34} are bounded in separate ways, specifically by means of Propositions \ref{prop1alltau} and \ref{prop2alltau} below for the case $\tau>4$, while using Propositions \ref{prop134} and \ref{prop234} when $\tau \in (3,4)$. Before stating such results, however, we recall a few useful estimates from previous sections that we use again here.
	
	From Proposition \ref{ordersize} we know that, whenever $n$ is sufficiently large,
	\begin{equation}\label{oneminusnun}
	|1-\nu_n|\leq C_1 n^{-\frac{\tau-3}{\tau-1}}
	\end{equation}
	for some finite constant $C_1>0$ which depends on $c_F$ and $\tau$; moreover, if $\tau>4$, we also have that
	\begin{equation}\label{secmomsizebiasedtau4}
	\left|\mathbb{E}((W^*_n)^2) - \frac{\mathbb{E}(W^3)}{\mathbb{E}(W)} \right|=O\big(n^{-\frac{\tau-4}{\tau-1}}\big).
	\end{equation}
	We also recall from \eqref{46Prime} that, when $\tau\in (3,4)$, we can bound
	\begin{equation}\label{secmomsizebiasedtau34}
	C_2n^{\frac{4-\tau}{\tau-1}}\leq \mathbb{E}((W^*_n)^2) \leq C_3n^{\frac{4-\tau}{\tau-1}}
	\end{equation}
	for all large enough $n$, with $C_2,C_3>0$ two finite constants which depend on $c_F$ and $\tau$. 
	
	\begin{prop}\label{prop1alltau}
		Let $\tau>4$ and set $T_1\coloneqq \lfloor n^{2/3}/A^{1/4}\rfloor$, $h\coloneqq \lfloor n^{1/3}/A^{1/4}\rfloor$. Then, for all large enough $n$, we have that
		\begin{equation*} 
		\mathbb{P}\big(|\AAA^{BP}_t|<h\hspace{0.15cm}\forall t\in [T_1-1]\big)\leq \frac{C'}{A^{1/4}},
		\end{equation*}
		where $C'>0$ is some finite constant which depends on $c_F$ and $\tau$.
	\end{prop}

	\begin{prop}\label{prop2alltau}
		Let $\tau>4$ and set $T_1\coloneqq \lfloor n^{2/3}/A^{1/4} \rfloor$, $h\coloneqq \lfloor n^{1/3}/A^{1/4}\rfloor$ and $T_2\coloneqq \lceil n^{2/3}/A \rceil$. Then, for all large enough $n$, we have that
		\begin{equation*}
		\mathbb{P}\big(\tau_j-\tau_{j-1}<T_2\hspace{0.15cm}\forall j,\exists t\in [T_1-1]:|\AAA^{BP}_t|\geq h\big)\leq \frac{C}{A^{1/2}},
		\end{equation*}
		where $C>0$ is some finite constant which depends on $c_F$ and $\tau$.
	\end{prop}

	\begin{prop}\label{prop134}
		\textcolor{black}{Let $\tau\in (3,4)$ and set $h\coloneqq \lfloor \delta n^{1/(\tau-1)}\rfloor$, with $\delta>0$ some sufficiently small (fixed) quantity. Let $T_2=\lceil n^{\frac{\tau-2}{\tau-1}}/A\rceil$. Then, for all large enough $n$, we have that
			\begin{equation*}
			\mathbb{P}(|\mathcal{A}^{BP}_1|< h)\leq \frac{C'}{n^{\frac{\tau-2}{\tau-1}}},
			\end{equation*}
			for some finite constant $C'=C'(\delta)>0$ which also depends on $c_F$ and $\tau$.}
	\end{prop}

	\begin{prop}\label{prop234}
		Let $\tau\in (3,4)$ and set $h\coloneqq \lfloor \delta n^{1/(\tau-1)}\rfloor$, with  $\delta>0$ some sufficiently small (fixed) quantity. Let $T_2=\lceil n^{\frac{\tau-2}{\tau-1}}/A\rceil$. Then, for all large enough $n$, we have that
		\begin{equation*}
		\mathbb{P}(\tau_1<T_2, |\mathcal{A}^{BP}_1|\geq h)\leq \frac{C}{A},
		\end{equation*}
		for some finite constant $C=C(\delta)>0$ which also depends on $c_F$ and $\tau$.
	\end{prop}

	We are now in the position to establish the upper bounds for the probability of observing unusually small components stated in Theorems \ref{teo1taugreater4} and \ref{teo2tauin34}. Indeed, when $\tau>4$ it follows from \eqref{twoprob} together with Propositions \ref{prop1alltau} and \ref{prop2alltau} that
	\begin{equation*}
	\mathbb{P}(|\mathcal{C}_{\max}|<n^{2/3}/A)\leq \mathbb{P}(|\mathcal{C}_{\max}|<T_2)
	\leq \frac{c_2}{A^{1/4}}
	\end{equation*}
	for some constant $c_2>0$ which depends on $c_F$ and $\tau$. On the other hand, when $3<\tau<4$, it follows from \eqref{twoprob34} together with Propositions \ref{prop134} and \ref{prop234} that
	\begin{equation*}
	\mathbb{P}(|\mathcal{C}_{\max}|<n^{\frac{\tau-2}{\tau-1}}/A)\leq \mathbb{P}(|\mathcal{C}_{\max}|<T_2)
	\leq c_4(A^{-1}\vee n^{-\frac{\tau-2}{\tau-1}}),
	\end{equation*}
	for some constant $c_4>0$ which depends on $c_F$ and $\tau$. Note that, without loss of generality, we can assume that $A<n^{\frac{\tau-2}{\tau-1}}$ \textcolor{black}{(otherwise the probability on the left-hand side of the last display would be zero)} and hence the expression on the right hand side of the last inequality is $c_4/A$, as required.
	
	Before starting with the actual proofs of the above propositions, we establish a technical lemma which we will need throughout. We remind the reader that $X_{v_t}=|\mathcal{M}_t|$ is the number of children of node $v_t$ in the exploration of the branching process trees.
	
	\begin{lem}\label{implemmaforlowerbound} 
		Let 
		$T=T(n)=o(n)$ and set 
		\begin{equation}\label{defincr}
		I_i\coloneqq |\mathcal{M}_i|-|\widetilde{\mathcal{M}}_i|=\sum_{l=1}^{X_{v_i}}\mathbbm{1}_{\big\{J^{v_i}_l\in (\mathcal{A}^{BP}_{i-1}\cup \{m^{BP}_i\})\cup \mathcal{E}^{BP}_{i-1}\cup \mathcal{L}^{v_i}_{l-1}\big\}},
		\end{equation}
		for $1\leq i\leq T$. Then
		\begin{equation}
		\mathbb{E}(I_1)=O(w^2_1/n) \text{ when }\tau\in (3,4) \text{ and }\mathbb{E}(I_1)=O(1/n) \text{ when }\tau>4.
		\end{equation}
		Moreover, if $i\geq 2$ and $\tau>3$, we have 
		\begin{equation*}
		\mathbb{E}(I_i)=O\Big(\frac{w_1\vee i \vee \mathbb{E}((W^*_n)^2)}{n} \Big).
		\end{equation*}
	\end{lem}
	
	\begin{proof} 
		Suppose first that $i=1$ and $\tau\in (3,4)$. Recall that, in this regime, we start exploring a branching process tree whose root carries the deterministic mark $J_0=1$. Consequently, the random variable $X_{v_1}$ (that corresponds to the random number of children of the root node) has the Poisson distribution with parameter $w_1$. Now observe that, since $\mathcal{A}^{BP}_0=\{1\}$ and $\mathcal{E}^{BP}_0=\emptyset$, we have
		\begin{multline}\label{case134}
		I_1=\sum_{l=1}^{X_{v_1}}\mathbbm{1}_{\big\{J^{v_1}_l\in \{1\}\cup \mathcal{L}^{v_1}_{l-1}\big\}}=\sum_{k\geq 1}\mathbbm{1}_{\{X_{v_1}=k\}}\sum_{l=1}^{k}\mathbbm{1}_{\big\{J^{v_1}_l\in \{1\}\cup \mathcal{L}^{v_1}_{l-1}\big\}}\\
		\leq \sum_{k\geq 1}\mathbbm{1}_{\{X_{v_1}=k\}}\sum_{l=1}^{k}\mathbbm{1}_{\big\{J^{v_1}_l\in \{1\}\big\}}+\sum_{k\geq 1}\mathbbm{1}_{\{X_{v_1}=k\}}\sum_{l=1}^{k}\mathbbm{1}_{\big\{J^{v_1}_l\in \mathcal{L}^{v_1}_{l-1}\big\}}.
		\end{multline}
		Recalling that the $J^{v_1}_l$ are i.i.d. with distribution $M$ given in (\ref{M}) we have 
		\[\mathbb{P}(J^{v_1}_l\in \mathcal{L}^{v_1}_{l-1})=\sum_{j=1}^{n}\frac{w_j}{l_n}(1-\mathbb{P}(J^{v_1}_r\neq j \text{ }\forall r\leq l-1))=\sum_{j=1}^{n}\frac{w_j}{l_n}\Big[1-\big(1-\frac{w_j}{l_n}\big)^{l-1}\Big] \]
		and hence, after taking expectation on both sides of (\ref{case134}) we obtain (since $\mathbb{P}(J^{v_1}_l\in \{1\})=w_1l^{-1}_n$)
		\begin{equation*}
		\mathbb{E}(I_1)\leq \frac{w_1}{l_n}\sum_{k\geq 1}\mathbb{P}(X_{v_1}=k)k+\sum_{k\geq 1}\mathbb{P}(X_{v_1}=k)\sum_{l=1}^{k}\sum_{j=1}^{n}\frac{w_j}{l_n}\Big(1-\big(1-\frac{w_j}{l_n}\big)^{l-1}\Big).
		\end{equation*}
		Since $(1-w_j/l_n)^{l-1}\geq 1-(l-1)w_j/l_n$ (and $X_{v_1}$ has the Poisson law with mean $w_1$) a short computation shows that
		\begin{equation*}
		\mathbb{E}(I_1)\leq \frac{w^2_1}{l_n}+\frac{\nu_n}{2l_n}(w^2_1+w_1)=O(w^2_1/l_n)=O(w^2_1/n),
		\end{equation*}
		where for the last identity we have used that $l_n\asymp n$. The previous expression establishes the lemma for the case $i=1,\tau\in (3,4)$. Hence, in the remainder of the proof, we assume that either $i\geq 2$ and $\tau>3$, or $i=1$ and $\tau>4$. Let's consider the former case first; that is, we let $i\geq 2$ and $\tau>3$. Denote by $\mathcal{F}^{BP}_{i}$ the $\sigma$-algebra collecting all the information revealed by the exploration process of the branching process trees until the end of step $i$, with $\mathcal{F}^{BP}_0$ being the trivial $\sigma$-field. Note that, by definition of $I_i$, we have
		\begin{equation}\label{dec}
		I_i	\leq \sum_{l=1}^{X_{v_i}}\mathbbm{1}_{\big\{J^{v_i}_l\in (\mathcal{A}^{BP}_{i-1}\cup \{m^{BP}_i\})\cup \mathcal{E}^{BP}_{i-1}\big\}}+\sum_{l=1}^{X_{v_i}}\mathbbm{1}_{\big\{J^{v_i}_l\in  \mathcal{L}^{v_i}_{l-1}\big\}}.
		\end{equation}
		We start focusing on the first sum appearing on the right-hand side of \eqref{dec} and subsequently we take into account the second sum. From Section \ref{comparisonsection} we know that, if $|\mathcal{A}^{BP}_{i-1}|\geq 1$, then $m^{BP}_i\in \mathcal{A}^{BP}_{i-1}$ and hence 
		\[(\mathcal{A}^{BP}_{i-1}\cup \{m^{BP}_i\})\cup \mathcal{E}^{BP}_{i-1}=\mathcal{A}^{BP}_{i-1}\cup \mathcal{E}^{BP}_{i-1},\]
		while if $|\mathcal{A}^{BP}_{i-1}|=0$ we have
		\[(\mathcal{A}^{BP}_{i-1}\cup \{m^{BP}_i\})\cup \mathcal{E}^{BP}_{i-1}=\{m^{BP}_i\}\cup \mathcal{E}^{BP}_{i-1}.\]
		Therefore, when $|\mathcal{A}^{BP}_{i-1}|\geq 1$ we have 
		\begin{equation}\label{case1}
		\sum_{l=1}^{X_{v_i}}\mathbbm{1}_{\big\{J^{v_i}_l\in (\mathcal{A}^{BP}_{i-1}\cup \{m^{BP}_i\})\cup \mathcal{E}^{BP}_{i-1}\big\}}=\sum_{l=1}^{X_{v_i}}\mathbbm{1}_{\big\{J^{v_i}_l\in \mathcal{A}^{BP}_{i-1} \cup \mathcal{E}^{BP}_{i-1}\big\}},
		\end{equation}
		whereas when $|\mathcal{A}^{BP}_{i-1}|=0$ we have 
		\begin{equation}\label{case0}
		\sum_{l=1}^{X_{v_i}}\mathbbm{1}_{\big\{J^{v_i}_l\in (\mathcal{A}^{BP}_{i-1}\cup \{m^{BP}_i\})\cup \mathcal{E}^{BP}_{i-1}\big\}}=\sum_{l=1}^{X_{v_i}}\mathbbm{1}_{\big\{J^{v_i}_l\in \{m^{BP}_i\}\cup \mathcal{E}^{BP}_{i-1}\big\}}.
		\end{equation}
		Thus, when $|\mathcal{A}^{BP}_{i-1}|\geq 1$, we obtain 
		\begin{align*}
		\mathbb{E}\bigg(\sum_{l=1}^{X_{v_i}}\mathbbm{1}_{\big\{J^{v_i}_l\in (\mathcal{A}^{BP}_{i-1}\cup \{m^{BP}_i\})\cup \mathcal{E}^{BP}_{i-1}\big\}}|\mathcal{F}^{BP}_{i-1}\bigg)&=\sum_{k\geq 1}^{}\sum_{l=1}^{k}\mathbb{P}(X_{v_i}=k,J^{v_i}_l\in \mathcal{A}^{BP}_{i-1} \cup \mathcal{E}^{BP}_{i-1}|\mathcal{F}^{BP}_{i-1})\\
		&=\mathbb{E}(\text{Poi}(w_M))\sum_{m=1}^{n}\frac{w_m}{l_n}\mathbb{1}_{\{m\in \mathcal{A}^{BP}_{i-1} \cup \mathcal{E}^{BP}_{i-1}\}}\\
		&=\nu_n\sum_{m=1}^{n}\frac{w_m}{l_n}\mathbb{1}_{\{m\in \mathcal{A}^{BP}_{i-1} \cup \mathcal{E}^{BP}_{i-1}\}}.
		\end{align*}
		For $|\mathcal{A}^{BP}_{i-1}|=0$, observe that
		\begin{equation}\label{seccase}
		\mathbb{E}\bigg( \sum_{l=1}^{X_{v_i}}\mathbbm{1}_{\big\{J^{v_i}_l\in (\mathcal{A}^{BP}_{i-1}\cup \{m^{BP}_i\})\cup \mathcal{E}^{BP}_{i-1}\big\}}|\mathcal{F}^{BP}_{i-1}\bigg)=\sum_{k\geq 1}\sum_{l=1}^{k}\mathbb{P}(X_{v_i}=k,J^{v_i}_l\in \{m^{BP}_{i}\} \cup \mathcal{E}^{BP}_{i-1}|\mathcal{F}^{BP}_{i-1}).
		\end{equation}
		Since $|\mathcal{A}^{BP}_{i-1}|=0$, the random mark $m^{BP}_i$ equals $m\in [n]\setminus \mathcal{E}^{BP}_{i-1}$ with probability $w_m/l'_n(i)$, where we recall that $l'_n(i)= l_n-\sum_{j\in \mathcal{E}^{BP}_{i-1}}w_j$. We have
		\begin{align*}
		\nonumber l'_n(i)=l_n-\sum_{j \in \mathcal{E}^{BP}_{i-1}}^{}w_j \geq l_n-\sum_{j=1}^{i-1}w_j&\geq l_n-\sum_{j=1}^{T-1}w_j\\
		\nonumber&\geq l_n-O\Big(n^{\frac{1}{\tau-1}}\int_{1}^{T}x^{-\frac{1}{\tau-1}}dx\Big)\\
		&=l_n\Big[1-O\Big( \Big(\frac{T}{n}\Big)^{\frac{\tau-2}{\tau-1}}\Big)\Big]=l_n(1-o(1))\asymp n,
		\end{align*}
		where we have used that (by assumption) $i\leq T\ll n$. Moreover, given $m^{BP}_i=m$, we know that $X_{v_i}$ has the $\text{Poi}(w_m)$ distribution and (since $\sum_{m=1}^{n}w^2_m/l_n=\nu_n=O(1)$) a short computation reveals that the expression on the right-hand side of (\ref{seccase}) is at most
		\begin{multline}
		\frac{1}{l'_n(i)}\sum_{m\in [n]\setminus \mathcal{E}^{BP}_{i-1}}^{}\frac{w^3_m}{l_n}+\Big(\sum_{m\in [n]\setminus \mathcal{E}^{BP}_{i-1}}^{}\frac{w^2_m}{l_n}\Big)\Big(\sum_{j=1}^{n}\frac{w_j}{l_n}\mathbb{1}_{\{j\in \mathcal{E}^{BP}_{i-1}\}}\Big)\\
		= O\Big(\frac{\mathbb{E}((W^*_n)^2)}{n}\Big)+O\Big(\sum_{j=1}^{n}\frac{w_j}{l_n}\mathbb{1}_{\{j\in \mathcal{E}^{BP}_{i-1}\}}\Big).
		\end{multline}
		Next we bound the second sum on the right-hand side of (\ref{dec}). Proceeding in a similar way as before (when we considered the case $i=1,\tau>3$), we arrive at
		\begin{align*}
		\mathbb{E}\bigg(\sum_{l=1}^{X_{v_i}}\mathbbm{1}_{\{J^{v_i}_l\in  \mathcal{L}^{v_i}_{l-1}\}}|\mathcal{F}^{BP}_{i-1}\bigg)=O\Big(\frac{\mathbb{E}[X^2_{v_i}|\mathcal{F}^{BP}_{i-1}]}{n}\Big).
		\end{align*}
		The expectation which appears at the numerator in the ratio on the right-hand side of the last expression  equals $\mathbb{E}((W^*_n)^2)$ if $|\mathcal{A}^{BP}_{i-1}|\geq 1$, whereas it is $O(\mathbb{E}((W^*_n)^2))$ if $|\mathcal{A}^{BP}_{i-1}|=0$. 
		All in all, we have shown that when $i\geq 2$ and $\tau>3$, if $|\mathcal{A}^{BP}_{i-1}|\geq 1$ then $\mathbb{E}(I_i|\mathcal{F}^{BP}_{i-1})$ is at most
		\begin{equation*}
		\nu_n\sum_{m=1}^{n}\frac{w_m}{l_n}\mathbb{1}_{\{m\in \mathcal{A}^{BP}_{i-1} \cup \mathcal{E}^{BP}_{i-1}\}}+O\Big(\frac{\mathbb{E}((W^*_n)^2)}{n}\Big)=O\bigg(\sum_{m=1}^{n}\frac{w_m}{l_n}\mathbb{1}_{\{m\in \mathcal{A}^{BP}_{i-1} \cup \mathcal{E}^{BP}_{i-1}\}} \vee \frac{\mathbb{E}((W^*_n)^2)}{n}\bigg);
		\end{equation*}
		similarly, if $|\mathcal{A}^{BP}_{i-1}|=0$ then $\mathbb{E}(I_i|\mathcal{F}^{BP}_{i-1})$ is at most
		\begin{equation*}
		O\bigg(\sum_{j=1}^{n}\frac{w_j}{l_n}\mathbb{1}_{\{j\in \mathcal{E}^{BP}_{i-1}\}}\bigg)+O\Big(\frac{\mathbb{E}((W^*_n)^2)}{n}\Big)=O\bigg(\sum_{m=1}^{n}\frac{w_m}{l_n}\mathbb{1}_{\{m\in \mathcal{E}^{BP}_{i-1}\}} \vee \frac{\mathbb{E}((W^*_n)^2)}{n}\bigg).
		\end{equation*}
		Thus we arrive at
		\begin{equation*}
		\mathbb{E}(I_i)=\mathbb{E}(\mathbb{E}(I_i|\mathcal{F}^{BP}_{i-1}))=O\bigg(\sum_{m=1}^{n}\frac{w_m}{l_n}\mathbb{P}(m\in \mathcal{A}^{BP}_{i-1}\cup \mathcal{E}^{BP}_{i-1}) \vee \frac{\mathbb{E}((W^*_n)^2)}{n}\bigg).
		\end{equation*}
		There remains to bound (from above) the probability that $m$ is in $\mathcal{A}^{BP}_{i-1}\cup \mathcal{E}^{BP}_{i-1}$, where $i\leq T$. There are three ways for $m$ to be either active or explored at the end of step $i-1$ in the exploration of the branching process trees. Indeed: 
		\begin{itemize}
			\item [(a)] either at a step $s\leq i-1$ one of the marks $(J^{v_s}_l:l\in [X_{v_s}])$ assigned to the children of $v_s$ was equal to $m$;
			\item [(b)] or at some step $s\leq i-1$ we had $|\mathcal{A}^{BP}_{s-1}|=0$ and $m^{BP}_s=m$ (meaning that the root of the new tree started at time $s$ received mark $m$);
			\item [(c)] or $J_0=m$.
		\end{itemize}
		The event in (c) has probability $w_m/n$ when we use \textbf{Alg.BP.2} to explore the branching process trees (which occurs when $\tau>4$), whereas when we use \textbf{Alg.BP.3} (which occurs when $3<\tau<4$) it has probability one if $m=1$ (and probability $0$ otherwise). Thus we obtain
		\begin{equation*}
		\sum_{m=1}^{n}\frac{w_m}{l_n}\mathbb{P}(J_0=m)=\frac{w_1}{l_n} \text{ when }\tau\in (3,4) \text{ and }\sum_{m=1}^{n}\frac{w_m}{l_n}\mathbb{P}(J_0=m)=\frac{\nu_n}{l_n}\text{ when } \tau>4,
		\end{equation*}
		so that (since $\nu_n\ll w_1$) we have $\sum_{m=1}^{n}(w_m/l_n)\mathbb{P}(J_0=m)\leq w_1/l_n$ whenever $\tau>3$. Consider the event in (a) next. By a union bound we obtain that
		\begin{align*}
		\mathbb{P}(\exists s\leq i-1:J^{v_s}_l=m \text{ for some }l\in [X_{v_s}])&\leq \sum_{s=1}^{i-1}\mathbb{P}(J^{v_s}_l=m \text{ for some }l\in [X_{v_s}])\\
		&=\sum_{s=1}^{i-1}\sum_{k\geq 1}^{}\mathbb{P}(X_{v_s}=k)\mathbb{P}(J^{v_s}_l=m \text{ for some }l\in [k])\\
		&=\sum_{s=1}^{i-1}\sum_{k\geq 1}^{}\mathbb{P}(X_{v_s}=k)(1-(1-w_m/l_n)^k)\\
		&\leq \frac{w_m}{l_n}\sum_{s=1}^{i-1}\mathbb{E}(X_{v_s})=O\left(\frac{w_m i}{l_n}\right),
		\end{align*}
		whence
		\begin{equation*}
		\sum_{m=1}^{n}\frac{w_m}{l_n}\mathbb{P}(\exists s\leq i-1:J^{v_s}_l=m \text{ for some }l\in [X_{v_s}])=O(i/l_n).
		\end{equation*}
		There remains to consider the event in (b). In this case, a union bound and our previous estimate of $l'_n(i)$ yields 
		\[\mathbb{P}(\exists s\leq i-1:|\mathcal{A}^{BP}_{s-1}|=0,m^{BP}_s=m)=O(w_mi/l_n)\]
		so that also in this case we have
		\begin{equation*}
		\sum_{m=1}^{n}\frac{w_m}{l_n}\mathbb{P}(\exists s\leq i-1:|\mathcal{A}^{BP}_{s-1}|=0,m^{BP}_s=m)=O(i/l_n).
		\end{equation*}
		Consequently we arrive at
		\[\sum_{m=1}^{n}\frac{w_m}{l_n}\mathbb{P}(m\in \mathcal{A}^{BP}_{i-1}\cup \mathcal{E}^{BP}_{i-1})=O\left(\frac{w_1\vee i}{n}\right).\]
		Finally, let's consider the case where $i=1$ and $\tau>4$. In this case it is not difficult to see that $\mathbb{E}(I_1)=O(1/n)$, completing the proof of the lemma.\\
	\end{proof}

	\begin{center}
		\textbf{Proof of Proposition \ref{prop1alltau}}
	\end{center}
	Let us start by introducing an auxiliary process $W_t$ defined as follows. We set $W_0\coloneqq |\mathcal{A}^{BP}_{0}|=1$ and define recursively $W_t$ in the following way:
	\begin{itemize}
		\item $W_t=W_{t-1}+|\mathcal{M}_{t}|-1$, if $|\mathcal{A}^{BP}_{t-1}|\geq 1$;
		\item $W_t=W_{t-1}+|\mathcal{M}_t|$ if $|\mathcal{A}^{BP}_{t-1}|=0$.
	\end{itemize}
	Note that $W_t\geq |\mathcal{A}^{BP}_{t}|$ at all times $t<T_1$ \textcolor{black}{(and so in particular $W_t\geq 0$ for all $t$)}. Indeed, for $t=1$ we see that $W_1=|\mathcal{M}_1|\geq |\widetilde{\mathcal{M}}_1|=|\mathcal{A}^{BP}_{1}|$. If the inequality is true for some $1\leq t<T_1-1$, we see that, if $|\mathcal{A}^{BP}_{t}|\geq 1$, then $W_{t+1}\geq W_t+|\widetilde{\mathcal{M}}_{t+1}|-1\geq |\mathcal{A}^{BP}_{t}|+|\widetilde{\mathcal{M}}_{t+1}|-1=|\mathcal{A}^{BP}_{t+1}|$. Similarly, when $|\mathcal{A}^{BP}_{t}|=0$ we obtain that \textcolor{black}{$W_{t+1}=W_t+|\mathcal{M}_{t+1}|\geq |\mathcal{A}^{BP}_{t}|+|\widetilde{\mathcal{M}}_{t+1}|=|\widetilde{\mathcal{M}}_{t+1}|=|\mathcal{A}^{BP}_{t+1}|$, establishing the claim.}
	Define the (bounded) stopping time
	\begin{equation}\label{tau2h}
	\tilde\tau_{h} \coloneqq \left\{ \begin{aligned}
	& \inf \{t< T_1:W_t\geq 2h\}, &&  \text{ if }\{t< T_1:W_t\geq 2h\}\neq \emptyset,\\
	& T_1, && \text{ if }\{t< T_1:W_t\geq 2h\}=\emptyset.\\
	\end{aligned}
	\right.
	\end{equation}
	Note that
	\begin{align*}\label{623Prime}
	\mathbb{P}(|\mathcal{A}_t^{BP}|<h\text{ }\forall &t\in [T_1-1])\\
	&\leq \mathbb{P}(W_t<2h\text{ }\forall t\in [T_1-1])+\mathbb{P}(\exists t<T_1:W_t-|\mathcal{A}^{BP}_{t}|\geq h)\\
	&=\mathbb{P}(\tilde\tau_{h}=T_1)+\mathbb{P}(\exists t<T_1:W_t-|\mathcal{A}^{BP}_{t}|\geq h).
	\end{align*}
	We claim that, for $t<T_1$,
	\begin{equation}\label{keyineq}
	W_t-|\mathcal{A}^{BP}_{t}|\leq \sum_{i=1}^{t}I_i,
	\end{equation}
	where we recall that (by definition) $I_i=|\mathcal{M}_i|-|\widetilde{\mathcal{M}}_i|$. We establish (\ref{keyineq}) by induction on $t$. For $t=1$ we have $W_1-|\mathcal{A}^{BP}_{1}|=I_1$ and so the inequality is trivially true. Next, suppose that it holds for $1\leq t<T_1-1$. Note that, if $|\mathcal{A}^{BP}_t|=0$, then using the inductive hypothesis we obtain $W_{t+1}-|\mathcal{A}^{BP}_{t+1}|=W_t+I_{t+1}=W_t-|\mathcal{A}^{BP}_{t}|+I_{t+1}\leq \sum_{i=1}^{t+1}I_i$. Similarly, if $|\mathcal{A}^{BP}_{t}|\geq 1$, by the inductive hypothesis we obtain $W_{t+1}-|\mathcal{A}^{BP}_{t+1}|=W_t-|\mathcal{A}^{BP}_{t}|+I_{t+1}\leq \sum_{i=1}^{t+1}I_i$. This establishes (\ref{keyineq}). It follows that
	\begin{multline*}
	\mathbb{P}(\exists t<T_1:W_t-|\mathcal{A}^{BP}_{t}|\geq h)\leq \mathbb{P}(\exists t<T_1:\sum_{i=1}^{t}I_i\geq h)\\
	\leq \mathbb{P}(\sum_{i=1}^{T_1-1}I_i\geq h)\leq \frac{\sum_{i=1}^{T_1-1}\mathbb{E}[I_i]}{h},
	\end{multline*}
	where the second inequality exploits the fact that each $I_i$ is non-negative. By Lemma \ref{implemmaforlowerbound} we know that (since $l_n\asymp n$) $\mathbb{E}[I_i]=O(i/n)$ for every $i<T_1$. Therefore the ratio on the right-hand side of the last expression is $O(T^2_1/(hn)$ and hence we arrive at
	\begin{equation}\label{gobackhere}
	\mathbb{P}(|\mathcal{A}_t^{BP}|<h\text{ }\forall t\in [T_1-1])\leq \mathbb{P}(\tilde\tau_{h}=T_1)+O\Big(\frac{T^2_1}{hn}\Big).
	\end{equation}
	By Markov's inequality, $\mathbb{P}(\tilde\tau_{h}=T_1)\leq T^{-1}_1\mathbb{E}(\tilde\tau_{h})$ and so to complete the proof we need to bound the expected value of the stopping time $\tilde\tau_{h}$. This is achieved by means of Theorem \ref{OST} in the following way. First of all, note that each $W_t$ is $\mathcal{F}^{BP}_t$-measurable and, if $|\mathcal{A}^{BP}_{t-1}|\geq 1$, then 
	\begin{equation}\label{computationsecmom}
	\mathbb{E}(W^2_t\mid \mathcal{F}^{BP}_{t-1}) = W^2_{t-1}+2W_{t-1}\mathbb{E}(|\mathcal{M}_t|-1\mid \mathcal{F}^{BP}_{t-1})+\mathbb{E}((|\mathcal{M}_t|-1)^2\mid \mathcal{F}^{BP}_{t-1}),
	\end{equation}
	while if $|\mathcal{A}^{BP}_{t-1}|=0$ then (since $W_{t-1}|\mathcal{M}_t|\geq 0$)
	\begin{equation}\label{imptrick}
	\mathbb{E}(W^2_t\mid \mathcal{F}^{BP}_{t-1})\geq W^2_{t-1}+\mathbb{E}(|\mathcal{M}_t|^2 \mid \mathcal{F}^{BP}_{t-1}).
	\end{equation}
	\textcolor{black}{We wish to estimate (from below) the expressions on the right-hand side of (\ref{computationsecmom}) and (\ref{imptrick}). First of all, note that if} $|\mathcal{A}^{BP}_{t-1}|\geq 1$, then $|\mathcal{M}_t|=X_{v_t}$ has a mixed Poisson distribution with random parameter $M$ (given in (\ref{M})) and hence $\mathbb{E}(|\mathcal{M}_t|\mid \mathcal{F}^{BP}_{t-1})=\nu_n,\mathbb{E}(|\mathcal{M}_t|^2\mid |\mathcal{F}^{BP}_{t-1}|)=\nu_n+\mathbb{E}((W^*_n)^2)$, so that 
	\[-C_1n^{-\frac{\tau-3}{\tau-1}}\leq \nu_n-1=\mathbb{E}(|\mathcal{M}_t|-1\mid \mathcal{F}^{BP}_{t-1})\leq C_1n^{-\frac{\tau-3}{\tau-1}}\]
	and, since $\mathbb{E}((W^*_n)^2)\geq \mathbb{E}(W^3)/\mathbb{E}(W)-o(1)$,
	\[\mathbb{E}(|\mathcal{M}^2_t|\mid \mathcal{F}^{BP}_{t-1})\geq 1+\mathbb{E}(W^3)/\mathbb{E}(W)-o(1).\]
	On the other hand, if $|\mathcal{A}^{BP}_{t-1}|=0$ then $|\mathcal{M}_t|$ has a mixed Poisson distribution with random parameter $m^{BP}_t$ which takes values $m\in [n]\setminus \mathcal{E}^{BP}_{t-1}$ with probability $w_m/l'_n(t-1)\asymp w_m/l_n$ and in this case we have (since $l'_n(i)\leq l_n$ for every $i$ and $t<T_1$)
	\begin{align}\label{comb}
	\nonumber\mathbb{E}(|\mathcal{M}^2_t|\mid \mathcal{F}^{BP}_{t-1})=\sum_{m\in [n]\setminus \mathcal{E}^{BP}_{t-1}}^{}\frac{w_m}{l'_n(t)}(w_m+w^2_m)&\geq \nu_n+\mathbb{E}((W^*_n)^2)-\sum_{m=1}^{t-1}\frac{w^2_m}{l_n}-\sum_{m=1}^{t-1}\frac{w^3_m}{l_n}\\
	\nonumber&\geq 1+\frac{\mathbb{E}(W^3)}{\mathbb{E}(W)}-O\bigg(\Big(\frac{T_1}{n}\Big)^{\frac{\tau-4}{\tau-3}}\bigg)-o(1)\\
	&=1+\frac{\mathbb{E}(W^3)}{\mathbb{E}(W)}-o(1).
	\end{align} 
	Therefore, going back to (\ref{computationsecmom}), we see that if $|\mathcal{A}^{BP}_{t-1}|\geq 1$ then
	\begin{equation*}
	\mathbb{E}(W^2_t\mid \mathcal{F}^{BP}_{t-1}) \geq W^2_{t-1}-2C_1n^{-\frac{\tau-3}{\tau-1}}W_{t-1}+\frac{\mathbb{E}(W^3)}{\mathbb{E}(W)}-o(1).
	\end{equation*}
	If we also require $W_{t-1}<2h$ then, since $hn^{-\frac{\tau-3}{\tau-1}}=O(n^{-2\frac{\tau-4}{3(\tau-1)}})\ll 1$ (recall that $\tau>4$), we obtain that the expression on the right-hand side of the last inequality is at least 
	\[W^2_{t-1}+\frac{\mathbb{E}(W^3)}{\mathbb{E}(W)}-o(1).\]
	Thanks to (\ref{imptrick}) and (\ref{comb}) we know that that same is true when $|\mathcal{A}^{BP}_{t-1}|=0$, whence we conclude that the process defined by
	\[W^2_{t\wedge \tilde\tau_{h}} - (t\wedge \tilde\tau_{h})\frac{\mathbb{E}(W^3)}{2\mathbb{E}(W)} \]
	is a submartingale. By the Optional Stopping Theorem \ref{OST} applied with the stopping times $\tau_1=0$ and $\tau_2=\Tilde{\tau}_h$ we arrive at
	\[\mathbb{E}(\Tilde{\tau}_h)\leq \frac{\mathbb{E}(W^2_{\Tilde{\tau}_h})}{\mathbb{E}(W^3)/(2\mathbb{E}(W))}.\]
	Since $\mathbb{E}(W^2_{\Tilde{\tau}_h})\leq 4h^2$ for all large enough $n$, we conclude that 
	\[\mathbb{P}(|\mathcal{A}_t^{BP}|<h\text{ }\forall t\in [T_1-1])\leq \mathbb{P}(\tilde\tau_{h}=T_1)+O\Big(\frac{T^2_1}{hn}\Big)=O\Big(\frac{h^2}{T_1}\vee \frac{T^2_1}{hn}\Big).\]
	\textcolor{black}{Plugging the values of $h\asymp n^{1/3}/A^{1/4}$ and $T_1\asymp n^{2/3}/A^{1/4}$ into the last expression yields that $|\mathcal{A}_t^{BP}|$ stays below $h$ for $T_1-1$ steps with probability $O(1/A^{1/4})$, as desired.}\\

	\begin{center}
		\textbf{Proof of Proposition \ref{prop134}}
	\end{center}
	Recall that $3<\tau<4$. Since $|\mathcal{A}^{BP}_1|=|\widetilde{\mathcal{M}}_1|=|\mathcal{M}_1|-I_1$ we can write
	\begin{equation*}
	\mathbb{P}(|\mathcal{A}^{BP}_1|< h)=\mathbb{P}(|\mathcal{M}_1|<h+I_1)\leq \mathbb{P}(|\mathcal{M}_1|<2h)+\mathbb{P}(I_1\geq h).
	\end{equation*}
	By Lemma \ref{implemmaforlowerbound} and using Markov's inequality we have
	\[\mathbb{P}(I_1\geq h)=O\Big(\frac{w_1^2}{h n}\Big). \]
	Therefore, we obtain
	\begin{equation*}
	\mathbb{P}(|\mathcal{A}^{BP}_1|< h)\leq \mathbb{P}(|\mathcal{M}_1|<2h)+ O\Big(\frac{w_1^2}{h n}\Big)
	=\mathbb{P}(\text{Poi}(w_1)<2h)+O\Big(\frac{w_1^2}{h n}\Big).
	\end{equation*}
	Since $h=\lfloor \delta n^{\frac{1}{\tau-1}} \rfloor\leq \delta c^{1/(\tau-1)}_F w_1$, 
	by taking a small enough $\delta<1/(4c_F^{1/(\tau-1)})$ and using Chernoff's inequality we see that $\mathbb{P}(\text{Poi}(w_1)<2h)\leq \mathbb{P}(\text{Poi}(w_1)<w_1/2)$ can be made exponentially small (in $n$). Moreover, 
	\[\frac{w_1^2}{h n}=O\bigg(\frac{n^{\frac{2}{\tau-1}}}{n^{1+\frac{1}{\tau-1}}}\bigg)=O\bigg(\frac{1}{n^{\frac{\tau-2}{\tau-1}}}\bigg) \]
	and therefore we conclude that
	\begin{equation*}
	\mathbb{P}(|\mathcal{A}^{BP}_1|< h)\leq \frac{C}{n^{\frac{\tau-2}{\tau-1}}}
	\end{equation*}
	for some finite constant $C>0$ which depends on $c_F$ and $\tau$.\\

	\begin{center}
		\textbf{Proof of Proposition \ref{prop2alltau}}
	\end{center}
	Recall that here we want to bound from above
	\begin{equation}\label{wwn}
	\mathbb{P}(\tau_j-\tau_{j-1}<T_2\hspace{0.15cm}\forall j,\exists t\in [T_1-1]:|\AAA^{BP}_t|\geq h).
	\end{equation}
	Define $\tau_h\coloneqq \min\{t<T_1:|\AAA^{BP}_t|\geq h\}$ if this set is nonempty, otherwise let $\tau_h\coloneqq T_1$. Note that, on the event $\{\exists t\in [T_1-1]:|\AAA^{BP}_t|\geq h\}$, we have $|\AAA^{BP}_{\tau_h}|\geq h$. Moreover, if $\tau_j-\tau_{j-1}$ is smaller than $T_2$ for every $j$, then there must be a time $t<T_2$ such that $|\AAA^{BP}_{\tau_h+s}|>0$ for all $s\leq t-1$ and $|\AAA^{BP}_{\tau_h+t}|=0$. Consequently, recalling that $|\AAA^{BP}_{i}|=|\AAA^{BP}_{i-1}|+|\widetilde{\mathcal{M}}_i|-1$ if $|\AAA^{BP}_{i-1}|>0$ and $|\widetilde{\mathcal{M}}_i|=|\mathcal{M}_i|-I_i$ for every $i$, we conclude that there must be a time $t<T_2$ such that the process
	\[R'_{\tau_h+s}\coloneqq |\AAA^{BP}_{\tau_h+s}|+\sum_{i=1}^{s}(|\mathcal{M}_{\tau_h+i}|-1)\]
	stays above $\sum_{i=1}^{s}I_{\tau_h+i}$ for all $s\leq t-1$ and $R'_{\tau_h+t}\leq \sum_{i=1}^{t}I_{\tau_h+i}$. It follows from Lemma \ref{implemmaforlowerbound} that, since $\tau_h\leq T_1$ and $T_2\leq T_1$, with probability at least $1-O((T_1T_2)/(hn))$ we have $\sum_{i=1}^{T_2}I_{\tau_h+i}\leq h/2$. All in all, we can conclude that there must be a time $t<T_2$ at which the process 
	\textcolor{black}{\[R_{\tau_h+t}\coloneqq |\AAA^{BP}_{\tau_h+t}|-\frac{h}{2}+\sum_{i=1}^{t}(|\mathcal{M}_{\tau_h+i}|-1)\leq 0.\]}
	Define $\tau_0\coloneqq \min\{t\geq 1: R_{\tau_h+t}\leq 0\}$ if this set is nonempty, otherwise let $\tau_0\coloneqq T_2$. \textcolor{black}{Based on our previous discussion we conclude that} the probability in (\ref{wwn}) is at most
	\begin{equation}\label{kj}
	\mathbb{P}(\tau_0<T_2\mid |\mathcal{A}^{BP}_{\tau_h}|\geq h)+O\Big(\frac{T_1T_2}{hn}\Big).
	\end{equation}
	Write $\mathbb{P}_h(\cdot)$ for the conditional probability given $\{|\mathcal{A}^{BP}_{\tau_h}|\geq h\}$ and denote by $\mathbb{E}_h[\cdot]$ for conditional expectation given that event. Define 
	\[W_s\coloneqq \frac{h}{2}-\Big(\frac{h}{2}\wedge R_{\tau_h+s}\Big).\]
	Note that, if $0<W_{s-1}<h/2$ (which means that $0<R_{\tau_h+s-1}<h/2$) \textcolor{black}{then it is not hard to show that}
	\begin{equation}\label{te}
	W^2_s-W^2_{s-1}\leq (|\mathcal{M}_{\tau_h+s}|-1)^2+2(1-|\mathcal{M}_{\tau_h+s}|)W_{s-1}.
	\end{equation}
	Taking (conditional) expectation on both sides of (\ref{te}) given $R_{\tau_h+s-1},\tau_h$, and since
	\[\mathbb{E}(|\mathcal{M}_{\tau_h+s}|\mid R_{\tau_h+s-1},\tau_h)=\nu_n, \text{ } \mathbb{E}(|\mathcal{M}_{\tau_h+s}|^2\mid R_{\tau_h+s-1},\tau_h)=\nu_n+\mathbb{E}((W^*_n)^2)\]
	and $h|1-\nu_n|\ll 1$ we arrive at
	\[\mathbb{E}(W^2_s-W^2_{s-1}\mid R_{\tau_h+s-1},\tau_h)\leq 1+\mathbb{E}((W^*_n)^2)+o(1)\leq 2+\mathbb{E}(W^3)/\mathbb{E}(W).\]
	Since the same bound holds true when $R_{\tau_h+s-1}\geq h$, we conclude that the process 
	\[W^2_{s\wedge \tau_0}-(2+\mathbb{E}(W^3)/\mathbb{E}(W))(s\wedge \tau_0), \text{ }0\leq s\leq T_2\]
	is a supermartingale. Moreover, under $\mathbb{E}_h$ such a supermartingale starts at $0$ and so we can use Theorem \ref{OST} to conclude that
	\[\mathbb{E}_h(W^2_{T_2\wedge \tau_0})\leq (2+\mathbb{E}(W^3)/\mathbb{E}(W))(T_2\wedge \tau_0)\leq (2+\mathbb{E}(W^3)/\mathbb{E}(W))T_2\eqqcolon c_0(w)T_2.\]
	Whence we arrive at
	\textcolor{black}{\begin{equation*}
		\mathbb{P}_h(\tau_0<T_2)\leq \mathbb{P}_h(W^2_{T_2\wedge \tau_0}\geq h^2/4)\leq \frac{4\mathbb{E}_h(W^2_{T_2\wedge \tau_0})}{h^2}\leq \frac{4c_0(w)T_2}{h^2},
		\end{equation*}}
	which together with (\ref{kj}) yields the desired result.\\

	\begin{center}
		\textbf{Proof of Proposition \ref{prop234}}
	\end{center}
	The proof follows that same step carried out in the proof of Proposition \ref{prop2alltau}. Specifically, by noticing that in this case we have $\tau_h=1$, following precisely the same steps we see that the process
	\[W^2_{s\wedge \tau_0}-(2+\mathbb{E}((W^*_n)^2))(s\wedge \tau_0), \text{ }0\leq s\leq T_2\]
	is a supermartingale and so we obtain
	\[\mathbb{E}_h(W^2_{T_2\wedge \tau_0})\leq (2+\mathbb{E}((W^*_n)^2))(T_2\wedge \tau_0)\leq (2+\mathbb{E}((W^*_n)^2))T_2.\]
	Whence we arrive at
	\textcolor{black}{\begin{equation*}
		\mathbb{P}_h(\tau_0<T_2)\leq \mathbb{P}_h(W^2_{T_2\wedge \tau_0}\geq h^2/4)\leq \frac{4\mathbb{E}_h(W^2_{T_2\wedge \tau_0})}{h^2}=O\Big(\frac{\mathbb{E}((W^*_n)^2)T_2}{h^2}\Big).
		\end{equation*}}
	Plugging the exact values of $h$ and $T_2$ in the ratio above and using the fact that $\mathbb{E}((W^*_n)^2)=O(n^{\frac{4-\tau}{\tau-1}})$ yields the desired result.

	\subsection{Proofs of Lemma \ref{maxw}, Proposition \ref{ordersize} and Lemmas \ref{OurLemma6}, \ref{lempositive} and \ref{LemUmbi}}\label{subsectionauxiliary}
	In this subsection we prove all the auxiliary results that have been used to obtain the bounds stated in Theorems \ref{teo1taugreater4}, \ref{teo2tauin34} and \ref{teoexp}.
	\begin{proof}[\textbf{Proof of Lemma \ref{maxw}}]
		\textcolor{black}{Suppose first that (\ref{COND}) holds.} Then we have
		\begin{align}\label{wi}
		\nonumber w_i=[1-F]^{-1}(i/n)=& \inf\{s: 1-F(s)\leq i/n\}\\\nonumber
		\nonumber=& \inf\Big\{s: s\geq \Big(\frac{n c_F}{i}\Big)^{1/(\tau-1)}\Big\}\\
		=&\Big(\frac{n c_F}{i}\Big)^{1/(\tau-1)}, \forall i\in [n].
		\end{align}
		Now suppose that $1-F(x)\leq c_Fx^{-(\tau-1)}$ (for every $x\geq 0$) for some $c_F>0,\tau>3$. Let $f(x)$ and $g(x)$ be two functions such that for all $x\geq 0$, $f(x)\leq g(x)$. Since for any $s\geq 0$,  $g(s)\leq t$ implies $f(s)\leq t$, then $\{s:g(s)\leq t\}\subseteq \{s: f(s)\leq t\}$, so $\inf\{s:g(s)\leq t\}\geq \inf\{s: f(s)\leq t\}$. Taking $f(x)=1-F(x)$, $g(x)=c_Fx^{-(\tau-1)}$ and $t=i/n$ it follows from (\ref{wi})
		that $w_i\leq (nc_F/i)^{1/(\tau-1)}$.
		Thus in particular $\omega_1=\text{max}_{i\in [n]} \omega_i=O\left(n^{1/(\tau-1)}\right)$.
		
	\end{proof}

	\begin{proof}[\textbf{Proof of Proposition \ref{ordersize}}]
		To prove Proposition \ref{ordersize} we need Lemma A.1 in \cite{Hofstadt}, which we state here for the reader's convenience.
		\begin{lem}\label{lemmaA1}
			Let $W$ have distribution $F$ and let $W_n$ have distribution $F_n$ as in \eqref{fn}. Let $h:[0,\infty)\rightarrow \C$ be a differentiable function with $h(0)=0$ such that $|h^{'}(x)|[1-F(x)]$ is integrable on $[0,\infty)$. Then for every $a>0$
			\[|\E(h(W_n))-\E(h(W))|\leq \int_{a}^{\infty}|h^{'}(x)|[1-F(x)]dx+\frac{1}{n}\int_{0}^{a}|h^{'}(x)|dx.\]
		\end{lem}
		Suppose first that $\tau>3$. Taking $h(x)=x$ in Lemma \ref{lemmaA1} we get
		\begin{align*}
		\nonumber|\E(W_n)-\E(W)|&\leq \int_{a}^{\infty}(1-F(x))dx+\frac{a}{n}\\
		&\leq \frac{c_F}{\tau-2}a^{-(\tau-2)}+\frac{a}{n}.
		\end{align*}
		Taking $a=(c_Fn)^{1/(\tau-1)}$ we obtain
		\begin{align}\label{98Prime} 
		\nonumber|\E(W_n)-\E(W)|\leq c^{1/(\tau-1)_F}\frac{\tau-1}{\tau-2}n^{-\frac{\tau-2}{\tau-1}}.
		\end{align}
		Next, let $h(x)=x^2$ and observe that $|h'(x)|(1-F(x))\leq 2c_Fx^{-(\tau-2)}$ is integrable since $\tau-2>1$. Thus we can apply Lemma \ref{lemmaA1} to obtain
		\begin{align}
		\nonumber|\E(W^2_n)-\E(W^2)|&\leq \frac{2c_F}{\tau-3}a^{-(\tau-3)}+\frac{a^2}{n}.
		\end{align}
		Taking $a=(c_Fn)^{1/(\tau-1)}$ we arrive at
		\begin{equation}\label{due}
		\nonumber|\E(W^2_n)-\E(W^2)|\leq \frac{2c^{2/(\tau-1)}_F}{\tau-3}n^{-\frac{\tau-3}{\tau-1}}(1+O(n^{-1/(\tau-1)})).
		\end{equation}
		Therefore we obtain 
		\begin{enumerate}
			\item [(i)] $\mathbb{E}(W)-c^{1/(\tau-1)}_F n^{-\frac{\tau-2}{\tau-1}}\frac{\tau-1}{\tau-2}\leq  \mathbb{E}(W_n)$;
			\item [(ii)] $\mathbb{E}(W_n)\leq \mathbb{E}(W)+ c^{1/(\tau-1)}_F n^{-\frac{\tau-2}{\tau-1}}\frac{\tau-1}{\tau-2}$;
			\item [(iii)] $\mathbb{E}(W^2) - \frac{2c^{2/(\tau-1)}_F}{\tau-3}n^{-\frac{\tau-3}{\tau-1}}\left(1+O(n^{-1/(\tau-1)})\right)\leq \mathbb{E}(W^2_n)$;
			\item [(iv)] $\mathbb{E}(W^2_n)\leq \mathbb{E}(W^2)+ \frac{2c^{2/(\tau-1)}_F}{\tau-3}n^{-\frac{\tau-3}{\tau-1}}\left(1+O(n^{-1/(\tau-1)})\right)$.
		\end{enumerate}
		Consequently, letting $N\in \mathbb{N}$ be so large that
		\[\mathbb{E}(W)-c^{1/(\tau-1)}_F N^{-\frac{\tau-2}{\tau-1}}\frac{\tau-1}{\tau-2}>0\]
		we see that, for $n\geq N$,
		\[\frac{ \mathbb{E}(W^2_n)}{ \mathbb{E}(W_n)}\leq \frac{\mathbb{E}(W^2)+ \frac{2c^{2/(\tau-1)}_F}{\tau-3}n^{-\frac{\tau-3}{\tau-1}}\Big(1+O(n^{-1/(\tau-1)})\Big)}{\mathbb{E}(W)-c^{1/(\tau-1)}_F n^{-\frac{\tau-2}{\tau-1}}\frac{\tau-1}{\tau-2}}\]
		and
		\[\frac{ \mathbb{E}(W^2_n)}{ \mathbb{E}(W_n)}\geq \frac{\mathbb{E}(W^2) - \frac{2c^{2/(\tau-1)}_F}{\tau-3}n^{-\frac{\tau-3}{\tau-1}}\Big(1+O(n^{-1/(\tau-1)})\Big)}{\mathbb{E}(W)+ c^{1/(\tau-1)}_F n^{-\frac{\tau-2}{\tau-1}}\frac{\tau-1}{\tau-2}}.\]
		Using the fact that $\mathbb{E}(W^2)/\mathbb{E}(W)=1$ we obtain, for $n\geq N$,
		\[\Big| \frac{ \mathbb{E}(W^2_n)}{ \mathbb{E}(W_n)}-1 \Big|\leq \frac{p_n}{\mathbb{E}(W)- c^{1/(\tau-1)}_F n^{-\frac{\tau-2}{\tau-1}}\frac{\tau-1}{\tau-2}}\]
		where we set
		\begin{equation*}
		p_n\coloneqq \frac{2c^{2/(\tau-1)}_F}{\tau-3}n^{-\frac{\tau-3}{\tau-1}}\left(1+O(n^{-1/(\tau-1)})\right)+c^{1/(\tau-1)}_F n^{-\frac{\tau-2}{\tau-1}}\frac{\tau-1}{\tau-2}.
		\end{equation*}
		Therefore for all large enough $n\geq N$ we obtain
		\begin{equation}
		|\nu_n-1|\leq \frac{4\frac{c^{2/(\tau-1)}_F}{\tau-3}n^{-\frac{\tau-3}{\tau-1}}}{\mathbb{E}(W)-c^{1/(\tau-1)}_F\frac{\tau-1}{\tau-2}n^{-\frac{\tau-2}{\tau-1}}},
		\end{equation}
		establishing the first part of the proposition. Next note that, if $\tau>4$ then $|h'(x)|(1-F(x))\leq 3c_Fx^{-(\tau-3)}$ is integrable and hence we can use once again Lemma \ref{lemmaA1} to bound
		\begin{equation}
		|\E(W_n^3)-\E(W^3)|\leq \frac{3c_F}{\tau-4}a^{-(\tau-4)}+\frac{a^2}{n}.
		\end{equation}
		Taking $a=(c_Fn)^{1/(\tau-1)}$ we see that
		\begin{equation}
		|\E(W_n^3)-\E(W^3)|\leq 3\frac{c^{3/(\tau-1)_F}}{\tau-4}n^{-\frac{\tau-4}{\tau-1}}\left(1+O(n^{-1/(\tau-1)})\right).
		\end{equation}
		Consequently we arrive at
		\begin{equation}
		\Big|\frac{\mathbb{E}(W^3_n)}{\mathbb{E}(W_n)}-\frac{\mathbb{E}(W^3)}{\mathbb{E}(W)}\Big|\leq \frac{\mathbb{E}(W)5\frac{c^{3/(\tau-1)}_F}{\tau-4}n^{-\frac{\tau-4}{\tau-1}}}{\mathbb{E}(W)\Big(\mathbb{E}(W)-c^{1/(\tau-1)}_F\frac{\tau-1}{\tau-2}n^{-\frac{\tau-2}{\tau-1}}\Big)},
		\end{equation}
		completing the proof.
	\end{proof}

	\begin{proof}[\textbf{Proof of Lemma \ref{OurLemma6}}]
		Since $S_{\gamma}=S_{\gamma-1}+\Upsilon_{\gamma}-1$ we obtain 
		\begin{align}\label{0}
		&\mathbb{P}(S_{\gamma}-H\geq k|S_{\gamma}\geq H ,\gamma \in \Sigma)\nonumber\\
		&\hspace{0.5cm}=\sum_{h=1}^{H-1}\sum_{m\in \Sigma}^{}\frac{\mathbb{P}(S_m-H\geq k, S_m\geq H, \gamma = m,S_{m-1}=H-h)}{\mathbb{P}(S_{\gamma}\geq H, \gamma \in \Sigma)}.
		\end{align}
		Now setting $\mathbb{N}_H\coloneqq \{n \in \mathbb{N}: n\geq H\}$ we see that
		\begin{equation*}
		\{\gamma = m\}=\Big\{S_j\in (0,H)\cap \mathbb{N}\text{ }\forall j \leq m-1, S_m\in \{0\}\cup \mathbb{N}_H  \Big\}
		\end{equation*}
		and on the event $\{S_m-H\geq k\}=\{S_m\geq H+k\}$ we clearly have $S_m\in \{0\}\cup \mathbb{N}_H $, whence the numerator of the ratio in (\ref{0}) can be written as 
		\begin{align}\label{fromhere}
		\nonumber&\mathbb{P}(S_m-H\geq k, S_m\geq H, \gamma = m,S_{m-1}=H-h)\\
		\nonumber&\hspace{0.5cm}=\mathbb{P}(S_m-H\geq k, S_m\geq H, S_j\in (0,H)\cap \mathbb{N}\text{ }\forall j\leq m-1,S_{m-1}=H-h)\\
		\nonumber&\hspace{0.5cm}=\mathbb{P}(\Upsilon_m\geq h+1+k,S_j\in (0,H)\cap \mathbb{N}\text{ }\forall j\leq m-1,S_{m-1}=H-h)\\
		&\hspace{0.5cm}=\mathbb{P}(\Upsilon_m \geq h+1+k)\mathbb{P}(S_j\in (0,H)\cap \mathbb{N}\text{ }\forall j\leq m-1,S_{m-1}=H-h),
		\end{align}
		where last equality follows from the fact that $\Upsilon_m$ is independent of $(S_j)_{1\leq j\leq m-1}$. Now
		\begin{align*}
		\mathbb{P}(\Upsilon_m \geq h+1+k)&=\mathbb{P}(\Upsilon_1 \geq h+1+k)\\
		&=\mathbb{P}(\Upsilon_1 \geq h+1+k, \Upsilon_1\geq h+1)\\
		&=\mathbb{P}(\Upsilon_1 \geq h+1+k|\Upsilon_1\geq h+1)\mathbb{P}(\Upsilon_1\geq h+1)\\
		&=\mathbb{P}(\Upsilon_1 \geq h+1+k|\Upsilon_1\geq h+1)\mathbb{P}(\Upsilon_m\geq h+1)
		\end{align*}
		and also
		\begin{multline}\label{iyy}
		\mathbb{P}(\Upsilon_m\geq h+1)\mathbb{P}(S_j\in (0,H)\cap \mathbb{N}\text{ }\forall j\leq m-1,S_{m-1}=H-h)\\
		=\mathbb{P}(\Upsilon_m\geq h+1, S_j\in (0,H)\cap \mathbb{N}\text{ }\forall j\leq m-1,S_{m-1}=H-h).
		\end{multline}
		Note that if $S_{m-1}=H-h$ and $\Upsilon_m\geq h+1$ then 
		\[S_m=S_{m-1}+\Upsilon_m-1= H-h+\Upsilon_m\geq  H.\]
		Thus the probability in \eqref{iyy} is at most
		\begin{align*}
		&\mathbb{P}(S_m\geq h, S_j\in (0,H)\cap \mathbb{N}\text{ }\forall j\leq m-1,S_{m-1}=H-h)\\
		&\hspace{0.5cm}=\mathbb{P}(S_m\geq H, S_j\in (0,H)\cap \mathbb{N}\text{ }\forall j\leq m-1,S_m\in \{0\}\cup \mathbb{N}_H,S_{m-1}=H-h)\\
		&\hspace{0.5cm}=\mathbb{P}(S_m\geq H, \gamma=m,S_{m-1}=H-h).
		\end{align*}
		Consequently the probability in \eqref{fromhere} is at most
		\begin{equation}
		\mathbb{P}(\Upsilon_1\geq k+h+1|\Upsilon_1\geq h+1)\mathbb{P}(S_m\geq H, \gamma=m,S_{m-1}=H-h)
		\end{equation}
		and hence the ration in \eqref{0} is bounded from above by
		\begin{multline*}
		\sum_{h=1}^{H-1}\sum_{m\in \Sigma}^{}\frac{\mathbb{P}(\Upsilon_1\geq k+h+1|\Upsilon_1\geq h+1)\mathbb{P}(S_m\geq H, \gamma=m,S_{m-1}=H-h)}{\mathbb{P}(S_{\gamma}\geq H, \gamma \in \Sigma)}\\
		=\sum_{h=1}^{H-1}\mathbb{P}(\Upsilon_1\geq k+h+1|\Upsilon_1\geq h+1)\mathbb{P}(S_{\gamma-1}=H-h|S_{\gamma}\geq H, \gamma \in \Sigma).
		\end{multline*}	
		Next we evaluate the probabilities $\mathbb{P}(\Upsilon_1\geq h+k+1|\Upsilon_1\geq h+1)$ appearing within last sum. Since $\Upsilon_1$ follows a mixed Poisson distribution $Poi(w_{M_1})$, then, conditional on $M_1=j$, $\Upsilon_1$ is distributed as a Poisson random variable with mean $w_j$, $j\in [n]$. Let $\{Y_{w_j}\}_{j\in [n]}$ be a sequence of random variables such that $Y_{w_j}\sim Poi(w_j)$, $j\in [n]$. Then a short calculation reveals that
		\begin{align}\label{3}
		&\mathbb{P}(\Upsilon_1\geq h+k+1|\Upsilon_1\geq h+1)\nonumber\\
		=&\sum_{i=1}^n \mathbb{P}(Y_{w_i}\geq h+k+1|Y_{w_i}\geq h+1)\mathbb{P}(M_1=i|\Upsilon_1\geq h+1).
		\end{align}
		We will show that, for $i\in [n]$,
		\begin{equation}\label{boundforstoch}
		\mathbb{P}(Y_{w_i}\geq h+k+1|Y_{w_i}\geq h+1)\leq \Prob(Y_{w_1}\geq k).
		\end{equation}
		Note that, if \eqref{boundforstoch} were true, then we would obtain
		\begin{multline*}
		\sum_{i=1}^n \mathbb{P}(Y_{w_i}\geq h+k+1|Y_{w_i}\geq h+1)\mathbb{P}(M_1=i|\Upsilon_1\geq h+1)\\
		\leq \sum_{i=1}^{n}\Prob(Y_{w_1}\geq k)\mathbb{P}(M_1=i|\Upsilon_1\geq h+1)=\Prob(Y_{w_1}\geq k)
		\end{multline*}
		and hence
		\begin{align*}
		\mathbb{P}(S_{\gamma}-H\geq k|S_{\gamma}\geq H,\gamma\in \Sigma)&\leq \Prob(Y_{w_1}\geq k)\sum_{h=1}^{H-1}\Prob(S_{\gamma-1}=H-h|S_{\gamma}\geq H,\gamma \in \Sigma)\\
		&=\Prob(Y_{w_1}\geq k),
		\end{align*}
		which is the required result. To establish \eqref{boundforstoch}, observe that
		\begin{equation}
		\mathbb{P}(Y_{w_i}\geq h+k+1|Y_{w_i}\geq h+1)=\frac{\mathbb{P}(Y_{w_i}\geq h+k+1)}{\mathbb{P}(Y_{w_i}\geq h+1)}.
		\end{equation}
		For $N\geq \lceil w_1 \rceil$ ($\geq w_i$) let $B_{N,w_i/N}$ be a random variable with the $\text{Bin}(N,w_i/N)$ distribution. Then for every $k\geq 0$ we see that
		\begin{equation*}
		\mathbb{P}(Y_{w_i}\geq h+k+1)=\lim_{N\rightarrow \infty}\mathbb{P}(B_{N,w_i/N}\geq h+k+1)
		\end{equation*}
		and hence
		\begin{align*}
		\frac{\mathbb{P}(Y_{w_i}\geq h+k+1)}{\mathbb{P}(Y_{w_i}\geq h+1)}&=\lim_{N\rightarrow \infty}\frac{\mathbb{P}(B_{N,w_i/N}\geq h+k+1)}{\mathbb{P}(B_{N,w_i/N}\geq h+1)}\\
		&=\lim_{N\rightarrow \infty}\mathbb{P}(B_{N,w_i/N}-(h+1)\geq k|B_{N,w_i/N}\geq h+1).
		\end{align*}
		Using Lemma 5 in \cite{NP} we obtain
		\begin{equation*}
		\mathbb{P}(B_{N,w_i/N}-(h+1)\geq k|B_{N,w_i/N}\geq h+1)\leq \mathbb{P}(B_{N,w_i/N}\geq k)
		\end{equation*}
		whence
		\begin{align*}
		\lim_{N\rightarrow \infty}\mathbb{P}(B_{N,w_i/N}-(h+1)\geq k|B_{N,w_i/N}\geq h+1)&\leq \lim_{N\rightarrow \infty}\mathbb{P}(B_{N,w_i/N}\geq k) \\
		&=\mathbb{P}(Y_{w_i}\geq k)\leq \mathbb{P}(Y_{w_1}\geq k),
		\end{align*}
		where for the last inequality we have used the fact that $w_i\geq w_{i+1}$ for $1\leq i\leq n-1$.
		
	\end{proof}

	\textcolor{black}{\begin{proof}[\textbf{Proof of Lemma \ref{lempositive}}]
			Note that
			\begin{equation}\label{THENUM}
			1-\nu_n=1-\E(W^*_n)=1-\frac{\E(W_n^2)}{\E(W_n)}
			\end{equation}
			By Lemma \ref{maxw} we have
			\begin{equation}\label{DEN}
			\E(W_n)=c_F^{1/(\tau-1)}n^{\frac{2-\tau}{\tau-1}}\sum_{i\in [n]}^{}i^{-1/(\tau-1)}
			\end{equation}
			and 
			\begin{equation}\label{NUM}
			\mathbb{E}(W^2_n)=c_F^{2/(\tau-1)}n^{\frac{3-\tau}{\tau-1}}\sum_{i\in [n]}^{}i^{-2/(\tau-1)}.
			\end{equation}
			Since
			\begin{equation*}
			\int_{1}^{b+1}x^{-r}dx\leq \sum_{x=1}^{b}x^{-r}\leq 1+\int_{1}^{b}x^{-r}dx
			\end{equation*}
			for every $r>0$ and all $x\geq 1$, it is not difficult to see that
			\begin{align*}
			\sum_{i=1}^{n}i^{-\frac{1}{\tau-1}}\geq \frac{\tau-1}{\tau-2}(n+1)^{\frac{\tau-2}{\tau-1}}-\frac{\tau-1}{\tau-2}
			\end{align*}
			and
			\begin{align*}
			\sum_{i=1}^{n}i^{-\frac{2}{\tau-1}}\leq 1+\frac{\tau-1}{\tau-3}n^{\frac{\tau-3}{\tau-1}}-\frac{\tau-1}{\tau-3}.
			\end{align*}
			Substituting these estimates into (\ref{NUM}) and (\ref{DEN}) and recalling that, under (\ref{COND}), we have $c^{1/(\tau-1)}_F=(\tau-3)(\tau-1)^{-1}$, we easily see that (for all large enough $n$) 
			(\ref{THENUM}) is at least
			\[\frac{n^{-\frac{\tau-3}{\tau-1}}}{\tau-1} >0,\]
			completing the proof.
	\end{proof}}

	\begin{proof}[\textbf{Proof of Lemma \ref{LemUmbi}}]
		Note that for every $i$ we have
		\begin{multline*}
		|\widetilde{\mathcal{M}}_{i}|=\Big|\{l\in [X_{v_i}]:J^{v_i}_l\notin (\AAA^{BP}_{i-1} \cup \{m^{BP}_i\})\cup \mathcal{E}^{BP}_{i-1}\cup \mathcal{L}^{v_i}_{l-1}\}\Big|\\
		\leq \Big|\{l\in [X_{v_i}]:J^{v_i}_l\notin \mathcal{E}^{BP}_{i-1}\}\Big|\eqqcolon F_i.
		\end{multline*}
		Intuitively, the reason why this is a good upper bound for $|\widetilde{\mathcal{M}}_{i}|$ is that the number of active marks never grows too much and hence (at least for $i$ sufficiently large) the main contribution to $|(\AAA^{BP}_{i-1} \cup \{m^{BP}_i\})\cup \mathcal{E}^{BP}_{i-1}\cup \mathcal{L}^{v_i}_{l-1}|$ comes from $|\mathcal{E}^{BP}_{i-1}|=i-1$ (recall that at each step during the exploration of the branching process trees we explore precisely one mark).
		\textcolor{black}{It follows that 
			\begin{equation*}
			\mathbb{E}_{Q_{\beta^*}}\Big(e^{r|\widetilde{\mathcal{M}}_{\beta^*+j}|-1}\Big)\leq \mathbb{E}_{Q_{\beta^*}}\Big(e^{r(F_{\beta^*+j}-1)}\Big).
			\end{equation*}
			Now observe that, for $l\in [X_{v_{\beta^*+j}}]$, we have
			\begin{align*}
			\mathbb{P}_{Q_{\beta^*}}(J^{v_{\beta^*+j}}_l\notin \mathcal{E}^{BP}_{\beta^*+j-1})&=1-\mathbb{P}_{Q_{\beta^*}}(J^{v_{\beta^*+j}}_l\in \mathcal{E}^{BP}_{\beta^*+j-1})\\
			&\leq 1-\mathbb{P}_{Q_{\beta^*}}(J^{v_{\beta^*+j}}_l\in \{n-j+1,\dots,n\}),
			\end{align*}}
		where the last inequality follows from the fact that $|\mathcal{E}^{BP}_{\beta^*+j-1}|=\beta^*+j-1\geq j$ and $w_i\geq w_{i+1}$ for $1\leq i\leq n-1$. Now recalling that $l_n$ is the sum of all the weights, we obtain
		\begin{align*}
		1-\mathbb{P}_{Q_{\beta^*}}(J^{v_{\beta^*+j}}_l\in \{n-j+1,\dots,n\})=\sum_{m=1}^{n}\frac{w_m}{l_n}-\sum_{m=n-j+1}^{n}\frac{w_m}{l_n}=\sum_{m=1}^{n-j}\frac{w_m}{l_n},
		\end{align*}
		whence
		\[\mathbb{P}_{Q_{\beta^*}}(J^{v_{\beta^*+j}}_l\notin \mathcal{E}^{BP}_{\beta^*+j-1})\leq \sum_{m=1}^{n-j}\frac{w_m}{l_n}.\]
		Since the number of children of  vertices different to $V_n$ are i.i.d. random variables with distribution $X\sim \text{Poi}(w_M)$, we obtain
		\begin{align*}
		\mathbb{E}_{Q_{\beta^*}}\Big(e^{rF_{\beta^*+j}}\Big)&\leq \mathbb{E}_{Q_{\beta^*}}\Big(e^{r\text{Bin}(X,\sum_{m=1}^{n-j}(w_m/l_n))}\Big)\\
		&=\mathbb{E}_{Q_{\beta^*}}\bigg(\mathbb{E}\Big(e^{r\text{Bin}(X,\sum_{m=1}^{n-j}(w_m/l_n))}|X\Big)\bigg)\\
		&=\mathbb{E}_{Q_{\beta^*}}\bigg(\Big(1+(e^r-1)\sum_{m=1}^{n-j}\frac{w_m}{l_n}\Big)^X\bigg)\\
		&=\sum_{h\geq 0}^{}\Big(1+(e^r-1)\sum_{m=1}^{n-j}\frac{w_m}{l_n}\Big)^h\mathbb{E}_{Q_{\beta^*}}(\mathbb{P}(X=h|M))\\
		&\leq \sum_{h\geq 0}^{}e^{h(e^r-1)\sum_{m=1}^{n-j}\frac{w_m}{l_n}}\mathbb{E}_{Q_{\beta^*}}\Big(e^{-w_M}\frac{w^h_M}{h!}\Big),
		\end{align*}
		where for the last inequality we have used the standard bound $1+x\leq e^x$, which is valid for all $x$. Since
		\begin{equation}
		\mathbb{E}_{Q_{\beta^*}}\Big(e^{-w_M}\frac{w^h_M}{h!}\Big)=\sum_{x=1}^{n}e^{-w_x}\frac{w^h_x}{h!}\frac{w_x}{l_n}
		\end{equation}
		we obtain
		\begin{align*}
		\sum_{h\geq 0}^{}e^{h(e^r-1)\sum_{m=1}^{n-j}\frac{w_m}{l_n}}\textcolor{black}{\mathbb{E}_{Q_{\beta^*}}\Big(e^{-w_M}\frac{w^h_M}{h!}\Big)}&=\sum_{x=1}^{n}e^{-w_x}\frac{w_x}{l_n}\sum_{h\geq 0}^{}\frac{\Big(e^{(e^r-1)\sum_{m=1}^{n-j}\frac{w_m}{l_n}}w_x\Big)^h}{h!}\\
		&=\sum_{x=1}^{n}e^{-w_x}\frac{w_x}{l_n}\exp\Big\{e^{(e^r-1)\sum_{m=1}^{n-j}\frac{w_m}{l_n}}w_x\Big\}\\
		&=\sum_{x=1}^{n}\frac{w_x}{l_n}\exp\Big\{w_x\left[e^{(e^r-1)\sum_{m=1}^{n-j}\frac{w_m}{l_n}}-1\right]\Big\}.
		\end{align*}
		Next we bound the term $e^{(e^r-1)\sum_{m=1}^{n-j}\frac{w_m}{l_n}}$ appearing within the last expression. Recall that $w_x=\Big(\frac{c_Fn}{x}\Big)^{\frac{1}{\tau-1}}$. Then 
		\begin{equation*}
		\sum_{m=1}^{n-j}\frac{w_m}{l_n}=\frac{(c_Fn)^{\frac{1}{\tau-1}}}{l_n}\sum_{m=1}^{n-j}m^{-\frac{1}{\tau-1}}.
		\end{equation*}
		Now
		\begin{equation}\label{partialsum}
		\sum_{m=1}^{n-j}m^{-\frac{1}{\tau-1}}\leq \frac{\tau-1}{\tau-2}(n-j)^{\frac{\tau-2}{\tau-1}}+1
		\end{equation}
		and
		\begin{equation}\label{allsum}
		l_n\geq c^{1/(\tau-1)}_F\frac{\tau-1}{\tau-2}n-O\left(n^{1/(\tau-1)}\right),
		\end{equation}
		where the constant in the $O-$notation depends on $c_F$ and $\tau$. Using jointly \eqref{partialsum} and \eqref{allsum} we see that
		\begin{align}\label{90Prime}
		\sum_{m=1}^{n-j}\frac{w_m}{l_n}\leq& (c_Fn)^{1/(\tau-1)}\frac{\frac{\tau-1}{\tau-2}(n-j)^{\frac{\tau-2}{\tau-1}}+1}{c^{1/(\tau-1)}_F\frac{\tau-1}{\tau-2}n-O\Big(n^{1/(\tau-1)}\Big)}\nonumber\\
		=&\frac{(c_Fn)^{1/(\tau-1)}\frac{\tau-1}{\tau-2}(n-j)^{\frac{\tau-2}{\tau-1}}}{c^{1/(\tau-1)}_F\frac{\tau-1}{\tau-2}n-O\Big(n^{1/(\tau-1)}\Big)}+O\Big(n^{-\frac{\tau-2}{\tau-1}}\Big).
		\end{align}
		Using the \textcolor{black}{inequalities $\log(1-x)>-x-x^2$ and $e^y\leq 1+y+y^2$, valid for $x\in (0,0.69)$ and $0\leq y\leq 1$,} respectively, a simple computation yields
		\begin{align}\label{90PrimePrime}
		\frac{(c_Fn)^{1/(\tau-1)}\frac{\tau-1}{\tau-2}(n-j)^{\frac{\tau-2}{\tau-1}}}{c^{1/(\tau-1)}_F\frac{\tau-1}{\tau-2}n-O\Big(n^{1/(\tau-1)}\Big)}&\leq \frac{(c_Fn)^{1/(\tau-1)}\frac{\tau-1}{\tau-2}(n-j)^{\frac{\tau-2}{\tau-1}}}{c^{1/(\tau-1)}_F\frac{\tau-1}{\tau-2}n}\Big(1+O\Big(n^{-\frac{\tau-2}{\tau-1}}\Big)\Big)\nonumber\\
		&=\Big(1-\frac{j}{n}\Big)^{\frac{\tau-2}{\tau-1}}\Big(1+O\Big(n^{-\frac{\tau-2}{\tau-1}}\Big)	\Big).
		\end{align}
		Since $\log(1+x)\leq x$ for all $x>-1$ and $e^{-x}\leq 1-x+x^2$ for $x\geq 0$ we obtain
		\begin{equation*}
		\Big(1-\frac{j}{n}\Big)^{\frac{\tau-2}{\tau-1}}=e^{\frac{\tau-2}{\tau-1}\log(1-j/n)}\leq e^{-\frac{\tau-2}{\tau-1}\frac{j}{n}}\leq 1-\frac{\tau-2}{\tau-1}\frac{j}{n}+\Big(\frac{\tau-2}{\tau-1}\Big)^2\Big(\frac{j}{n}\Big)^2.
		\end{equation*}
		Now 
		\begin{equation}\label{90PrimePrimePrime}
		\bigg(1-\frac{\tau-2}{\tau-1}\frac{j}{n}+\Big(\frac{\tau-2}{\tau-1}\Big)^2\Big(\frac{j}{n}\Big)^2\bigg)O\Big(n^{-\frac{\tau-2}{\tau-1}}\Big)=O\Big(n^{-\frac{\tau-2}{\tau-1}}\Big)
		\end{equation}
		and hence by \eqref{90Prime}, \eqref{90PrimePrime} and \eqref{90PrimePrimePrime} we obtain
		\begin{align*}
		\sum_{m=1}^{n-j}\frac{w_m}{l_n}\leq 1-\frac{\tau-2}{\tau-1}\frac{j}{n}+\Big(\frac{\tau-2}{\tau-1}\Big)^2\Big(\frac{j}{n}\Big)^2+O\Big(n^{-\frac{\tau-2}{\tau-1}}\Big).
		\end{align*}
		Therefore, taking $r\in (0,1)$ we arrive at
		\begin{multline*}
		e^{(e^r-1)\sum_{m=1}^{n-j}\frac{w_m}{l_n}}
		\leq \exp\bigg\{(r+r^2)\bigg(1-\frac{\tau-2}{\tau-1}\frac{j}{n}+\Big(\frac{\tau-2}{\tau-1}\Big)^2\left(\frac{j}{n}\right)^2\bigg)+rO\Big(n^{-\frac{\tau-2}{\tau-1}}\Big)\bigg\}.
		\end{multline*}
		Using once more the bound $e^{x}\leq 1+x+x^2$, valid for $x\in [0,1]$, we see that last exponential is at most
		\[1+(r+r^2)\bigg(1-\frac{\tau-2}{\tau-1}\frac{j}{n}+\Big(\frac{\tau-2}{\tau-1}\Big)^2\Big(\frac{j}{n}\Big)^2\bigg)+2r^2+rO\Big(n^{-\frac{\tau-2}{\tau-1}}\Big).\]
		Consequently
		\begin{align*}
		e^{(e^r-1)\sum_{m=1}^{n-j}\frac{w_m}{l_n}}-1\leq& (r+r^2)\bigg(1-\frac{\tau-2}{\tau-1}\frac{j}{n}+\Big(\frac{\tau-2}{\tau-1}\Big)^2\Big(\frac{j}{n}\Big)^2\bigg)+2r^2+rO\Big(n^{-\frac{\tau-2}{\tau-1}}\Big)\\
		=& (r+r^2)a(j,\tau,n)+2r^2+rO\Big(n^{-\frac{\tau-2}{\tau-1}}\Big),
		\end{align*}
		where we set
		\begin{equation*}
		a(j,\tau,n) \coloneqq 1-\frac{\tau-2}{\tau-1}\frac{j}{n}+\Big(\frac{\tau-2}{\tau-1}\Big)^2\Big(\frac{j}{n}\Big)^2.
		\end{equation*}
		Summarizing, so far we managed to show that
		\begin{align*}
		&\mathbb{E}_{Q_{\beta^*}}\Big[e^{rF_{\beta^*+j}}\Big]\le \sum_{x=1}^{n}\frac{w_x}{l_n}\exp\bigg\{w_x\bigg[(r+r^2)a(j,\tau,n)+2r^2+rO\Big(n^{-\frac{\tau-2}{\tau-1}}\Big)\bigg]\bigg\}.
		\end{align*}
		If $r\leq 1/w_1$ then the exponential term within last sum is at most
		\begin{equation*}
		1+w_x\bigg[(r+r^2)a(j,\tau,n)+2r^2+rO\Big(n^{-\frac{\tau-2}{\tau-1}}\Big)\bigg]
		+w^2_x\Big[(r+r^2)a(j,\tau,n)+2r^2+rO\Big(n^{-\frac{\tau-2}{\tau-1}}\Big)\Big]^2
		\end{equation*}
		and
		\begin{equation*}
		\Big[(r+r^2)a(j,\tau,n)+2r^2+rO\Big(n^{-\frac{\tau-2}{\tau-1}}\Big)\Big]^2=r^2a(j,\tau,n)^2+O(r^3)\leq r^2+O(r^3).
		\end{equation*}
		Consequently, since $a(j,\tau,n)\leq 1$ we obtain
		\begin{multline*}
		\exp\Big\{w_x\Big[(r+r^2)a(j,\tau,n)+2r^2+rO\Big(n^{-\frac{\tau-2}{\tau-1}}\Big)\Big]\Big\}\\
		\leq 1+w_x\Big[(r+r^2)a(j,\tau,n)+2r^2+rO\Big(n^{-\frac{\tau-2}{\tau-1}}\Big)\Big]
		+w^2_x\Big[r^2+O(r^3)\Big].
		\end{multline*}
		Therefore when $r\leq 1/w_1$ we arrive at
		\begin{align*}
		&\mathbb{E}_{Q_{\beta^*}}\Big(e^{rF_{\beta^*+j}}\Big)\\
		&\leq \sum_{x=1}^{n}\frac{w_x}{l_n}\Big(1+w_x\Big[(r+r^2)a(j,\tau,n)+2r^2+rO\Big(n^{-\frac{\tau-2}{\tau-1}}\Big)\Big]+w^2_x(r^2+O(r^3))\Big)\\
		&=1+\Big[(r+r^2)a(j,\tau,n)+2r^2+rO\Big(n^{-\frac{\tau-2}{\tau-1}}\Big)\Big]\sum_{x=1}^{n}\frac{w^2_x}{l_n}+(r^2+O(r^3))\sum_{i=1}^{n}\frac{w^3_x}{l_n}\\
		&=1+(r+r^2)a(j,\tau,n)\nu_n+2r^2\nu_n+rO\Big(n^{-\frac{\tau-2}{\tau-1}}\Big)\nu_n+(r^2+O(r^3))\mathbb{E}((W^*_n)^2)\\
		&\leq \exp\Big\{(r+r^2)a(j,\tau,n)\nu_n+2r^2\nu_n+rO\Big(n^{-\frac{\tau-2}{\tau-1}}\Big)\nu_n+(r^2+O(r^3))\mathbb{E}((W^*_n)^2)\Big\},
		\end{align*}
		where last inequality follows from the fact that $1+x\leq e^x$ for all $x\in \mathbb{R}$. Using once more the inequality $a(j,\tau,n)\leq 1$ and since $\nu_n=O(1)$ we see that the last expression is at most
		\begin{equation*}
		\exp\Big\{ra(j,\tau,n)\nu_n +3r^2\nu_n +rO\Big(n^{-\frac{\tau-2}{\tau-1}}\Big)+(r^2+O(r^3))\mathbb{E}((W^*_n)^2)\Big\}.
		\end{equation*}
		Thus, using the definition of $a(j,\tau,n)$, we arrive at
		\begin{multline*}
		\mathbb{E}_{Q_{\beta^*}}\Big[e^{r(F_{\gamma^*+j}-1)}\Big]	\leq \exp\bigg\{(r^2+O(r^3))\mathbb{E}((W^*_n)^2)-r\nu_n\frac{\tau-2}{\tau-1}\frac{j}{n}\bigg\}\cdot\\
		\cdot \exp\bigg\{r(\nu_n-1)+ r\nu_n\Big(\frac{\tau-2}{\tau-1}\Big)^2\Big(\frac{j}{n}\Big)^2+3r^2\nu_n +rO\Big(n^{-\frac{\tau-2}{\tau-1}}\Big)\bigg\}
		\end{multline*}
		and hence we obtain
		\begin{multline*}
		\mathbb{E}_{Q_{\beta^*}}\Big[e^{r\sum_{j=1}^{t}(|\widetilde{\mathcal{M}}_{j}|-1)}\Big]\leq \exp\bigg\{(r^2+O(r^3))t\mathbb{E}((W^*_n)^2)-r\nu_n\frac{\tau-2}{\tau-1}\frac{t^2}{2n}\bigg\}\cdot\\
		\cdot \exp\bigg\{rt(\nu_n-1)+ r\nu_n\Big(\frac{\tau-2}{\tau-1}\Big)^2O\Big(\frac{t^3}{n^2}\Big)+3r^2t\nu_n +rtO\Big(n^{-\frac{\tau-2}{\tau-1}}\Big)\bigg\}.
		\end{multline*}
		Now since $t\leq T\ll n$, $r\leq 1/w_1$ and $w_1\asymp n^{1/\tau-1}$ we see that
		\textcolor{black}{\begin{equation*}
			rtO\Big(n^{-\frac{\tau-2}{\tau-1}}\Big)= O\Big(\frac{T}{w_1}n^{-\frac{\tau-2}{\tau-1}}\Big)=O(T/n)\ll 1.
			\end{equation*}}
		Moreover, 
		\begin{equation*}
		r\nu_n\Big(\frac{\tau-2}{\tau-1}\Big)^2O\Big(\frac{t^3}{n^2}\Big)= O\Big(\frac{T^3}{n^2w_1}\Big)
		\end{equation*}
		and consequently  for $n$ large enough we arrive at
		\begin{align*}
		\mathbb{E}_{Q_{\beta^*}}\Big(e^{r\sum_{j=1}^{t}(|\widetilde{\mathcal{M}}_{j}|-1)}\Big)\leq &2\exp\bigg\{(r^2t\mathbb{E}((W^*_n)^2)\Big(1+\frac{c'}{w_1}\Big)-r\nu_n\frac{\tau-2}{\tau-1}\frac{t^2}{2n}\bigg\}\cdot\\
		&\cdot \exp\bigg\{rt(\nu_n-1)+ \bar{c}\Big(\frac{T^3}{n^2w_1}\Big)+3r^2t\nu_n \bigg\}
		\end{align*}
		for some finite constants $c',\bar{c}>0$, which is the desired result.
	\end{proof}

	\newpage
	\section*{Appendix}
	
	\begin{proof}[Proof of Proposition \ref{newnewprop}]
		Consider the graph $NR_n(\mathbf{w})$ and the cluster exploration process \textbf{Alg.1} (resp. \textbf{Alg.2} and \textbf{Alg.3}) with $V_n$ denoting the vertex from which \textbf{Alg.1} (resp. \textbf{Alg.2} and \textbf{Alg.3}) starts. Consider $(\mathcal{A}_t)_{t\geq 0}$, $(\mathcal{E}_t)_{t\geq 0}$ and $(\mathcal{U}^*_t)_{t\geq 1}$, the sequences of active  and explored vertices and  the sequence of unseen neighbors of $u_t$, respectively, where  $u_t$ is the vertex under investigation at time $t\geq 1$ by  \textbf{Alg.1} (resp. \textbf{Alg.2} and \textbf{Alg.3}). Recall that $\mathcal{A}_0=\{V_n\}$ and $u_1=V_n$.
		
		Consider also the procedure \textbf{Alg.1.BP} (resp. \textbf{Alg.2.BP} and \textbf{Alg.3.BP}) exploring the thinned Poisson branching processes, with $J_0$ denoting the mark of the root of the corresponding tree from which \textbf{Alg.1.BP} (resp. \textbf{Alg.2.BP} and \textbf{Alg.3.BP}) starts.  Consider $(\mathcal{A}_t^{BP})_{t\geq 0}$, $(\mathcal{E}_t^{BP})_{t\geq 0}$ and $(\widetilde{\mathcal{M}}_t)_{t\geq 1}$, the sequences of active and explored marks and the sequence of sets of \textit{distinct} marks assigned to the children of $v_t$, respectively, where $v_t$ is the vertex corresponding to the mark $m_t^{BP}$ in \textbf{Alg.1.BP} (resp. \textbf{Alg.2.BP} and \textbf{Alg.3.BP}). Recall that $\mathcal{A}_0^{BP}=\{J_0\}$ and $m_1^{BP}=J_0$.
		
		We claim that the sequence $\big((\mathcal{A}_t, \mathcal{E}_t)\big)_t$ has the same distribution as the sequence $\big((\mathcal{A}_t^{BP},\mathcal{E}_t^{BP})\big)_t$.
		
		Assuming the claim is true, recalling that $t_0=\tau_0=0$ and (for $j\geq 1$) $t_j=\min\{t>t_{j-1}:\mathcal{A}_t>0\}$ while $\tau_j=\min\{t>\tau_{j-1}:\mathcal{A}_t^{BP}>0\}$, then the two sequences of stopping times $(t_j)_{j\geq 0}$ and $(\tau_j)_{j\geq 0}$ are equal in distribution. Finally, since $|C_j|=t_j-t_{j-1}$, then  in distribution $|C_j|=\tau_j-\tau_{j-1}$. Hence in the rest of the proof we focus on establishing the claim.
		
		We proceed by induction. By definition, $V_n$ and $J_0$ have the same distribution. Since $\mathcal{A}_0=\{V_n\}$, $\mathcal{E}_0=[n]\setminus \mathcal{A}_0$, $\mathcal{A}_0^{BP}=\{J_0\}$ and $\mathcal{E}_0^{BP}=[n]\setminus \mathcal{A}_0^{BP}$, then  $(\mathcal{A}_0, \mathcal{E}_0) $ and  $(\mathcal{A}_0^{BP},\mathcal{E}_0^{BP})$ have the same distribution . Let's now assume that the claim is true until $t-1$ (for some $t \in \mathbb{N}$).
		
		Suppose that $|\mathcal{A}_{t-1}|>0$.  
		In the procedure \textbf{Alg.1} (resp. \textbf{Alg.2} and \textbf{Alg.3}), take a vertex $u_t$ which is the vertex in $\mathcal{A}_{t-1}$ with the smallest label. Given that $u_t=m$, then $u_t$ is connected with the vertex $j\in \mathcal{U}^*_{t}$ with probability $1-\exp(w_{m}w_j/l_n)$. Moreover, the connections to different vertices in $\mathcal{U}^*_{t}$  are independent by assumption. For an integer $L$, let us denote by $[n]^L$ stand for the collection of all $L$-elements subsets of $[n]$. Thus, if $\{j_1,\dots,j_L\}\in [n]^L\setminus \big(\mathcal{A}_{t-1}\cup \mathcal{E}_{t-1}\big)$, then
		\begin{multline}\label{UnNeigh}
		\mathbb{P}\big(\mathcal{U}^*_{t}=\{j_1,\dots,j_L\}|u_t=m, \big((\mathcal{A}_i,\mathcal{E}_i): i\in [1,t-1]\big)\big)\\
		=\mathbb{P}\big(mj_h\in E(\textbf{w})\text{ }\forall h\in [L], mj_h\notin E(\textbf{w})\text{ }\forall h\in [L]|u_t=m, \big((\mathcal{A}_i,\mathcal{E}_i): i\in [1,t-1]\big)\big)\\
		=e^{-w_m \sum_{j\in [n]\setminus \{j_1,\dots,j_L\}}^{}\frac{w_j}{l_n}}\prod_{h\in [L]}^{}\big(1-e^{-w_m w_{j_h}/l_n}\big).
		\end{multline}
		Now assume that $|\mathcal{A}_{t-1}^{BP}|>0$ and define, for $j\in [n]$
		\begin{equation}
		N^{(t)}_{j}\coloneqq \left|\{l\in \mathcal{M}_t:J^{v_t}_l=j\}\right|, 
		\end{equation}
		the number of children of  $v_t$ carrying the mark $j$ in \textbf{Alg.1.BP} (resp. \textbf{Alg.2.BP} and \textbf{Alg.3.BP}), where $v_t$ is the vertex corresponding to the mark $m_t^{BP}$, and $m_t^{BP}$ is the smallest element in $\mathcal{A}_{t-1}^{BP}$. 
		
		Given  $m_t^{BP}=m$, the random variable $|\mathcal{M}_t|=X_{v_t}$ follows a Poisson distribution with rate $w_m$; thus, conditionally on $m_t^{BP}=m$, the random variable $N^{(t)}_{j}$ counts the number of outcomes equal to $j$ in a multinomial experiment with $\text{Poisson}(w_m)$ (independent) trials.  Therefore, setting $k=\sum_{j\in [n]}^{}k_j$ we obtain
		\begin{align*}
		\mathbb{P}(N^{(t)}_{j}=k_j \hspace{0.15cm}\forall j \in [n]|m_t^{BP}=m)&=\mathbb{P}(N^{(t)}_{j}=k_j \hspace{0.15cm}\forall j \in [n], X_{v_t}=k|m_t^{BP}=m)\\
		&=\mathbb{P}(X_{v_t}=k|m_t^{BP}=m)\mathbb{P}(N^{(t)}_{j}=k_j \hspace{0.15cm}\forall j \in [n]|X=k,m_t^{BP}=m)\\
		&=e^{-w_m}\frac{w^k_m}{k!}\binom{k}{k_1,\dots,k_n}\Big(\frac{w_1}{l_n}\Big)^{k_1}\cdots\Big(\frac{w_n}{l_n}\Big)^{k_n}\\
		&=e^{-w_m}\prod_{j\in [n]}^{}\frac{(w_m w_j/l_n)^{k_j}}{k_j!}\\
		&=e^{-w_m\sum_{j\in [n]}^{}\frac{w_j}{l_n}}\prod_{j\in [n]}^{}\frac{(w_m w_j/l_n)^{k_j}}{k_j!}\\
		&=\prod_{j\in [n]}^{}e^{-w_m w_j/l_n}\frac{(w_m w_j/l_n)^{k_j}}{k_j!}.
		\end{align*}
		That is, conditionally on  $m_t^{BP}=m$, $(N^{(t)}_{j}:j\in [n])$ is a sequence of independent random variables such that $N^{(t)}_{j}$ has the $\text{Poisson}(w_mw_j/l_n)$ distribution, for $j\in [n]$. 
		
		Recall also that $|\mathcal{M}_t|\geq |\widetilde{\mathcal{M}}_t|$, where $\widetilde{\mathcal{M}}_t$ is the set of distinct marks of children of $v_t$ (that is, the set of marks which did not appear at previous steps and which does not contain duplicates). 
		
		Then observe that, for $\{j_1,\dots,j_L\}\in [n]^L\setminus \big(\mathcal{A}_{t-1}^{BP}\cup \mathcal{E}_{t-1}^{BP}\big)$,
		\begin{multline}\label{NewMarks}
		\mathbb{P}\big(\widetilde{\mathcal{M}}_t=\{j_1,\dots,j_L\}|m_t^{BP}=m, \big((\mathcal{A}_i^{BP},\mathcal{E}_i^{BP}): i\in [1,t-1]\big)\big)\\ 
		=\mathbb{P}\big(N^{(t)}_{j_h}\geq 1 \hspace{0.15cm}\forall h\in [L], N^{(t)}_{j}=0\hspace{0.15cm}\forall j\in [n]\setminus \{j_1,\dots,j_L\}|m_t^{BP}=m, \big((\mathcal{A}_i^{BP},\mathcal{E}_i^{BP}): i\in [1,t-1]\big)\big)\\
		=e^{-w_m \sum_{j\in [n]\setminus \{j_1,\dots,j_L\}}^{}\frac{w_j}{l_n}}\prod_{h\in [L]}^{}\big(1-e^{-w_m w_{j_h}/l_n}\big).
		\end{multline}
		Consequently we obtain that $(u_t\cup\mathcal{U}^*_{t})$ and  $(m_t^{BP}\cup\widetilde{\mathcal{M}}_t)$, have the same distribution.
		
		When $\mathcal{A}_{t-1}=0$ and $\mathcal{E}_{t-1}\neq n$, $u_t$ is a random vertex chosen from $[n]\setminus \mathcal{E}_{t-1}$ with probability proportional to its weight, that is with probability $w_j/l'_n(t)$, for $j\in [n]\setminus \mathcal{E}_{t-1}$, and where $l'_n(t)\coloneqq \sum_{i\in[n]\setminus\mathcal{E}_{t-1}}w_i$.  Thus, given $u_t=m$ and $\big((\mathcal{A}_i,\mathcal{E}_i): i\in [1,t-1]\big)$, for $\{j_1,\dots,j_L\}\in [n]^L\setminus \big(\{m\}\cup \mathcal{E}_{t-1}\big)$ we obtain  
		\begin{multline}\label{UnNeigh1}
		\mathbb{P}\big(\mathcal{U}^*_{t}=\{j_1,\dots,j_L\}|u_t=m, \big((\mathcal{A}_i,\mathcal{E}_i): i\in [1,t-1]\big)\big)\\
		=\mathbb{P}\big(j_hm \in E(\textbf{w})\hspace{0.15cm}\forall h\in [L], jm \notin E(\textbf{w})\hspace{0.15cm}\forall j\in [n]\setminus\{j_1,\dots,j_L\}|u_t=m, \big((\mathcal{A}_i,\mathcal{E}_i): i\in [1,t-1]\big)\big)\\
		=e^{-w_m \sum_{j\in [n]\setminus \{j_1,\dots,j_L\}}^{}\frac{w_j}{l_n}}\prod_{h\in [L]}^{}\big(1-e^{-w_m w_{j_h}/l_n}\big).
		\end{multline}
		Similarly, when $\mathcal{A}_{t-1}^{BP}=0$ and $\mathcal{E}_{t-1}^{B}\neq n$, $m_t^{BP}$ is a random mark chosen from $[n]\setminus \mathcal{E}_{t-1}^{BP}$ with probability proportional to its weight, that is with probability $w_j/l'_n(t)$, for $j\in [n]\setminus \mathcal{E}_{t-1}^{BP}$,  and where $l'_n(t)\coloneqq \sum_{i\in[n]\setminus\mathcal{E}_{t-1}^{BP}}w_i$. Thus, given $m_t^{BP}=m$ and $\big((\mathcal{A}_i^{BP},\mathcal{E}_i^{BP}): i\in [1,t-1]\big)$, for $\{j_1,\dots,j_L\}\in [n]^L\setminus \big(\{m\}\cup \mathcal{E}_{t-1}^{BP}\big)$  we obtain 
		\begin{multline}\label{NewMarks1}
		\mathbb{P}\big(\widetilde{\mathcal{M}}_t=\{j_1,\dots,j_L\}|m_t^{BP}=m, \big((\mathcal{A}_i^{BP},\mathcal{E}_i^{BP}): i\in [1,t-1]\big)\big)\\ 
		=\mathbb{P}\big(N^{(t)}_{j_h}\geq 1 \hspace{0.15cm}\forall h\in [L], N^{(t)}_{j}=0\hspace{0.15cm}\forall j\in [n]\setminus \{j_1,\dots,j_L\}|m_t^{BP}=m, \big((\mathcal{A}_i^{BP},\mathcal{E}_i^{BP}): i\in [1,t-1]\big)\big)\\
		= e^{-w_m \sum_{j\in [n]\setminus \{j_1,\dots,j_L\}}^{}\frac{w_j}{l_n}}\prod_{h\in [L]}^{}\big(1-e^{-w_m w_{j_h}/l_n}\big).
		\end{multline}
		Therefore we conclude that also in this case $(u_t\cup\mathcal{U}^*_{t})$ and  $(m_t^{BP}\cup\widetilde{\mathcal{M}}_t)$, have the same distribution.
		
		Finally, assuming that $\mathcal{A}_{t-1}>0$, we have $\mathcal{A}_t=(\mathcal{U}^*_{t} \cup \mathcal{A}_{t-1})\setminus\{u_t\}$ and $\mathcal{E}_t=(\mathcal{E}_{t-1} \cup \{u_t\})$; also, when $\mathcal{A}_{t-1}^{BP}>0$ we have $\mathcal{A}_t^{BP}=(\widetilde{\mathcal{M}}_{t} \cup \mathcal{A}_{t-1}^{BP})\setminus\{m_t^{BP}\}$ and $\mathcal{E}_t^{BP}=(\mathcal{E}_{t-1}^{BP} \cup \{m_t^{BP}\})$. Similarly,
		when $\mathcal{A}_{t-1}=0$ we have $\mathcal{A}_t=\mathcal{U}^*_{t}$ and $\mathcal{E}_t=(\mathcal{E}_{t-1} \cup \{u_t\})$,  and when $\mathcal{A}_{t-1}^{BP}=0$ we have $\mathcal{A}_t^{BP}=\widetilde{\mathcal{M}}_{t}$ and $\mathcal{E}_t^{BP}=(\mathcal{E}_{t-1}^{BP} \cup \{m_t^{BP}\})$. Then using the inductive hypothesis and our previous findings we conclude that $(\mathcal{A}_t,\mathcal{E}_t)$  is distributed like $(\mathcal{A}_t^{BP},\mathcal{E}_t^{BP})$, and the sequences $\big((\mathcal{A}_i,\mathcal{E}_i)\big)_{i\leq t}$ and $\big((\mathcal{A}_i^{BP},\mathcal{E}_i^{BP})\big)_{i\leq t}$ are equal in distribution until time $t$.  
	\end{proof}

	\vskip0.5cm
	
	\section*{Acknowledgements} 
	Both authors thank Guillem Perarnau for interesting discussions concerning the martingale method of Nachmias and Peres, as well as an anonymous referee for useful suggestions that helped improving the quality of the paper.

\end{document}